\documentclass[UTF8,10pt]{article}
\usepackage[left=3cm,right=3cm,top=3cm,bottom=2.5cm]{geometry}
\usepackage{amsfonts}
\usepackage{amsmath}
\usepackage{amssymb}
\usepackage{pgfplots}
\usepackage{tikz}
\usetikzlibrary{arrows}
\usepackage{titlesec}
\usepackage{bm}
\usepackage{appendix}
\usepackage{tikz-cd}
\usepackage{mathtools}
\usepackage{amsthm}
\usepackage{enumitem}
\usepackage[colorlinks,linkcolor=black,anchorcolor=black,citecolor=black]{hyperref}
\usepackage[numbers]{natbib}
\usepackage{cases}
\usepackage{fancyhdr}
\usepackage[scaled=0.92]{helvet}	
\usepackage{txfonts}

\theoremstyle{definition}
\newtheorem{thm}{Theorem}[section]

\newtheorem{defn}{Definition}[section]
\newtheorem{lem}[thm]{Lemma}
\newtheorem{prop}[thm]{Proposition}

\newtheorem{cor}[thm]{Corollary}

\newtheorem*{rmk}{Remark}

\newcommand{\R}{\mathbb{R}}

\newcommand{\TT}{\mathcal{T}}

\newcommand{\curl}{\text{curl }}
\newcommand{\dive}{\text{div }}

\newcommand{\q}{\quad}
\newcommand{\p}{\partial}

\newcommand{\DD}{\mathcal{D}}
\newcommand{\lei}{L^{2}(\mathcal{D}_{t})}
\newcommand{\leb}{L^{2}(\p\mathcal{D}_{t})}
\newcommand{\lli}{L^{2}(\Omega)}
\newcommand{\llb}{L^{2}(\p\Omega)}
\newcommand{\lylt}{L^2([0,T];L^2(\Omega))}

\newcommand{\nab}{\nabla}

\newcommand{\lap}{\Delta}

\newcommand{\di}{\text{div}\,}
\newcommand{\detg}{\text{det}\,g}

\newcommand{\cnab}{\overline{\nab}}
\newcommand{\cp}{\overline{\partial}{}}
\newcommand{\dx}{\,dx}
\newcommand{\dy}{\,dy}
\newcommand{\dt}{\,dt}
\newcommand{\vol}{\text{vol}\,}
\newcommand{\volo}{\text{vol}\,\Omega}

\newcommand{\symdot}{\tilde{\cdot}}
\newcommand{\linf}{L^{\infty}(\p\Omega)}
\newcommand{\kk}{\kappa}

\newcommand{\EE}{\mathcal{E}}

\newcommand{\PP}{\mathcal{P}}

\newcommand{\wk}{\frac{\rho'(p)}{\rho}}
\newcommand{\wt}{\rho'(p)}
\newcommand{\swt}{\sqrt{\rho'(p)}}
\newcommand{\hl}{h^{\lambda}}
\newcommand{\hh}{\tilde{h}}
\newcommand{\wl}{w^{\lambda}}
\newcommand{\wh}{\tilde{w}}
\newcommand{\idt}{\int_{\mathcal{D}_t}}
\newcommand{\ipdt}{\int_{\partial\mathcal{D}_t}}
\newcommand{\io}{\int_{\Omega}}

\numberwithin{equation}{section}
\usepackage{xcolor}

\begin{document}
\bibliographystyle{plain}
\title{\textbf{A priori Estimates for the Free-Boundary problem of Compressible Resistive MHD Equations \\ and Incompressible Limit}}
\author{Junyan Zhang \thanks{Johns Hopkins University, 3400 N Charles St, Baltimore, MD 21218, USA. 
Email: \texttt{zhang.junyan@jhu.edu}} }
\maketitle

\begin{abstract}
In this paper, we prove the a priori estimates in Sobolev spaces for the free-boundary compressible inviscid magnetohydrodynamics equations with magnetic diffusion under the Rayleigh-Taylor physical sign condition. Our energy estimates are uniform in the sound speed. As a result, we can prove the convergence of solutions of the free-boundary compressible resistive MHD equations to the solution of the free-boundary incompressible resistive MHD equations, i.e., the incompressible limit. The key observation is that the magnetic diffusion together with elliptic estimates directly controls the Lorentz force, magnetic field and pressure wave simultaneously.
\end{abstract}

\bigskip


\setcounter{tocdepth}{1}
\tableofcontents

\section{Introduction}
In this paper, we consider the 3D resistive magnetohydrodynamics(MHD) equations
\begin{equation}
\begin{cases}
\rho(\p_t u+u\cdot\p u)=B\cdot\p B-\p (p+\frac{1}{2}|B|^2)~~~& \text{in}~\DD; \\
\p_t\rho+\dive (\rho u)=0~~~&\text{in}~\DD; \\
\p_t B+u\cdot\p B-\lambda\Delta B=B\cdot\p u-B\dive u,~~~&\text{in}~\DD; \\
\dive B=0~~~&\text{in}~\DD,
\end{cases}\label{MHD}
\end{equation}
describing the motion of a compressible conducting fluid in an electro-magnetic field with magnetic diffusion, $\lambda>0$ is the magnetic diffusivity constant. $\DD={\cup}_{0\leq t\leq T}\{t\}\times \DD_t$ and $\DD_t\subset \R^3$ is the domain occupied by the conducting fluid  whose boundary $\p\DD_t$ moves with the velocity of the fluid. $\p=(\p_1,\p_2,\p_3)$ is the standard spatial derivative and $\dive X:=\p_k X^k$ is the standard divergence for any vector field $X$. Throughout this paper, $X^k=\delta^{kl}X_l$ for any vector field $X$, i.e., we use Einstein summation notation. The fluid velocity $u=(u_1,u_2,u_3)$, the magnetic field $B=(B_1,B_2,B_3)$, the fluid density $\rho$, the pressure $p$ and the domain $\DD\subseteq[0,T]\times\R^3$ are to be determined. Here we note that the fluid pressure $p=p(\rho)$ is assumed to be a given strictly increasing smooth function of the density $\rho$.

Given a simply-connected bounded domain $\DD_0\subset \R^3$ homeomorphic to the unit ball in $\R^3$ and the initial data $u_0$, $\rho_0$ and $B_0$ satisfying the constraints $\dive B_0=0$, we want to find a set $\DD$, the vector field $u$, the magnetic field $B$, and the density $\rho$ solving \eqref{MHD} satisfying the initial conditions:
\begin{equation}\label{MHDI}
\DD_0=\{x: (0,x)\in \DD\},\q (u,B,\rho)=(u_0, B_0,\rho_0),\q \text{in}\,\,\{0\}\times \DD_0.
\end{equation}

\begin{rmk}
Note that the divergence-free constraint on $B$ is only required for initial data. Such condition automatically holds for any positive time provided that it holds initially. In fact, one can get the heat equation of $\dive B$ by
\[
D_t(\dive B)-\lambda\Delta\dive B=-(\dive B)(\dive u).
\] We will prove $\dive u\in H^3$ and thus in $L^{\infty}$. Then standard energy estimate yields $\dive B=0$ provided it holds initially.
\end{rmk}

We also require the following conditions on the free boundary $\p\DD={\cup}_{0\leq t\leq T}\{t\}\times \p\DD_t$:
\begin{equation}\label{MHDB}
\begin{cases}
(\p_t+u\cdot \p)|_{\partial\DD}\in T(\partial\DD) \\
p=0&\text{on}~\partial\DD, \\
B=\mathbf{0}&\text{on}~\partial\DD,
\end{cases}
\end{equation} 
where $N$ is the exterior unit normal to $\p\DD_t$.

The first condition of \eqref{MHDB} means that the boundary moves with the velocity of the fluid. We will use the notation $D_t=\p_t+u\cdot\p$ throughout the rest of this paper, and $D_t$ is called the material derivative. The second condition in \eqref{MHDB} means that outside the fluid region $\DD_t$ is the vacuum. Since $p=p(\rho)$ and $p|_{\p\DD}=0$, we know the fluid density also has to be a constant $\bar{\rho_0}\geq 0$ on the boundary. We assume $\bar{\rho_0}>0$, corresponding to the case of liquid as opposed to a gas. Hence
\begin{equation}
p(\bar{\rho_0})=0,~~p'(\rho)>0,~~\text{for }\rho\geq \bar{\rho_0},
\end{equation}where we further assume $\bar{\rho_0}=1$ for simplicity.

Before we explain the third boundary condition $B=\mathbf{0}$ on $\p\DD_t$, it is necessary to introduce its original physical model. In fact, the free-boundary problem originates from the plasma-vacuum model: The plasma is confined in a vacuum in which there is another magnetic field $\hat{B}$. It is formulated as follows (see also chapter 4 of \cite{MHDphy} for the detailed formulation): Suppose that the free-interface between the plasma region $\Omega_+(t)$ and the vacuum region $\Omega_{-}(t)$ is $\Gamma(t)$ which moves with the plasma. Then it requires that \eqref{MHD} holds in the plasma region $\Omega_+(t)$ and the following equations hold for the magnetic field $\hat{B}$ in vacuum $\Omega_{-}(t)$:
\begin{equation}\label{outsideB}
\curl\hat{B}=\mathbf{0},~~~\dive\hat{B}=0.
\end{equation}  
On the interface $\Gamma(t)$, it is required that there is no jump for the pressure or the \textbf{normal component of magnetic fields}:
\begin{equation}\label{interface}
B\cdot N=\hat{B}\cdot N,
\end{equation}  where $N$ is the exterior unit normal to $\Gamma(t)$. Note that for ideal MHD (i.e. $\lambda=0$) the normal continuity $B\cdot N=\hat{B}\cdot N$ on $\Gamma(t)$ should not be an imposed boundary condition, otherwise the ideal MHD system is over-determined as a hyperbolic system. Instead, this is a direct result of propagation of the initial boundary condition $B_0\cdot N=\hat{B}_0\cdot N$. See also Hao-Luo \cite{hao2014priori} for details.

Now we are able to explain the third boundary condition $B=\mathbf{0}$ on $\p\DD_t$ (and also in the vacuum). In the ideal case $(\lambda=0)$, this condition can also be considered as the propagation from initial data, otherwise the ideal MHD (hyperbolic) system is over-determined if we set $B=\mathbf{0}$ on $\p\DD_t$ to be an imposed constraint. However, for resistive MHD $(\lambda>0)$, this condition no longer can be propagated from the initial data because the magentic diffusivity makes the plasma no longer a perfect conductor. Instead, it should be considered as an imposed constraint, which makes sense for a parabolic equation as opposed to the ideal case (hyperbolic system), and thus adding such a constraint will not make the system over-determined. Besides, this condition also yields that the physical energy is conserved when $\lambda=0$ and thus the energy is non-increasing for resistive MHD (see Section \ref{1.3} for detailed proof). 

Hence, the boundary conditions \eqref{MHDB} is the case that the outside magnetic field $\hat{B}$ vanishes in vacuum region in the classical plasma-vacuum model plus the imposed condition $B=\mathbf{0}$ on the boundary. In other words, the model we discuss in this paper is an isolated plasma liquid confined in a vacuum region.

\subsection{Free-boundary compressible resistive MHD equations}\label{1.1}

The free-boundary resistive compressible MHD system considered in this paper is
\begin{equation}
\begin{cases}
\rho D_t u=B\cdot\p B-\p (p+\frac{1}{2}|B|^2)~~~& \text{in}~\DD; \\
D_t\rho+\rho\dive u=0~~~&\text{in}~\DD; \\
D_t B-\lambda\Delta B=B\cdot\p u-B\dive u,~~~&\text{in}~\DD; \\
\dive B=0~~~&\text{in}~\DD,
\end{cases}\label{CMHD}
\end{equation}together with the initial conditions \eqref{MHDI} and the boundary conditions \eqref{MHDB}. As for the pressure $p$, we impose the following natural conditions on $\rho'(p)$ for some fixed constant $c_0:$
\begin{equation}\label{weight}
|\rho^{(m)}(p)|\leq c_0,\text{ and }~c_0^{-1}|\wt|^m\leq|\rho^{(m)}(p)|\leq c_0|\wt|^m,~~for~1\leq m\leq 6.
\end{equation}

To make the initial-boundary value problem \eqref{CMHD}, \eqref{MHDI} and \eqref{MHDB} solvable, the initial data has to satisfy certain compatibility conditions on the boundary. In fact, the continuity equation implies that $\dive v|_{\p\DD}=0$ and thus we have to require $p_0|_{\p\DD_0}=0$ and $\dive v_0|_{\p\DD_0}=0.$ Also the boundary condition $B=\mathbf{0}$ requires that $B_0|_{\p \DD_0}=\mathbf{0}$. Furthermore, we define the $k$-th($k\geq 0$) order compatibility condition as follows:
\begin{equation}\label{cck}
D_t^j p|_{\p\DD_0}=0,~~D_t^j B|_{\p\DD_0}=\mathbf{0}~~\text{at time }t=0~~\forall 0\leq j\leq k.
\end{equation}

Let $N$ be the exterior unit normal vector to $\p\DD_t$. We will prove the a priori bounds for \eqref{CMHD}, \eqref{MHDI} and \eqref{MHDB} in Sobolev spaces under the Rayleigh-Taylor physical sign condition
\begin{equation}\label{sign}
-\nabla_N P\geq\epsilon_0>0~~on~\p\DD_t,
\end{equation} where $\nabla_N:=N^i\p_i$,~~$\epsilon_0>0$ is a constant, and $P:=p+\frac{1}{2}|B|^2$ is the total pressure. This physical sign condition says that the total pressure is higher in the interior than that on the boundary. When $B=0$, i.e., in the case of the free-boundary compressible Euler's equations, the system will be illposed without this physical sign condition (See Ebin \cite{ebin1987equations} for counterexamples). For the free-boundary MHD equations, \eqref{sign} plays the same role as the Rayleigh-Taylor sign condition for the free-boundary Euler's equations, which was pointed out in Hao-Luo \cite{hao2014priori}. Moreover, Hao-Luo \cite{hao2018ill} proved that the free-boundary problem of 2D incompressible MHD equations is illposed when \eqref{sign} fails.

\subsection{History and background}\label{1.2}

The study of the motion of fluid has a long history. In particular, the free-boundary problem of inviscid fluid has blossomed over the past decades. Most of the results are focusing on the incompressible cases. The first breakthrough is the wellposeness of incompressible irrotational water wave problem solved in Wu's work \cite{wu1997LWPww,wu1999LWPww,wu2009GWPww,wu2011GWPww}. For the general incompressible problem with nonzero vorticity, Christodoulou-Lindblad \cite{christodoulou2000motion} first obtained the energy bound under the Rayleigh-Taylor sign condition from a geometric perspective. Then Lindblad \cite{lindblad2005well} proved the local wellposedness(LWP) with Nash-Moser iteration and Coutand-Shkoller \cite{coutand2007LWP} proved the local wellposedness by tangential smoothing which avoided using Nash-Moser iteration. See also Ambrose-Masmoudi \cite{masmoudi05}, Shatah-Zeng \cite{shatah2008geometry,shatah2008priori,shatah2011local}, Zhang-Zhang \cite{zhangzhang08Euler} and Alazard-Burq-Zuily \cite{alazard2014}. 

For the free-boundary compressible Euler equations in the case of a liquid, Lindblad \cite{lindblad2005cwell} proved the LWP in the case of a liquid by using Nash-Moser iteration. Later on, Lindblad-Luo \cite{lindblad2018priori} generalized the method in \cite{christodoulou2000motion} to compressible Euler in the case of a liquid and Ginsberg-Lindblad-Luo \cite{GLL2019LWP} proved the local wellposedness for the motion of compressible self-gravitating liquid. As for the incompressible limit, Lindblad-Luo \cite{lindblad2018priori} proved the incompressible limit in Sobolev norms for the free-boundary problem and the nonzero surface tension case was done by Disconzi-Luo \cite{luo2019limit}. In the case of a gas, we refer to \cite{coutand2010priori, coutand2012LWP, jang2014gas,luozeng2014} and references therein.

However, the theory of the free-boundary MHD equations are much less developed, and nearly all of the available results are focusing on the incompressible case. Actually, MHD equations are quite different from Euler's equations. The strong coupling between the velocity and the magnetic fields in MHD equations often produce extra difficulty. One key difference is the irrotationality assumption for Euler equations no longer hold for MHD. Hao-Luo \cite{hao2014priori} generalized the method developed by Christodoulou-Lindblad \cite{christodoulou2000motion} to incompressible ideal MHD, to get the a priori bounds under the physical sign condition \eqref{sign} and then Hao \cite{hao2017motion} generalized it to the plasma-vacuum model with nonvanishing magnetic field in vacuum. For the wellposedness result, Sun-Wang-Zhang \cite{sun2015well,sun2017well} proved the local wellposedness for the current-vortex sheet and plasma-vacuum model for incompressible MHD respectively under the non-colinearity condition $|B\times\hat{B}|\geq c_0>0$\footnote{The non-collinearity condition gives extra 1/2-order enhanced regularity of the free surface than Rayleigh-Taylor sign condition}. Lee \cite{leeMHD1,leeMHD2} proved the LWP of the 3D free-boundary viscous-resistive MHD equations with infinite and finite depth respectively. See also Padula-Solonnikov \cite{Solonnikov}. In Lee \cite{leeMHD2}, a local unique solution was obtained for the free-boundary ideal incompressible MHD equations by passing to vanishing viscosity-resistivity limit. By using tangential smoothing, Gu-Wang \cite{gu2016construction} proved the LWP of the incompressible MHD equations under the physical sign condition \eqref{sign}. Hao-Luo \cite{haoluo2019} proved the LWP of linearized incompressible MHD equations under the physical sign condition by generalizing Lindblad \cite{lindblad2002}. The author joint with C. Luo \cite{luozhang2019MHD2.5} proved a low regularity estimate. In the case of nonzero surface tension, the author joint with C. Luo \cite{luozhang2019MHDST} first proved the a priori estimates for the incompressible ideal MHD, which is the first step to establish the local existence. Besides, Chen-Ding \cite{chendingMHDlimit} obtained the inviscid limit for the free-boundary ideal incompressible MHD with or without surface tension. Wang-Xin \cite{wangxinMHD2} proved the global well-posedness of incompressible inviscid-resistive MHD. Guo-Ni-Zeng \cite{GuoMHDSTviscous} proved the decay rate of the solutions to viscous-resistive incompressible MHD.

The structure of free-boundary compressible MHD equations is much more delicate than both incompressible MHD equations and compressible Euler's equations due to the extra coupling of the magnetic fields and sound wave. Compared with free-boundary incompressible MHD equations, the top order derivative of the pressure $p$ and $\curl B$ loses control in the free-boundary compressible MHD equations. This does not appear in the incompressible case thanks to $\dive u=0$. On the other hand, compared with compressible Euler's equations, the presence of the magnetic field $B$ in the pressure term $\nabla(p+\frac{1}{2}|B|^2)$ destroys the control of the wave equation of $p$ which is obtained by taking divergence of the first equation in \eqref{CMHD}. This crucial difficulty does not appear in the study of the free-boundary compressible Euler's equations, of which the corresponding wave equation only contains lower order terms.

We first review the results in fixed-domain problems in compressible ideal MHD which is a quasilinear symmetric hyperbolic system with characteristic boundary conditions. Due the the failure of div-curl control mentioned above, even the linearized equation has a loss of normal derivative. Indeed, Ohon-Shirota \cite{MHDexample} constructed an explicit counterexample to prove the ill-posedness in $H^l(l\geq 2)$ for the linearized compressible MHD system. Instead, one may have to consider using anisotropic Sobolev spaces $H_*^m$ which was first introduced by Chen Shuxing \cite{CSX} to solve the hyperbolic system with characteristic boundary conditions. Yanagisawa-Matsumura \cite{MHDfirst} proved the LWP for the fixed domain problem and Secchi \cite{secchi1996,secchi1995} proved a refined result of no regularity loss in anisotropic Sobolev space $H_*^m (m\geq 16)$. As for the incompressible limit, Jiang-Ju-Li \cite{jiangMHDlimit1, jiangMHDlimit2} got the results for the weak solution in the whole space $\R^3$, but no higher order energy control. 

As for the free-boundary problem, Chen-Wang \cite{chengqMHD} and Trakhinin \cite{trakhininMHD2009} proved the existence of the current-vortex sheet for 3D compressible MHD. The only LWP results of the free-boundary problem of the plasma-vacuum model for compressible ideal MHD are Secchi-Trakhinin \cite{secchi2013well} and Trakhinin \cite{trakhininMHD2016} under the non-colinearity condition. To the best of our knowledge, there is NO available result on the free-boundary problem of compressible MHD equations under the physical sign condition \eqref{sign} before the presence of the first version\footnote{The first version of the presenting manuscript was announced on November 10, 2019} of this manuscript. Very recently, Trakhinin-Wang \cite{trakhininMHD2020} proved the LWP of compressible ideal MHD under Rayleigh-Taylor sign condition by using Nash-Moser. The author \cite{ZhangCRMHD2} proved the LWP of compressible resistive MHD under Rayleigh-Taylor sign condition as a continuation of the presenting manuscript.

In this paper, we obtain the a priori estimates and incompressible limit for the free-boundary problem of compressible MHD equations with magnetic diffusion from a geometric point of view introduce by Christodoulou-Lindblad \cite{christodoulou2000motion}. Our energy bound is also uniform in the sound speed $c:=\sqrt{p'(\rho)}$ and thus implies the incompressible limit. We will discuss the details in Section \ref{1.4} and Section \ref{1.5}.

\subsection{Energy conservation/dissipation and higher order energy}\label{1.3}

\subsubsection*{Energy conservation/dissipation}

First we would like to explain the energy conservation for compressible ideal MHD and the energy dissipation for the compressible resistive MHD, mentioned in the introduction.

In fact, for the ideal compressible MHD, if we set $Q(\rho)=\int_1^{\rho} p(R)/R^2 dR$, then we use \eqref{CMHD} to get
\begin{equation}\label{conserve1}
\begin{aligned}
&~~~~\frac{d}{dt}\left(\frac{1}{2}\idt\rho |u|^2\dx+\frac{1}{2}\idt|B|^2 \dx +\idt \rho Q(\rho)\dx\right)  \\
&=\idt \rho u\cdot D_tu\dx+\idt B\cdot D_t B\dx+\idt \rho D_t Q(\rho)\dx+\frac{1}{2}\idt \rho D_t(1/\rho)|B|^2\dx\\
&=\idt u\cdot (B\cdot\p B)\dx-\idt u\cdot\p P \dx+\idt B\cdot(B\cdot\p u)\dx-\idt |B|^2\dive u\dx \\
&~~~~+\idt p(\rho) \frac{D_t \rho}{\rho} \dx -\frac{1}{2}\idt \frac{D_t \rho}{\rho}|B|^2\dx.
\end{aligned}
\end{equation}

Integrating by part in the first term in the last equality, this term will cancel with $\idt B\cdot(B\cdot\p u)\dx$ because the boundary term and the other interior term vanishes due to $B=\mathbf{0}$ and $\dive B=0$ respectively. Also we integrate by parts in the second term and then use the continuity equation to get
\begin{equation}\label{conserve2}
\begin{aligned}
-\idt u\cdot\p P\dx&=\idt P\dive u\dx-\underbrace{\ipdt (u\cdot N)P dS}_{=0}=-\idt p\frac{D_t \rho}{\rho} \dx +\frac{1}{2}\idt |B|^2\dive u\dx\\
&=-\idt p\frac{D_t \rho}{\rho} \dx +\idt |B|^2\dive u\dx -\frac{1}{2}\idt |B|^2\dive u\dx \\
&=-\idt p\frac{D_t \rho}{\rho} \dx +\idt |B|^2\dive u\dx+\frac{1}{2}\idt \frac{D_t \rho}{\rho}|B|^2\dx.
\end{aligned}
\end{equation}

Summing up \eqref{conserve1} and \eqref{conserve2}, one can get the energy conservation for the free-boundary ideal compressible MHD:
\begin{equation}
\frac{d}{dt}\left(\frac{1}{2}\idt\rho |u|^2\dx+\frac{1}{2}\idt|B|^2 \dx +\idt \rho Q(\rho)\dx\right)  =0.
\end{equation} Also one can see this energy conservation coincides with the analogue for the free-boundary compressible Euler's equations in Lindblad-Luo \cite{lindblad2018priori}.

For the resistive compressible MHD as stated in \eqref{CMHD}, there will be one extra dissipation term, and one can integrate by part to get the energy dissipation.
\begin{equation}
\begin{aligned}
&~~~~\frac{d}{dt}\left(\frac{1}{2}\idt\rho |u|^2\dx+\frac{1}{2}\idt|B|^2 \dx +\idt \rho Q(\rho)\dx\right) \\
&=0+\lambda\idt B\cdot\Delta B \dx
=-\lambda\idt |\p B|^2 \dx<0.
\end{aligned}
\end{equation}
\bigskip
\subsubsection*{Higher order energy}

Now we introduce ``$Q$-tensor" to define the higher order energies. Let $Q$ be a positive definite quadratic form $Q$ on $(0,r)$-tensors, which is the inner product of the tangential components when restricted on the boundary, i.e.,
\begin{equation}\label{Qdef}
Q(\alpha,\beta)=\langle \Pi\alpha, \Pi\beta\rangle \text{   on }\p\DD_t,
\end{equation}where the projection of a $(0,r)$-tensor to the boundary is defined by
\begin{equation}\label{pidef}
(\Pi\alpha)_{i_1\cdots i_r}=\gamma_{i_1}^{j_1}\cdots\gamma_{i_r}^{j_r}\alpha_{j_1\cdots j_r},~~\text{ where } \gamma_i^j=\delta_i^j-N_iN^j,
\end{equation}and $N$ is the unit outer normal to $\p\DD_t$. To be more specific, we define
\begin{equation}\label{Qin}
Q(\alpha,\beta)=q^{i_1j_1}\cdots q^{i_rj_r}\alpha_{i_1\cdots i_r}\beta_{j_1\cdots j_r},
\end{equation}where 
\[
q^{ij}:=\delta^{ij}-\eta(d)^2 N^i N^j,~~d(x)=\text{dist}(x,\p\DD_t),~~N^i=-\delta^{ij}\p_j d.
\]Here $\eta$ is a smooth cut-off function satisfying $0\leq \eta(d)\leq 1$, and $\eta (d)=1$ when $d\leq d_0/4$; $\eta (d)=0$ when $d>d_0/2$, where $d_0$ is a fixed numer smaller than the injective radius $\iota_0$ of the normal exponential map, defined to be the largest number $\iota_0$ such that thet map:
\begin{equation}\label{inj rad}
\p\DD_t\times (-l_0,l_0)\to \{x:\text{dist}(x,\p\DD_t)<\iota_0\},
\end{equation}given by 
\[
(\bar{x},L)\mapsto x=\bar{x}+L N(\bar{x})
\] is an injection.

We propose the higher order energies to be 
\begin{equation}\label{Er}
E_r:=\sum_{s+k=r} E_{s,k}+K_r+W_{r+1}^2+H_{r+1}^2,~~\text{ and } E_r^*:=\sum_{r'\leq r}E_{r'}.
\end{equation} Here for $s\geq 1$
\begin{equation}\label{Esk}
\begin{aligned}
E_{s,k}(t)&=\frac{1}{2}\idt \rho Q(\p^s D_t^k u, \p^s D_t^k u)\dx +\frac{1}{2}\idt Q(\p^s D_t^k B, \p^s D_t^k B)\dx \\
&~~~~+\frac{1}{2}\idt \wk Q(\p^s D_t^k p, \p^s D_t^k p)\dx \\
&~~~~+\frac{1}{2}\ipdt Q(\p^s D_t^k P, \p^s D_t^k P)\nu~dS,
\end{aligned}
\end{equation}with $\nu:=(-\nabla_N P)^{-1}$ and 
\begin{equation}\label{E0r}
E_{0,r}(t)=\frac{1}{2}\idt \rho \wt |D_t^r u|^2\dx +\frac{1}{2}\idt |D_t^r B|^2\dx +\frac{1}{2}\idt \wk |D_t^r p|^2\dx, \\
\end{equation}
and
\begin{equation}\label{Kr}
K_r:=\idt\rho|\p^{r-1}\curl u|^2+|\p^{r-1}\curl B|^2\dx,
\end{equation}
\begin{equation}\label{Wr}
W_r:=\frac{1}{2}\left(\|\wt D_t^rp\|_{\lei}+\|\swt\nabla D_t^{r-1}p\|_{\lei}\right),
\end{equation}and
\begin{equation}\label{Hr}
H_r^2(t):=\int_0^t\left(\int_{\DD_{\tau}} |D_t^r B(\tau,x)|^2\dx\right)d\tau+\frac{\lambda}{2}\left\|\p D_t^{r-1} B\right\|_{\lei}^2.
\end{equation}

Here $H_r^2$ is the $r$-th order energy of the heat equation of $B$
\begin{equation}\label{heat00}
D_t B-\lambda\Delta B=B\cdot\p u-B\dive u,
\end{equation} and $W_r$ is the $r$-th order energy of the wave equation of $p$
\begin{equation}\label{wave00}
\rho'(p)D_t^2 p-\Delta p=B^k\Delta B_k+w,
\end{equation}where 
\begin{equation}\label{w00}
w=\left(\wk-\rho''(p)\right)(D_tp)^2+\wk\p_i p\left((B\cdot\p B_i)-\p_i P\right)+\rho\p_i u^k\p_k u^i-\p^i B_k\p_k B^i +|\p B|^2.
\end{equation} This wave equation is derived by taking divergence in the first equation of MHD system \eqref{CMHD}. 

\begin{rmk}
We note that the weight function in \eqref{E0r} and \eqref{Wr} is necessary for passing to the incompressible limit, otherwise there will be no control of $D_t^5 p$ uniform in the sound speed $c:=\sqrt{p'(\rho)}.$ When $B=\mathbf{0}$, our energy is exactly the energy functional for the free-boundary compressible Euler's equation in Lindblad-Luo \cite{lindblad2018priori}.
\end{rmk}

Although $E_r$ only contains the tangential components, it actually allows us to control all the components by the Hodge type decomposition
\[
|\p X|\lesssim|\bar{\p} X|+|\dive X|+|\curl X|.
\] The curl part can be controlled by $K_r$, while the divergence of $u$ can be controlled via the wave equation \eqref{wave00} of $p$ through the continuity equation $D_t\rho=\rho \dive u$ and $p=p(\rho)$. The energy of heat equation helps us to close the control of the wave equation, because the right hand side of \eqref{w00} contains a higher order term of $B$ which is out of control without the magentic resistivity term. The details will be discussed in Section \ref{1.5}. 

The boundary term in \eqref{Esk} and the choice of $\nu$ are constructed to exactly cancel a boundary term coming from integration by part in the interior. Besides, the tangential projection in the bundary term is necessary to make it be a lower order term. Indeed, since $P=p+\frac{1}{2}|B|^2=0$ on the boundary and so is $\Pi\p P=\bar{\p}P$, one has 
\[
\Pi\p^r P=O(\p^{r-1}P).
\]

The physical sign condition \eqref{sign} implies $|\nabla_N P|\geq \epsilon_0$ which allows us to control the regularity of the free boundary, i.e., the second fundamental form $\theta$:
\[
|\bar{\p}^{r-2}\theta|_{\leb}^2\lesssim\epsilon_0^{-1}E_r^*+\sum_{r'\leq r-1}|\p^{r'} P|_{\leb}^2
\]from 
\[
\Pi\p^r P=(\bar{\p}^{r-2}\theta)\nabla_N P+O(\p^{r-1}P)+O(\bar{\p}^{r-3}\theta).
\]

We will use the following notations throughout the rest of this paper:
\begin{itemize}
\item $\|f\|_{s,k} = \|\p^sD_t^k f\|_{\lei}$,
\item $|f|_{s,k}= |\p^sD_t^k f|_{\leb}$.
\end{itemize}

One can reduce the estimates of $Q$-tensor and $\curl$terms to the control of $\|\cdot\|_{s,k}$,$\|\cdot\|_{s,k+1}$ and $|\cdot|_{s,k}$ norms of $u, B, p$ with $s+k\leq r$, which can be further reduced to the control of wave and heat equations by elliptic estimates Proposition \ref{elliptic estimate}. Finally, we close the energy bound by controlling wave and heat equation. More detailed strategy will be discussed in Section \ref{1.5}.

\subsection{Main results}\label{1.4}

\begin{itemize}
\item \textbf{A priori estimates}

The first result in this paper is the a priori bound of the free-boundary compressible resistive MHD system \eqref{CMHD}.

\begin{thm}\label{priori00}
Suppose $0\leq r\leq 4$. Let $(u,B,p)$ be a solution\footnote{The local well-posedness of this problem is established very recently by the author \cite{ZhangCRMHD2}.} to the free-boundary MHD system \eqref{CMHD} together with the initial-boundary conditions $(u_0,B_0,p_0)\in H^4(\DD_0)\times H^5(\DD_0)\times H^4(\DD_0)$, \eqref{MHDI}, \eqref{MHDB} and compatibility conditions \eqref{cck} up to 6-th order\footnote{The reason for requiring 6-th order is that $D_t^6 p$ appears in the higher order wave equations.}. $E_r$ be defined as in \eqref{Er}. Then the following energy bound holds for $T>0:$
\begin{equation}\label{energy00}
E_r^*(T)-E_r^*(0)\lesssim_{K,M,c_0,\vol\DD_t,1/\epsilon_0,1/\lambda,E_{r-1}^*}\int_0^T \PP(E_r^*(t))\dt
\end{equation}for some polynomial $\PP$ with positive co-efficients under the a priori assumptions
\begin{equation}\label{asspriori}
\begin{aligned}
|\theta|+\frac{1}{\iota_0}&\leq K~~on~~\p \DD_t,\\
-\nabla_N P\geq \epsilon_0&>0~~~on~~\p\DD_t,\\
1\leq|\rho|&\leq M~~in~~\DD_t,\\
\sum_{s+k\leq 2}|\p^s D_t^k p|+|\p^s D_t^k B|+|\p^s D_t^k u|&\leq M~~in~~\DD_t.\\
\end{aligned}
\end{equation}
\begin{flushright}
$\square$
\end{flushright}
\end{thm}

\begin{rmk}
In the a priori assumptions \eqref{asspriori}, the first bound gives us the control of the geometry of the free boundary $\p\DD_t$: The bound for $\theta$ actually gives the bound for the curvature of $\p\DD_t$; the lower bound for the injective radius $\iota_0$ of the exponential map characterizes how far away the surface is from self-intersection. All these a priori assumed quantites are controlled in Lemma \ref{justify0}.
\end{rmk}

\begin{rmk}
In \eqref{energy00}, one can apply the nonlinear Gronwall-type inequality introduced in Tao \cite{tao2006nonlinear} to conclude that, there exists a positive time $T(c_0,K,\EE(0),E_4^*(0),\volo)>0$, then any solution of \eqref{CMHD} in $t\in [0,T]$ satisfies 
\[
\sup_{0\leq t\leq T}E_r^*(t)\lesssim_{1/\lambda}2E_r^*(0).
\] See also Proposition \ref{complete}. Our a priori bound depends on $1/\lambda$. Hence, we cannot get the vanishing-resistivity limit by letting $\lambda\to 0$. The necessity of magnetic diffusion is discussed in Section \ref{1.5}. Therefore we can assume the magnetic diffusion constant $\lambda=1$ without loss of generality to discuss the incompressible limit.
\end{rmk}

\item \textbf{Incompressible limit}

From Theorem \ref{priori00}, one can use Gronwall-type argument to see our energy $E_r(t)$ is bounded by the initial data as long as the a priori quantities are bounded in $L^{\infty}$ norm. In fact, this energy bound remains valid uniformly as the sound speed $c^2:=p'(\rho)$ goes to infinity. We define $\kk:=p'(\rho)|_{\rho=1}$ to parametrize the sound speed. Under this setting, we denote the fluid velocity, density, the magnetic field and the pressure by $u_{\kk},\rho_{\kk},B_{\kk}$ and $p_{\kk}$ respectively in \eqref{CMHD}. We also assume the following holds for a fixed constant $c_0$
\[
|\rho^{(m)}_{\kk}(p_{\kk})|\leq c_0,\text{ and }~c_0^{-1}|\rho'_{\kk}(p_{\kk})|^m\leq|\rho^{(m)}_{\kk}(p_{\kk})|\leq c_0|\rho'_{\kk}(p_{\kk})|^m,~~for~1\leq m\leq 6,
\]and as $\kk\to\infty$,
\[
\rho_{\kk}(p_{\kk})\to 1,
\]which can be considered to be passing to the incompressible limit. The result is stated as follows (See also Theorem \ref{limit}).

\begin{thm}\label{limit00}
Let $v_0,B_0$ be two divergence free vector fields with $B_0|_{\p\DD_0}=0$ such that its corrsponding pressure $q_0$ defined by
\[
\Delta \left(q_0+\frac{1}{2}|B_0|^2\right)=-(\p_i v_0^k \p_kv_0^i)+(\p_i B_0^k)(\p_k B_0^i),~~p_0|_{\p\DD_0}=0,
\]satisfies the Rayleigh-Taylor physical sign condition 
\[
-\nabla_N \left(q_0+\frac{1}{2}|B_0|^2\right)\bigg|_{\p\DD_0}\geq \epsilon_0>0.
\] Let $(v,B,q)$ be the solution to the incompressible resistive MHD equations with data $(v_0,B_0)$, i.e.,
\begin{equation}\label{iMHD00}
\begin{cases}
D_t v=B\cdot\p B-\p (q+\frac{1}{2}|B|^2)~~~& \text{in}~\DD; \\
\dive v=0~~~&\text{in}~\DD; \\
D_t B-\Delta B=B\cdot\p v,~~~&\text{in}~\DD; \\
\dive B=0~~~&\text{in}~\DD,\\
q,B|_{\p\DD_0}=0~~~&~~~\\
(v,B)|_{t=0}=(v_0,B_0).~~~&~~~
\end{cases}
\end{equation}
Furthermore, let $(u_{\kk},B_{\kk},\rho_{\kk})$ be the solution to the compressible resistive MHD equations \eqref{CMHD} with density function $\rho_{\kk}(p)$ with initial data $(u_{0,\kk},B_0,\rho_{0,\kk})$ satisfying the compatibility condition up to $(r+1)$-th order (see \eqref{cck}) as well as the physical sign condition in \eqref{sign}. If we have $\rho_{0,\kk}\to\rho_0=\beta$ ($\beta$ is the constant density in the incompressible case, WLOG set $\beta=1$) and $u_{0,\kk}\to v_0$ such that $E_{r,\kk}^*(0)$ is uniformly bounded in $\kk$, then one has 
\[
(u_{\kk},B_{\kk},\rho_{\kk})\to(v, B, \beta).
\]
\begin{flushright}
$\square$
\end{flushright}
\end{thm}
\begin{rmk}
The energy bounds are uniform with respect to the sound speed because it does not depend on the lower bound of any $\rho_{\kk}^{(m)}(p)$ which converges to 0 as $\kk\to\infty$. Also we note that, in our energy \eqref{Er}, only the highest order time derivative together with $\p D_t^4 p$ is assigned with the weight function $\wt$ or $\swt$, This together with Sobolev embedding theorem yields that the a priori quantities in \eqref{asspriori} also have $L^{\infty}$ bounds uniform in $\kk$ up to a fixed time, and thus the convergence of solutions to compressible MHD to incompressible MHD then follows.
\end{rmk}
\begin{rmk}
The density of the compressible system converges to the incompressible counterpart as the sound speed goes to infinity, but the pressure does not have analogous convergence. Instead, it should be the enthalpy $h(\rho):=\int_1^{\rho}\frac{p'(r)}{r}dr$ that converges to the pressure of incompressible system. See also \cite{lindblad2018priori}.
\end{rmk}
\item \textbf{Existence of the initial data satifying the compatibility conditions}

In Section \ref{section9}, we prove that  for every given divergence-free vector fields $v_0$ and $B_0$ with $B_0|_{\p\DD_0}=0$, there exists initial data $(u_{0,\kk},B_0,p_{0,\kk})$ satisfying the compatibility conditions \eqref{cck} when $\kk$ is sufficiently large, and also converges in our energy norm to the incompressible data as $\kk\to\infty$. Therefore, the incompressible limit exists.  

\begin{thm}
Let $(v_0,B_0,q_0)$ be the initial data for the incompressible resistive MHD equations defined in \eqref{iMHD00} with $v_0\in H^5$ and $B_0\in H^6$, $B_0|_{\p\DD_0}=0$. Then there exists initial data $(u_{0,\kk}, B_0,p_{0,\kk})$ satisfying the compatibility condition \eqref{cck} up to $6$-th order such that $(u_{0,\kk},\rho_{0,\kk})\xrightarrow{C^2} (v_{0},\beta)$ as $\kk\to\infty$, and $E_{r,\kk}^*(0)$ is uniformly bounded in $\kk$.
\end{thm}
\begin{flushright}
$\square$
\end{flushright}

\end{itemize}
\subsection{Illustration on Strategies and Difficulties}\label{1.5}
In this part, we would like to introduce our basic strategies in our proof. In particular, we will point out the essential difficulty of ideal compressible MHD, and thus the necessity of magnetic resitivity in this paper. We generalize the method in Lindblad-Luo \cite{lindblad2018priori}, but our model is very different from the free-boundary compressible Euler's equations due to the presence of $B$, the strong coupling among $B$ and $u, p$ and the presence of magnetic diffusion. Therefore, new ideas are needed to avoid the essential difficulty by utilizing the magnetic diffusion in a suitable way. These also tell the crucial difference between compressible MHD equations and Euler equations/incompressible MHD equations.

\bigskip

\noindent\textbf{Difficulty in ideal compressible MHD and necessity of magnetic diffusion}

The magnetic diffusion is necessary in our proof. We illustrate this by showing the difficulties in the study of compressible ideal MHD.

\begin{itemize}
\item \textbf{Difference from the free-boundary compressible Euler's equations: $(r+1)$-th order wave equation is out of control}

The highest order energy $E_4$ (i.e. $r=4$) contains the 5-th order energy $W_5^2$ of wave equations of $p$, which also appears in the energy of compressible Euler's equation (see Lindblad-Luo \cite{lindblad2018priori}). To bound $D_t^5 p$ and $\p D_t^4 p$, we need to take $D_t^4$ on both sides of \eqref{wave00} and study the 5-th order wave equation
\begin{equation}\label{waverr}
\wt D_t^6 p-\Delta D_t^4 p=B\cdot \Delta D_t^4 B+\cdots,
\end{equation} where the omitted terms are all of $\leq 5$ derivatives (see \eqref{wavek+10}). The control of this wave equation requires the $L^2$ norm of $\Delta D_t^4 B$. But for compressible ideal MHD, $B$ only satifies a transport equation and thus one cannot expect to enhance the regularity of $B$. This difficulty does not appear in the control of the free-boundary compressible Euler's equations, of which the corresponding wave equation \eqref{waverr} only contains $\leq 5$-th order terms on the right hand side (see Lindblad-Luo \cite{lindblad2018priori}, Section 4).

However, if we add magnetic diffusion on $B$, i.e., the equation of $B$ is modified to be
\[
D_t B-\lambda\Delta B=B\cdot\p u-B\dive u=B\cdot\p u+B\wk D_t p,~~\lambda>0\text{ is a constant},
\] and thus
\begin{equation}\label{heatrr}
D_t^5 B-\lambda\Delta D_t^4 B=B\cdot D_t^4 u+B\wk D_t^5 p+\cdots,
\end{equation} then we can plug \eqref{heatrr} into \eqref{waverr} to exactly eliminate the problematic term $B\cdot \Delta D_t^4 B$ in \eqref{waverr}.  The detailed computation is shown in Section \ref{section7}.

\item \textbf{Difference from both compressible Euler and incompressible MHD: $\curl B$ loses control}

Another crucial difference is that the control of $\curl B$ also contains a higher order term $\|\wk \p^4 D_t p\|_{\lei}$ which also requires the energy estimates of 5-th order wave equation after using elliptic estimates Proposition \ref{elliptic estimate}. This difficulty does not appear in the case of incompressible MHD (see Hao-Luo \cite{hao2014priori}, Gu-Wang \cite{gu2016construction}) due to $\dive u=0$ for incompressible MHD. Indeed, if there is no magnetic diffusion, i.e., for compressible ideal MHD, one has
\begin{equation}\label{curlcancel}
\begin{aligned}
\frac{d}{dt} K_4&=\frac{d}{dt}\idt\rho|\p^3\curl u|^2+|\p^3\curl B|^2\dx \\
&=\idt \p^3\curl (\rho D_t u)\cdot\p^3\curl u\dx+\idt \p^3\curl (D_t B)\cdot \p^3\curl B\dx +\cdots\\
&=\idt \p^3\curl (B\cdot\p B)\cdot\p^3\curl u\dx-\idt \underbrace{\p^3\curl (\p P)}_{=0}\cdot\p^3\curl u\dx \\
&~~~~+\idt \p^3\curl(B\cdot \p u)\cdot\p^3\curl B\dx+\idt \p^3\curl\left(B\wk D_tp\right)\cdot\p^3\curl B\dx+\cdots.
\end{aligned}
\end{equation} The first term on the third line will cancel the first term on the fourth line after integration by parts, up to some commutators that can be controlled. However, the last term requires the bound of $\wk\p^4 D_t p$ which is out of control. Such derivative loss is from normal derivative (div-curl decomposition reduces the normal derivatives to div and curl) and necessarily appears in compressible ideal MHD due to the coupling between magnetic field and sound wave.  One can recall the derivation of energy conservation that the term $-\frac12\idt |B|^2\dive u$ is cancelled by part of $-\idt u\cdot\nabla P$. But taking the curl eliminates the counterpart of $-\idt u\cdot\nabla P$ \textbf{before} such cancellation is produced because of $\curl\nabla P=0$. On the other hand, if we taking tangential derivatives instead of curl, then analogous cancellation is still preserved. 

\begin{rmk}
Secchi \cite{secchi1996,secchi1995} proved the LWP for the fixed-domain problem without loss of regularity in $H_*^m (m\geq 16)$, but these results are quite difficult to be applied to free-boundary problems because the free surface introduces extra derivative loss in anisotropic Sobolev spaces. So far, there is no available result proving the energy estimates without loss of regularity for the free-boundary compressible ideal MHD system.
\end{rmk}
\end{itemize}
\bigskip

\noindent\textbf{Strategy of energy estimates}

Our proof of the a priori bounds can be mainly divided into several steps: $Q$-tensor and curl estimates, boundary tensor estimates, interior and boundary elliptic estimates and the control of wave and heat equations. Important steps and illustrations are pointed out as follows, as well as in the summarizing diagram \eqref{process}.

\begin{itemize}

\item \textbf{Key observation: Magnetic diffusion together with elliptic estimates directly controls the Lorentz force}

After introducing magnetic diffusion, we should not seek for cancellation to eliminate the higher order terms which are exactly the space-time derivatives of $(B\cdot\p)B$ in $Q$-tensor and curl estimates. Notice that $(B\cdot\p)B$ vanishes on the boundary, one can apply the elliptic estimates Proposition \ref{elliptic estimate} and then plug the heat equation of $B$ to reduce to one order lower. For example, one can first reduce $\|\p^4(B\cdot\p)B\|_{\lei}$ to $\|\Delta((B\cdot \p)B)\|_{2,0}$. Then plugging $\lambda\Delta B=D_t B-(B\cdot\p)\dive u+B\dive u$ into $\|\Delta((B\cdot \p)B)\|_{2,0}$ to further reduce to the control of 4-th order derivatives. This observation is quite crucial to the whole proof: In fact $\|\p^5 B\|_{\lei}$ is out of control even if we use the magnetic diffusion, because the elliptic estimate of $\p^5 B$ requires the bound of $|\cnab^3 \theta|_{L^{\infty}(\p\DD_t)}$ which is impossible to be bounded. Our proof shows that the higher order spatial derivatives of $B$ must fall on the Lorentz force $(B\cdot\p)B$ so that we avoid the difficulty mentioned above.

\item\textbf{Boundary energies}

In the control of boundary energies, we will get
\begin{equation}\label{tgbdry00}
\ipdt Q(\p^s D_t^k P, \p^sD_t^k(D_t P)-\p_i P\p^s D_t^k u^i-\nu^{-1}N_i \p^s D_t^ku^i) dS+\cdots
\end{equation} So we choose $\nu$ to be $-(\nabla_N P)^{-1}$ in order to exactly cancel the leading order term on the boundary. Hence, the boundary control will be reduced to $|\Pi\p^s D_t^k P|_{\leb}$ and $|\Pi\p^s D_t^{k+1} P|_{\leb}$ which can be controlled by tensor estimates Proposition \ref{tensor estimate} and Proposition \ref{theta estimate}. This step also illustrates the importance of Rayleigh-Taylor sign condition in the free-boundary problem.

\item\textbf{Control of $W_{r'}^2+H_{r'}^2$ for $r'\leq r$: Bound all terms with $\leq r$ derivatives by $E_r^*$}

After using elliptic estimates and tensor estimates, the control of all the terms with $\leq r$ derivatives together with the tangential projection terms has been reduced to the control of $W_{r'}^2+H_{r'}^2$ for $r'\leq r$. Direct computation in Section \ref{section6} shows that $E_4^*$ together with $\|\p^{s-2}\Delta D_t^{k+1} B\|_{\lei},\|\p^{s-2}\Delta D_t^{k+1} p\|_{\lei}$. The latter terms will be controlled by $W_{r+1}^2+H_{r+1}^2$ as stated below.

\item\textbf{Control of $W_{r+1}^2+H_{r+1}^2$}

As mentioned above, one can reduce all the estimates to the control of wave equation of $p$ and the heat equation of $B$. With magnetic diffusion, one can simplified $\Delta D_t^k B$ to the terms with 1 lower order derivatives, and thus we can seek for the control of wave equation. In fact, the RHS of $k$-th order heat equation contains $D_t^k p$ as well as other $k$-th derivative of $p$, and the RHS of $k$-th order wave equation contains $k$-th derivative of $p$. Therefore we can try to find a common control for $W_{r+1}^2+H_{r+1}^2$ by the time integral of itself plus other terms in $\sqrt{E_r^*}$. The detailed computation are shown in Section \ref{section7}.
\end{itemize}

Our basic idea and process to close the energy estimates is briefly summarized in the following diagram.

\begin{equation}\label{process}
\begin{tikzcd}
& E_r \arrow[dr, "\text{consists of}", "~~"' sloped]\arrow[dl, "\text{consists of}", "~~"' sloped] \\ 
 E_{s,k}+K_r  \arrow[dd, "\text{reduced to}"', "~~" ] \arrow[ddr, "\text{reduced to}"', "~~"sloped ]\arrow[ ddrr,bend left, "\text{reduced to}"', "~~"sloped ]&  & W_{r+1}^2+E_{r+1}^2 \arrow[ddd,bend left,"\text{controlled by}","~"]  \\
 & \p^{s-2}\Delta D_t^{k+1}B,\p^{s-2}\Delta D_t^{k+1}p \arrow[ur,"\text{reduced to}",]& & &  \\ 
\|u\|_{r,0}, \|B\|_{r,0}  \arrow[dd, "\text{div-curl}"', "~~" ] & \|\cdot\|_{s,k}, \|\cdot\|_{s,k+1}\text{ of }B,~p,~\p^sD_t^k (B\cdot\p B) \arrow[ddl, "\text{elliptic}",  ] \arrow[d, "\text{elliptic}",  ]\arrow[u,"\text{elliptic}","~"]  &\Pi\p^s D_t^{k+1} P  \arrow[dl, "~~",  "\text{reduced to}"] \arrow[d, "\text{tensor estimate}",  ]   \\
 & \Pi\p^s D_t^{k+1} B,\Pi\p^s D_t^{k+1} p \arrow[r, "~~",  "\text{tensor estimate}"] &   E_r^* \arrow[d]  \\ E_r^* \arrow[rr]  & & \text{Closed}
\end{tikzcd}
\end{equation}

\begin{center}
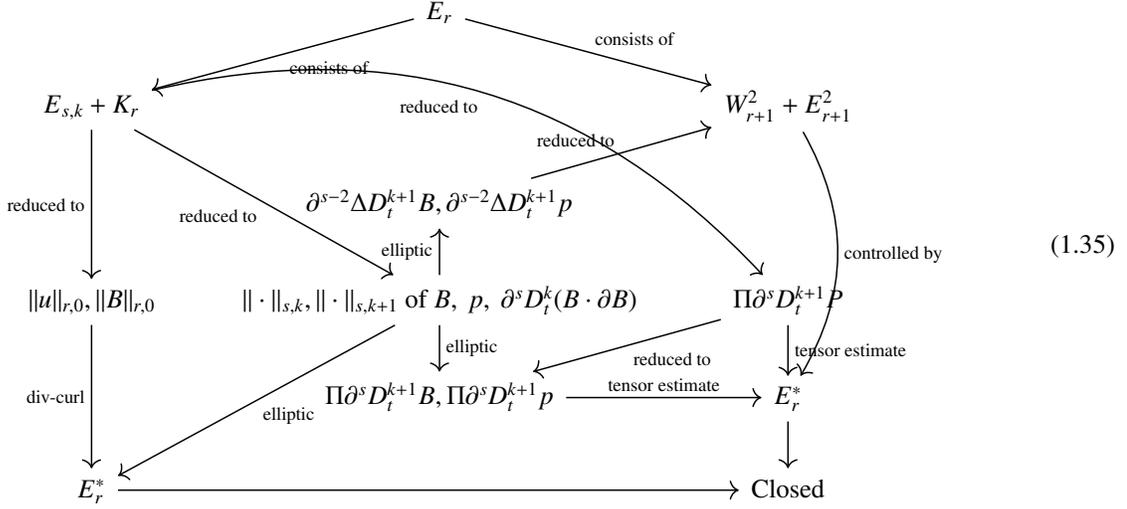
Diagram \eqref{process}: Illustration on our basic idea and process to do the a priori estimates.
\end{center}

\noindent\textbf{Incompressible limit}

Our a priori estimate in Proposition \ref{total} is uniform in the sound speed because it does not depend on the lower bound of $\wt$ which converges to 0 as the sound speed goes to infinity. We remark here that the choice of weight function in \eqref{Er} comes from the control of 5-th order wave equation \[
\wt D_t^6p-\Delta p=\frac12\Delta|B|^2+\cdots,
\]whose $L^2$-control should be established by multiplying $\wt D_t^5 p$ instead of $D_t^5 p$. Otherwise the LHS only gives the energy term $\|\swt D_t^5 p\|_0$ but RHS needs the control of $\|D_t^5 p\|_0$.
\bigskip

\noindent\textbf{Outline of this paper}

This paper is organised as follows: Section \ref{section2} and Section \ref{section3} are preliminaries on Lagrangian coordinates, elliptic estimates and tensor estimates. In Section \ref{section4} we reduce the $Q$-tensor estimates and curl estimates to the control of $\|\cdot\|_{s,k}$ norm of $u,~B,~p$ and higher order interior terms together with the boundary term \eqref{tgbdry00}. Then in Section \ref{section5} we use elliptic estimates to reduce the estimates further to the control of heat/wave equations, which is done for $\leq 4$-th order in Section \ref{section6} and for 5-th order in Section \ref{section7}. Finally, in Section \ref{section8} we summarize all of the estimates to obtain the a priori bound which is also uniform in the sound speed, and then construct the initial data satisfying the compatibility conditions to obtain the incompressible limit in Section \ref{section9}. One can also understand our idea and basic process of the energy control through the above diagram \eqref{process}.

\section{Preliminaries on Lagrangian coordinates}\label{section2}

In this section, we are going to introduce Lagrangian coordinates which reduces the free-boundary problem in $\R^n$ to an equivalent problem in a fixed domain with metric evolving as time goes. To be specific, let $\Omega$ be the unit ball in $\R^n$, and let $f_0:\Omega\to\DD_0$ be a diffeomorphism. Then the Lagrangian coordinate $(t,y)$ where $x=x(t,y)=f_t(y)$ are given by solving
\begin{align}
\frac{dx}{dt}=u(t,x(t,y)), \q x(0,y)=f_0(y),\q y\in\Omega. \label{change}
\end{align}
The boundary becomes fixed in the new coordinate, and we introduce the notation
\begin{align}
D_t = \frac{\p}{\p t}\bigg|_{y=\text{constant}} = \frac{\p}{\p t}\bigg|_{x=\text{constant}} + u^k\frac{\p}{\p x^k}. \label{lag}
\end{align}
to be the material derivative and
\begin{align*}
\p_i = \frac{\p}{\p x^i} = \frac{\p y^a}{\p x^i}\frac{\p}{\p y^a}.
\end{align*}
Due to (\ref{lag}), we can also consider the material derivative $D_t$ as the time derivative by slightly abuse of terminology.\\

 Sometimes it is convenient to work in the Eulerian coordinate $(t,x)$, and sometimes it is easier to work in the Lagrangian coordinate $(t,y)$. In the Lagrangian coordinate the partial derivative $\p_t=D_t$ has more direct significance than it in the Eulerian frame. However, this is not true for spatial derivatives $\p_i$. Instead, the ``suitable" spatial derivative to characterize the motion of the fluid is the covariant differentiation with respect to the metric $g_{ab}(t,y)=\delta_{ij}\frac{\p x^i}{\p y^a}\frac{\p x^j}{\p y^b}$ assigned to $\Omega$. 

Here we mention that covariant derivative is not involved in our imposed energy function. Instead, we use the standard Eulerian spatial derivatives. We will work mostly in the Lagrangian coordinate in this paper. However, our statements are coordinate independent.\\

The Euclidean metric $\delta_{ij}$ in $\DD_{t}$ induces a metric
\begin{align}
g_{ab}(t,y)=\delta_{ij}\frac{\p x^i}{\p y^a}\frac{\p x^j}{\p y^b}, \label{g}
\end{align}
in $\Omega$ for each fixed $t$. We will denote covariant differentiation in the $y_{a}$-coordinate by $\nab_a$, $a=1,\cdots,n$, and the differentiation in the $x_i$-coordinate by $\p_i$, $i=1,\cdots,n$. Here, we use the convention that differentiation with respect to Eulerian coordinates is denoted by letters $i,j,k,l$ and with respect to Lagrangian coordinate is denoted by $a,b,c,d$.

\indent The regularity of the boundary is measured by that of the normal: Let $N^a$ to be the unit normal to $\p\Omega$, ie.e, $g_{ab}N^aN^b=1,$ and let $N_a=g_{ab}N^b$ denote the unit conormal, $g^{ab}N_aN_b=1$. The induced metric $\gamma$ on the tangent space to the boundary $T(\p\Omega)$ extended to be $0$ on the orthogonal complement in $T(\Omega)$ is given by
\[
\gamma_{ab}=g_{ab}-N_aN_b,\q \gamma^{ab}=g^{ac}g^{bd}\gamma_{cd}=g^{ab}-N^aN^b.
\]
The orthogonal projection of an $(0,r)$ tensor $S$ onto the boundary is given by
\[
(\Pi S)_{a_1,\cdots,a_r}=\gamma_{a_1}^{b_1}\cdots\gamma_{a_r}^{b_r}S_{b_1,\cdots,b_r},
\]
where $\gamma_{a}^{b}=g^{bc}\gamma_{ac}=\delta_{a}^{b}-N_aN^b$. In particular, the covariant differentiation on the boundary $\overline{\nab}$ is given by
\[
\overline{\nab}S=\Pi \nab S.
\]
We note that $\overline{\nab}$ is invariantly defined since the projection and $\nab$ are. The second fundamental form of the boundary, denoted by $\theta$, is given by $\theta_{ab}=(\cnab N)_{ab}$, and the mean curvature of the boundary $\sigma=tr\theta=g^{ab}\theta_{ab}$.

It is now important to compute time derivative of the metric $D_tg$, the normal $D_tN$, as well as the time derivative of corresponding measures.

\lem \label{Lemma 1.1}
Let $x=f_t(y)=x(t,y)$ be the change of variable given by
\begin{equation}\label{lagrangiancoordinates}
\frac{dx}{dt}=u(t,x(t,y)), \q x(0,y)=f_0(y),\q y\in\Omega,
\end{equation}
and $g_{ab}(t,y)=\delta_{ij}\frac{\p x^i}{\p y^a}\frac{\p x^j}{\p y^b}$ to be the induced metric. In addition, we let $\gamma_{ab} = g_{ab}-N_a N_b$, where $N_a= g_{ab}N^b$ is the co-normal to $\p\Omega$. Now we set
\begin{align}
u_a(t,y) &= u_i(t,x)\frac{\p x^i}{\p y^a}, \q u^a = g^{ab}u_b,\\
d\mu_g:&\text{ The volume unit with respect to the metric}~g ,\\
d\mu_\gamma:&\text{ The surface area unit with respect to the metric}~\gamma.\label{dmu}
\end{align}
Then the following result holds
\begin{align}
D_t g_{ab} &= \nab_a u_b+\nab_b u_a,\label{Dtg}\\
D_t g^{ab} &= -g^{ac}g^{bd}D_tg_{cd},\label{Dtg_inverse}\\
D_t N_a& = -\frac{1}{2}N_a(D_tg^{cd})N_cN_d,\label{DtN}\\
D_t d\mu_g &= \di u\,d\mu_g\label{dg},\\
D_t d\mu_\gamma &=(\sigma u \cdot N)\,d\mu_\gamma. \label{T2'}
\end{align}
\begin{proof} We only briefly state the sketch of the proof. Actually these results all come from direct computation, of which the details can be found in \cite{lindblad2018priori}, Section 2.

The fact that $D_t$ commutes with $\p_y$ together with $D_t x(t,y) = u(t,y)$ yields that
\[
D_t\frac{\p x^i}{\p y^a} = \frac{\p u_i}{\p y^a} = \frac{\p x^k}{\p y^a}\frac{\p u_i}{\p x^k},
\]
and thus
\[
D_t g_{ab} = \sum_{i}D_t\bigg(\frac{\p x^i}{\p y^a}\frac{\p x^i}{\p y^b}\bigg)= \frac{\p x^k}{\p y^a}\frac{\p u_i}{\p x^k}\frac{\p x^i}{\p y^b}+\frac{\p x^i}{\p y^a}\frac{\p x^k}{\p y^b}\frac{\p u_i}{\p x^k} = \nab_a u_b+\nab_bu_a.
\]

\eqref{Dtg_inverse} follows from $0=D_t(g^{ab}g_{bc})=D_t(g^{ab})g_{bc}+g^{ab}D_tg_{bc}$, and \eqref{dg} follows since in local coordinate we have $d\mu_{g} = \sqrt{\detg}\,dy$ and $D_t\detg = (\detg) g^{ab}D_t g_{ab} = 2{\detg}\, \dive u$. To prove \eqref{DtN}, we choose the local foliation $f$ so that $\p\Omega=\{y:f(y)=0\}$ and $f<0$ in $\Omega$, then
\[
N_a=\frac{\p_a f}{\sqrt{{g^{cd}}\p_cf\p_df}},
\]
and \eqref{DtN} follows from direct computation.

Now, \eqref{DtN} together with $d\mu_\gamma = \frac{\sqrt{\det g}}{\sqrt{\sum N_n^{2}}}\,dS(y)$ implies $D_t d\mu_\gamma = \dive u+\frac{1}{2}(D_tg^{cd})N_cN_d,$ where $dS(y)$ is the Euclidean surface measure.

To prove \eqref{T2'}, one first uses $\dive u=g^{ab}D_tg_{ab}/2$ together with \eqref{Dtg} and \eqref{Dtg_inverse} to obtain
\[
D_t d\mu_\gamma = \frac{1}{2}g^{ab}D_tg_{ab}-\frac{1}{2}(D_tg_{ab})N^aN^b=\gamma^{ab}\nab_au_b.
\] And finally \eqref{T2'} holds since $\gamma^{ab}\nab_au_b=\gamma^{ab}\cnab_a(N_bu\cdot N)+\gamma^{ab}\cnab_a\overline{u}_b$, and $\gamma^{ab}\cnab_a\overline{u}_b=\di u|_{\p\Omega}=0$.
\end{proof}

\section{Elliptic estimates on a bounded domain with a moving boundary}\label{section3}

In this section, we are going to introduce the elliptic estimates and tensor estimates of tangential projections which will be used repeatedly  in the remaining part of this paper. All the results in this section will be stated in a coordinate-independent way.

 Throughout this section, $\Omega$ is a bounded domain in $\R^n$ with $n\geq 2$. $\nab$ denotes the covariant derivative with respect to the metric $g_{ij}$ in $\Omega$, and $\cnab$ denotes the covariant differentiation on $\p\Omega$ with respect to the induced metric $\gamma_{ij}=g_{ij}-N_iN_j$. In this section (and only), $\Omega$ denotes a general domain with smooth boundary. In addition, we assume the normal vector $N$ to $\p\Omega$ is extended to a vector field in the interior of $\Omega$ satisfying $g_{ij}N^iN^j\leq 1$ by the same way as in Lemma \ref{trace 1}.

\subsection{Elliptic estimates}

\defn \label{tensor} (Differentiations) Let $u:\Omega\subset\R^n\to\R^n$ be a smooth vector field, and $\beta_k=\beta_{Ik}=\nab_{I}^{r}u_k$ be the $(0,r)$-tensor defined based on $u_k$, where $\nab_{I}^{r}=\nab_{i_1}\cdots\nab_{i_r}$ and $I=(i_1,\cdots,i_r)$ is the set of indices. Define $\di \beta_k=\nab_i\beta^{i}=\nab^r\di u$ and $\curl \beta=\nab_i\beta_j-\nab_j\beta_i=\nab^r\curl u_{ij}$.
\defn (Norms) Suppose $|I|=|J|=r$, $g^{IJ}=g^{i_1j_1}\cdots g^{i_r j_r}$ and $\gamma^{IJ}=\gamma^{i_1j_1}\cdots\gamma^{i_rj_r}$. For any $(0,r)$ tensors $\alpha$, $\beta$, we define $\langle\alpha,\beta\rangle=g^{IJ}\alpha_I \beta_J$ and $|\alpha|=\langle\alpha,\alpha\rangle$. If $(\Pi\beta)_{I}=\gamma_{I}^{J}\beta_J$ is the projection, then $\langle \Pi \alpha,\Pi\beta\rangle=\gamma^{IJ}\alpha_I\beta_J$. Also we define
\[
\begin{aligned}
\|\beta\|_{\lli}&=\bigg(\int_{\Omega}|\beta|^2\,d\mu_g\bigg)^{\frac{1}{2}},\\
|\beta|_{\llb}&=\bigg(\int_{\p\Omega}|\beta|^2\,d\mu_\gamma\bigg)^{\frac{1}{2}},\\
|\Pi\beta|_{\llb}&=\bigg(\int_{\p\Omega}|\Pi\beta|^2\,d\mu_\gamma\bigg)^{\frac{1}{2}}.
\end{aligned}
\]

Now we introduce the following Hodge's decomposition theorem, which is crucial in the control of full spatial derivatives of $u$ and $B$.

\thm \label{hodge}(Hodge's Decomposition Theorem) Let $\beta$ be defined in Definition \ref{tensor}. Suppose $|\theta|+|\frac{1}{\iota_0}|\leq K$, where $\theta$ is the second fundamental form of $\p\Omega$ and $\iota_0$ is the injective radius defined in (\ref{inj rad}), then
\begin{align}
|\nab\beta|^2&\lesssim g^{ij}\gamma^{kl}\gamma^{IJ}\nab_k\beta_{Ii}\nab_l\beta_{Jj}+|\di \beta|^2+|\curl\beta|^2\\
\int_{\Omega}|\nab\beta|^2\,d\mu_g&\lesssim\int_{\Omega}(N^iN^jg^{kl}\gamma^{IJ}\nab_k\beta_{Ii}\nab_l\beta_{Jj}+|\di \beta|^2+|\curl\beta|^2+K^2|\beta|^2)\,d\mu_g.
\end{align}
\begin{proof}
See \cite{christodoulou2000motion} (Lemma 5.5) for details.
\end{proof}

\prop \label{elliptic estimate}(Interior/boudnary elliptic estimates) Let $q:\Omega\to\R$ be a smooth function. Suppose that $|\theta|+|\frac{1}{\iota_0}|\leq K$, then we have, for any $r\geq2$ and $\delta>0$,
\begin{align}
\|\nab^r q\|_{\lli}+|\nab^r q|_{\llb}&\lesssim_{K,\vol\Omega}\sum_{s\leq r}|\Pi\nab^s q|_{\llb}+\sum_{s\leq r-1}||\nab^s\lap q||_{\lli},\label{ell est I}\\
\|\nab^r q\|_{\lli}+|\nab^{r-1} q|_{\llb} &\lesssim_{K,\vol\Omega}\delta\sum_{s\leq r}|\Pi\nab^s q|_{\llb}+\delta^{-1}\sum_{s\leq r-2}\|\nab^s\lap q\|_{\lli}.\label{ell est II}
\end{align}
where we have applied the convention that $A\lesssim_{p,q}B$ means $A\leq C_{p,q}B$.
\begin{proof}
See \cite{christodoulou2000motion} (Proposition 5.8) for details.
\end{proof}

\subsection{Estimates of tangential projections}
The projection of the tensor $\Pi\nab^sD_t^kP$ appears in the boundary part of our imposed energy \eqref{Er} as well as the elliptic estimates as in Proposition \ref{elliptic estimate}. It is crucial to compensate the possible loss of regularity with the help of tensor estimates below. 

Actually, one may simply observe that: If $q=0$ on $\p\Omega$, then $\Pi\nab^2q$ only contains the first order derivatives of $q$ and all components of the second fundamental form. Specifically, one has
\begin{equation}
\Pi\nab^2q=\cnab^2q+\theta\nab_Nq,\label{21}
\end{equation}
where the tangential component $\cnab^2q=0$ on the boundary. 

Furthermore, \eqref{21} gives the following control:
\begin{equation}
|\Pi\nab^2q|_{\llb}\leq |\theta|_{L^{\infty}(\p\Omega)}|\nab_N q|_{\llb}.\label{22}
\end{equation}

To prove \eqref{21}, first invoking the components of the projection operator $\gamma_{i}^{j}=\delta_{i}^{j}-N_iN^j$, then one has
\[
\gamma_{j}^{k}\nab_i\gamma_{k}^{l}=-\gamma_{j}^{k}\nab_i(N_kN^l)=-\gamma_j^k\theta_{ik}N^l-\gamma_j^kN_k\theta_i^l=-\theta_{ij}N^l,
\]
and thus
\[
\begin{aligned}
&~~~~\cnab_i\cnab_jq=\gamma_{i}^{i'}\gamma_{j}^{j'}\nab_{i'}\gamma_{j'}^{j''}\nab_{j''}q\\
&=\gamma_i^{i'}\gamma_j^{j'}\gamma_{j'}^{j''}\nab_{i'}\nab_{j''}q+\gamma_i^{i'}\gamma_{j}^{j'}(\nab_{i'}\gamma_{j'}^{j''})\nab_{j''}q\\
&=\gamma_{i}^{i'}\gamma_{j}^{j'}\nab_{i'}\nab_{j'}q-\theta_{ij}\nab_Nq.
\end{aligned}
\]
In general, the higher order projection formula is of the form
$$
\Pi \nab^r q = (\cnab^{r-2}\theta)\nab_N q+O(\nab^{r-1}q)+O(\cnab^{r-3}\theta),
$$
which yields the following generalisation of (\ref{22}). Its detailed proof can be found in \cite{christodoulou2000motion}.
\prop \label{tensor estimate}(Tensor estimate of tangential projections)
Suppose that $|\theta|+|\frac{1}{\iota_0}|\leq K$, and for $q=0$ on $\p\Omega$, then for $m=0,1$
\begin{align}
|\Pi\nab^{r}q|_{\llb}&\lesssim_{K} |(\cnab^{r-2}\theta)\nab_{N}q|_{\llb}+\sum_{l=1}^{r-1}|\nab^{r-l}q|_{\llb},\\
&~~~~ +(|\theta|_{\linf}+\sum_{0\leq l\leq r-2-m}|\cnab^{l}\theta|_{\llb})(\sum_{0\leq l\leq r-2+m}|\nab^l q|_{\llb}), \label{tensor est}
\end{align}
where the second line drops for $0\leq r\leq 4$.
\begin{proof}
See \cite{christodoulou2000motion} (Proposition 5.9).
\end{proof}

\subsection{Estimate for the second fundamental form on the boundary}
The estimate on the second fundamental form $\theta$ is a direct result of Proposition \ref{tensor estimate} with $q=P$ together with the Rayleigh-Taylor sign condition, e.g., $|\nab_N P|\geq -\nab_N P\geq \epsilon_0>0$.
\prop \label{theta estimate}($\theta$ estimate) \footnotemark
Assume that $0\leq r\leq4$. Suppose that $|\theta|+|\frac{1}{\iota_0}|\leq K$, and the Taylor sign condition $|\nab_N P|\geq \epsilon>0$ holds, then
\begin{equation}
|\cnab^{r-2}\theta|_{\llb}\lesssim_{K, \frac{1}{\epsilon_0}}|\Pi \nab^rP|_{\llb}+\sum_{s=1}^{r-1}|\nab^{r-s}P|_{\llb}. \label{theta est}
\end{equation}

\begin{flushright}
$\square$
\end{flushright}
\begin{rmk}
We point ou that the estimates of $\theta$ suggests that the boundary regularity is in fact controlled by the boundary $L^2$ -norm of $P$, with a loss of $2$ derivatives.
\end{rmk}
\bigskip

\section{Energy estimates}\label{section4}

\subsection{Tangential ($Q$-tensor) estimates when $s\geq 1$}
In this section, we will show the estimates of $E_{s,k}$, i.e., the estimates of $Q$-tensors and $\curl$, when $s\geq 1$. We will work under the Eulerian coordiantes so that we need not worry about the Christoffel symbols. We use the notation 
\begin{itemize}
\item $\|f\|_{s,k} = \|\p^sD_t^k f\|_{\lei}$,
\item $|f|_{s,k}= |\p^sD_t^k f|_{\leb}$.
\end{itemize}

We start with the velocity field.

\begin{equation}
\begin{aligned}
&~~~~\frac{1}{2}\frac{d}{dt}\int_{\DD_t}\rho Q(\p^sD_t^k u, \p^sD_t^k u)dx=\int_{\DD_t}\rho Q(\p^sD_t^k u, \p^sD_t^{k+1}u)dx+\underbrace{\int_{\DD_t}\rho Q(\p^sD_t^k u, [D_t, \p^s]D_t^k u)dx}_{R_1} \\
&=\int_{\DD_t}\rho Q(\p^sD_t^k u, \p^sD_t^{k}(\rho D_t u))dx+\underbrace{\int_{\DD_t}\rho Q(\p^sD_t^k u, [\rho, \p^sD_t^k]D_t u)dx}_{R_2}+R_1\\
&=\int_{\DD_t}\rho Q(\p^sD_t^k u, \p^sD_t^{k}(B\cdot\p B))dx-\int_{\DD_t}\rho Q(\p^sD_t^k u^i, \p^sD_t^{k}\p_i P)dx+R_1+R_2\\
&=:I_1+I_2+R_1+R_2,
\end{aligned}
\end{equation}where we use the first equation of MHD system \eqref{CMHD}.

The estimates \eqref{extending nab n}-\eqref{Dtgamma} together with a priori assumptions imply the following inequalities, of which the proof can be found in Section 3 of \cite{christodoulou2000motion}.

\[|D_tq^{ij}|\lesssim M, \q |\p q^{ij}|\lesssim M+K,\q  |\sigma u\cdot N|_{L^{\infty}(\p\Omega)}\lesssim K+M,\]

\[
|D_t\nu|_{L^{\infty}(\p\Omega)} = |D_t(-\nab_N P)^{-1}|_{L^{\infty}(\p\Omega)}\lesssim 1+\frac{1}{M},
\]
and
\begin{align}
D_t\gamma^{ij}=-2\gamma^{im}\gamma^{jn}(\frac{1}{2}D_tg_{mn}) \label{bdy q}.
\end{align}

Now we have
\begin{equation}
I_1\lesssim_{K,M}\|u\|_{s,k}\|(B\cdot\p) B\|_{s,k}.
\end{equation}

For $I_2$, we first commute $\p_i$ with $\p^sD_t^k$, then integrate $\p_i$ by parts, and finally try to construct the $Q$-tensor of $p$ by using the continuity equation. 

\begin{equation}
\begin{aligned}
I_2&=-\int_{\DD_t} Q(\p^sD_t^k u^i,\p_i\p^sD_t^k P)dx\underbrace{-\int_{\DD_t}Q(\p^sD_t^k u^i, \p^s([\p_i,D_t^k]P))}_{R_3}\\
&=\int_{\DD_t}Q(\p^sD_t^k\dive u, \p^sD_t^k P)dx+\underbrace{\int_{\DD_t}Q(\p^s([\p_i, D_t^k]u^i), \p^sD_t^k P)dx}_{R_4}\underbrace{-\int_{\p\DD_t}Q(\p^sD_t^k P, N_i\p^sD_t^k u^i)dS}_{R_1^*}+R_3
\end{aligned}
\end{equation}

Plugging $P=p+\frac{1}{2}|B|^2$ and the continuity equation into the first term, we can get the $Q$-tensor of $p$.

\begin{equation}
\begin{aligned}
&~~~~\idt Q(\p^sD_t^k\dive u, \p^sD_t^k P)dx\\
&=-\idt Q(\p^sD_t^k(\wk), \p^sD_t^k(\frac{1}{2}|B|^2))-\idt Q(\p^sD_t^k(\wk),\p^sD_t^k p)dx\\
&=\underbrace{-\idt Q(\wk \p^sD_t^{k+1}p,\p^sD_t^k(\frac{1}{2}|B|^2))dx}_{I_{21}}-\idt\wk Q(\p^sD_t^{k+1}p,\p^sD_t^k p)dx \\
&~~~~\underbrace{-\idt Q([\p^sD_t^k, \wk]D_t p, \p^sD_t^k(\frac{1}{2}|B|^2))dx}_{R_5}\underbrace{-\idt Q([\p^sD_t^k, \wk]D_t p, \p^sD_t^k p)dx}_{R_6}\\
&=I_{21}-\frac{1}{2}\frac{d}{dt}\idt \wk Q(\p^sD_t^k p, \p^sD_t^k p)dx +R_5+R_6 \\
&~~~~\underbrace{-\idt \wk Q(\p^sD_t^k p, [\p^s, D_t]D_t^k p)dx}_{R_7}+\underbrace{\frac{1}{2}\idt\rho D_t(\frac{\rho'(p)}{\rho^2})Q(\p^sD_t^k p, \p^sD_t^k p)\dx}_{R_8}.
\end{aligned}
\end{equation}Also we have 
\begin{equation}
I_{21}\lesssim_{K,M}\|\rho'(p)\p^sD_t^{k+1}p\|_{\lei}\|B\|_{s,k}.
\end{equation}

Next we control the other terms in $E_{s,k}$. Since $|D_tq^{ij}|\lesssim M$ in the interior and on the boundary $q^{ij}=\gamma^{ij}$, and by \eqref{bdy q} $D_t \gamma$ is tangential, one has

\begin{equation}
\begin{aligned}
&~~~~\frac{d}{dt}\frac{1}{2}\ipdt Q(\p^sD_t^k P, \p^sD_t^k P)\nu dS\\
&=\underbrace{\ipdt Q(\p^sD_t^kP, D_t\p^sD_t^k P)\nu dS}_{R_2^*}\\
&~~~~+\underbrace{\ipdt  \frac{1}{2}Q(\p^sD_t^k P, \p^sD_t^k P)D_t\nu -(\sigma u\cdot N)Q(\p^sD_t^k P, \p^sD_t^k P)\nu dS}_{R_9}.\\
\end{aligned}
\end{equation}

For the $Q$-tensor estimates of the magnetic field $B$, one should not plug the third equation in \eqref{CMHD} here, otherwise $\lambda\Delta B$ will appear and produce higher order terms on the boundary which cannot be controlled. Instead, we directly use $\|B\|_{s,k+1}$ to control the $Q$-tensor, and then reduce it to the control of the parabolic equation of $B$ in Section 5. 
\begin{equation}
\begin{aligned}
&~~~~\frac{1}{2}\frac{d}{dt}\idt Q(\p^sD_t^k B, \p^s D_t^k B)dx \\
&=\idt Q(\p^sD_t^k B, \p^s D_t^{k+1}B)\dx\\
&~~~~+\idt Q(\p^sD_t^k B, [D_t,\p^s]D_t^k B)dx+\idt \rho D_t(1/\rho)Q(\p^sD_t^k B, \p^s D_t^k B)\dx\\
&=:I_3+R_{10}+R_{11},
\end{aligned}
\end{equation}where
\begin{equation}
I_3\lesssim_{K,M}\|B\|_{s,k}\|B\|_{s,k+1}.
\end{equation}

We point out that, $R_1, R_7, R_9, R_{10}$ and the boundary terms $R_1^*,R_2^*$ vanish if $s=0$, in the case of which we can drop the $Q$-tensor notation because there is no spatial derivative. Therefore we have
\begin{equation}
\begin{aligned}
\frac{d}{dt}\sum_{s+k=r,s\geq 1} E_{s,k}&\lesssim \|u\|_{s,k}\|(B\cdot \p)B\|_{s,k}+\|\rho'(p)\p^sD_t^{k+1}p\|_{\lei}\|B\|_{s,k}+\|B\|_{s,k}\|B\|_{s,k+1} \\
&~~~~+R_1+\cdots+R_{11}+R_1^*+R_2^*.
\end{aligned}
\end{equation}

\subsection{Energy estimates of full time derivatives}

When there is no spatial derivative, we need to add weight $\sqrt{\wt}$ in $u$, i.e. $$E_{0,r}=\frac{1}{2}\left(\idt \rho\wt |D_t^r u|^2\dx+\idt |D_t^r B|^2\dx+\idt \wk|D_t^r p|^2\dx\right).$$ When computing $\frac{d}{dt} E_{0,r}$, there will be some terms that $D_t$ falls on the weight function, but these terms can all be controlled by $E_{0,r}$ because $|\rho^{(m)}(p)|\lesssim c_0\swt$. Therefore one can get a similar estimate as above:
\begin{equation}
\begin{aligned}
\frac{d}{dt} E_{0,r}&\lesssim_{K,M} \|\swt D_t^4 u\|_{\lei}\|(B\cdot \p)B\|_{0,4}+\|\rho'(p)\p D_t^{4}p\|_{\lei}\|B\|_{0,4}+\|B\|_{0,4}\|B\|_{0,5} \\
&~~~~+R_2+\cdots+R_6+R_8+R_{11}.
\end{aligned}
\end{equation}

\subsection{Curl estimates}

Similarly as above, one has
\begin{equation}
\begin{aligned}
&~~~~\frac{1}{2}\frac{d}{dt}\idt \rho|\curl \p^{r-1}u|^2+|\curl \p^{r-1}B|^2 \dx \\
&=\idt \curl \p^{r-1} u\cdot\curl \p^{r-1} (\rho D_t u)\dx+\underbrace{\idt \curl \p^{r-1}B \cdot \curl \p^{r-1}D_t B\dx}_{I_4}+R_{12}+\cdots+R_{15}\\
&=\underbrace{\idt \curl \p^{r-1}u\cdot\curl\p^{r-1}(B\cdot\p B)\dx}_{I_5}+\idt\curl\p^{r-1}u\cdot \p^{r-1}\underbrace{\curl(\p P)}_{=0} \dx +I_4+R_{12}+\cdots+R_{15},
\end{aligned}
\end{equation}where the remainder terms $R_{12},\cdots, R_{15}$ are defined by:
\[
\begin{aligned}
R_{12}&:=\idt\rho \curl \p^{r-1}u\cdot [D_t, \curl \p^{r-1}]u \dx \\
R_{13}&:=\idt \curl \p^{r-1}B\cdot [D_t, \curl \p^{r-1}]B \dx \\
R_{14}&:=\idt\rho D_t(1/\rho^2)|\curl \p^{r-1}B|^2 \dx \\
R_{15}&:=\idt \curl\p^{r-1}u\cdot[\rho,\curl\p^{r-1}]D_t u \dx .
\end{aligned}
\]

$I_4$ and $I_5$ can also be similarly proceeded as $I_1$ and $I_3$:
\begin{equation}
I_4\lesssim_{K,M} \|B\|_{r,0}\|B\|_{r,1},~~~I_5\lesssim_{K,M}\|u\|_{r,0}\|(B\cdot\p)B\|_{r,0}.
\end{equation}

Combining all the estimates above, we now have:
\begin{equation}
\begin{aligned}
\frac{d}{dt}\left(\sum_{s+k=r} E_{s,k}+K_r\right)&\lesssim_{K,M} \|u\|_{s,k}\|(B\cdot \p)B\|_{s,k}+\|\rho'(p)\p^sD_t^{k+1}p\|_{\lei}\|B\|_{s,k}+\|B\|_{s,k}\|B\|_{s,k+1} \\
&~~~~+R_1+\cdots+R_{15}+R_1^*+R_2^*.
\end{aligned}
\end{equation}

Therefore the $Q$-tensor and curl estimates are all reduced to the higher order terms ($I_1,\cdots, I_5, R_1^*,R_2^*$) and the remainders. Next step we will control all the remainders by $\|u\|_{r,0},\|p\|_{s,k}$ and $\|B\|_{s,k}$. The reduction of those higher order terms will be shown in Section 5.

\subsection{The precise form of commutators between $D_t$'s and spatial derivatives}

Here we present the precise form of commutators which will be used repeatedly in the control of remainders. \eqref{commutator1}, \eqref{commutator2}, \eqref{commutator4} are the same as in (4.5)-(4.7) in Lindblad-Luo \cite{lindblad2018priori}. \eqref{commutator3} is a direct consequence of Leibniz rule and \eqref{commutator2}.

\begin{equation}\label{commutator1}
[D_t,\p^r]=\sum_{s=0}^{r-1}\p^s[D_t,\p]\p^{r-s-1}=\sum_{s=0}^{r-1}-C_{s+1}^r(\p^{1+s}u) \symdot \p^{r-s},  
\end{equation}where
\[
((\p^{1+s}u) \symdot \p^{r-s})_{i_1,\cdots,i_r}=\frac{1}{r!}\sum_{\sigma\in S_r}(\p^{1+s}_{i_{\sigma_1}\cdots i_{\sigma_{1+s}}}u^k)(\p^s_{k,i_{\sigma_{s+2}\cdots i_{\sigma_{r}}}}).
\]$S_r$ is the $r$-symmetric group.

\begin{equation}\label{commutator2}
~ [\p,D_t^k]= \sum_{l_1+l_2=k-1}c_{l_1,l_2}(\p D_t^{l_1}u)\symdot(\p D_t^{l_2})+ \sum_{l_1+\cdots+ l_n= k-n+1, \, n\geq 3} d_{l_1,\cdots,l_n}(\p D_t^{l_1} u)\cdots (\p D_t^{l_{n-1}} u) (\p D_t^{l_n}).
\end{equation}

\begin{equation}\label{commutator3}
[D_t^k, B\cdot \p]=\sum_{j=0}^{k-1}C_k^j D_t^{k-j} B^l \p_l D_t^j+\sum_{j=1}^k C_k^j (D_t^{k-j} B^l)[D_t^j,\p_l].
\end{equation}

\begin{equation}\label{commutator4}
\begin{aligned}
~[D_t^{r-1},\Delta]&=(\p D_t^{l_1}u)\cdot (\p^2 D_t^{l_2})\\
&~~~~+ \sum_{l_1+\cdots+ l_n= r-n, \, n\geq 3} d_{l_1,\cdots,l_n}(\p D_t^{l_3} u)\cdots (\p D_t^{l_{n}} u)\cdot (\lap D_t^{l_1} u)\cdot(\p D_t^{l_2})\\
&~~~~+ \sum_{l_1+\cdots+ l_n= r-n, \, n\geq 3} e_{l_1,\cdots,l_n}(\p D_t^{l_3} u)\cdots (\p D_t^{l_{n}} u)\cdot (\p^2 D_t^{l_1} u)\cdot(\p D_t^{l_2})\\
&~~~~+\sum_{l_1+\cdots+ l_n= r-n, \, n\geq 3} f_{l_1,\cdots,l_n}(\p D_t^{l_3} u)\cdots (\p D_t^{l_n}u)\cdot (\p D_t^{l_1} u)\cdot (\p^2 D_t^{l_2}),
\end{aligned}
\end{equation}

\subsection{Remainder and commutator estimates}

\subsubsection*{1. Boundary term $R_1^*+R_2^*$}

Recall that $\nu=(-\p P/\p N)^{-1}$, so $\nu^{-1}N_i=\p_i P$. Therefore, $R_1^*+R_2^*$ becomes
\[
R_1^*+R_2^*=\ipdt \rho Q(\p^sD_t^k P, D_t\p^sD_t^k P+(\p_iP)(\p^sD_t^k u^i))\nu dS.
\]
When $s=0$ or 1, $R_1^*+R_2^*$ vanishes because $D_t$ and $\Pi \p^1=\bar{\p}$ are both tangential derivatives of the moving boundary $\p\DD_t$ on which $P=0$. For $s\geq 2$, the simplification is exactly the same as (5.14)-(5.15) in Lindblad-Luo  \cite{lindblad2018priori}:

\begin{equation}
\begin{aligned}
s=r, k=0&:~~\Pi(D_t\p^r P+(\p_i P)\p^r u^i)=\Pi \p^rD_t P+\sum_{m=0}^{r-2}d_{mr}\Pi((\p^{m+1}u)\symdot \p^{r-m}P) \\
2\leq s<r&:~~\Pi(D_t\p^sD_t^k P+(\p_i P)(\p^s D_t^k u^i))=\Pi\p^sD_t^{k+1}P+\Pi((\p_i P)(\p^sD_t^k u^i))\\
&~~~~~~~~~~~~~~~~~~~~~~~~~~~~~~~~~~~~~~~~~~~~~~~~~~~~~~~~~~~~~~~~~+\sum_{m=0}^{s-1} d_{mr}\Pi((\p^{m+1}u)\symdot \p^{s-m}D_t^k P).
\end{aligned}
\end{equation}

\begin{rmk}
In the last term on the first line, the summation is taken from 0 to $r-2$ instead of $r-1$ because $\Pi \p^r P$ is cancelled by the commutator. This is essential for our estimate: One cannot control $\Pi \p^r u$ on the boundary because $u\neq 0$ on $\p\DD_t$ causes loss of regularity. However, $|\Pi\p^sD_t^k u|_{\leb}$ can be controlled when $k\geq 1$ since we can use the first equation of \eqref{CMHD} to reduce this term to $|\Pi\p^{s+1}D_t^{k-1} B|_{\leb}$ and $|\Pi\p^{s+1}D_t^{k-1} p|_{\leb}$, which can be controlled by the elliptic estimates.
\end{rmk}

Hence, by H\"older's inequality we have
\begin{equation}
\begin{aligned}
R_1^*+R_2^*&\lesssim_{K, M}\sum_{k+s=r,s\geq 2}\bigg(|\Pi\p^s D_t^k P|_{\leb} \Big(|\Pi\p^s D_t^{k+1} P|_{\leb}\\
&~~~~+|\Pi(\p_i P)(\p^sD_t^k u^i)|_{\leb}+\sum_{0\leq m\leq s-1}|\Pi((\p^{m+1}u)\symdot\p^{s-m}D_t^k P)|_{\leb}\Big)\bigg)\\
&~~~~+|\Pi\p^r P|_{\leb}\Big(|\Pi \p^r D_t P|_{\leb}+\sum_{0\leq m \leq r-2}|\Pi((\p^{m+1}u)\symdot \p^{r-m}P)|_{\leb}\Big).
\end{aligned}
\end{equation}

\subsubsection*{2. Interior terms $R_1+\cdots+R_{15}$}

We are going to control all the remainders $R_1\cdots,R_{15}$. For simplicity we only show the details for top order case, i.e. $s+k=4$. For the lower order cases, we only list the result and omit the proof.
\bigskip

(1) $R_1=\idt \rho Q(\p^sD_t^k u, [D_t,\p^s]D_t^k u)\dx$.

Since $$[D_t,\p^s]D_t^k u=-\sum_{m=0}^{s+1}C_{s}^{m+1}\p^{1+m}u\symdot\p^{s-m}D_t^k u,$$ we know
\begin{itemize}
\item $s\geq 2$: $R_1\lesssim_{K, M}\|u\|_{s,k}(\|u\|_{s,k}+\|u\|_{s-1,k})$;

\item $s=1,~k=3$: $R_1\lesssim_{K, M} \|u\|_{1,3}^2$;

\item $s=0,~k=4: R_1=0.$
\end{itemize}

(2) $R_2=\idt Q(\p^sD_t^k u,[\rho,\p^sD_t^k]D_tu)\dx$.

Let $D$ be $D_t$ or $\p$, then the commutator can be written as  $$[\rho,\p^sD_t^k]D_tu=\sum_{l=1}^4 C_4^l D^l\rho D^{4-l}D_t u=\sum_{l=1}^4D^{l-1}(\rho'(p)Dp)D^{4-l}D_t u.$$ Therefore we have:

\begin{itemize}
\item $s=4,~k=0:~R_2\lesssim_{K,M,c_0, \volo} \|u\|_{4,0}(\|p\|_{4,0}+\|u\|_{3,1}+\|u\|_{2,1})$;

\item $s=3,~k=1:~R_2\lesssim_{K,M,c_0, \volo} \|u\|_{3,1}(\|p\|_{3,1}+\|p\|_{3,0}+\|u\|_{3,1}+\|u\|_{3,0}+\|u\|_{2,2}+\|u\|_{1,2})$;

\item $s=2,~k=2:~R_2\lesssim_{K,M,c_0, \volo} \|u\|_{2,2}(\|p\|_{2,2}+\|p\|_{2,1}+\|u\|_{0,3}+\|u\|_{1,3}+\|u\|_{1,2}+\|u\|_{2,2}+\|u\|_{2,1})$;

\item $s=1,~k=3:~R_2\lesssim_{K,M,c_0, \volo} \|u\|_{1,3}(\|\wt\p D_t^3 p\|_{\lei}+\|p\|_{2,2}+\|u\|_{1,3}+\|u\|_{0,3}+\|u\|_{1,2})$;

\item $s=0,~k=4:$ we have 
\[
\begin{aligned}
R_2&=\idt\wt D_t^4 u\cdot[\rho,D_t^4]D_tu\dx \\
&\lesssim_{M, c_0} \|\swt D_t^4 u\|_{\lli}(\|\swt D_t^4 u\|_{\lei}+\|\swt D_t^3 u\|_{\lei} \\
&~~~~~~~~~~~~~~~~~~~~~~~~~~~~~~~~~~~~~~~~~~~~~~~+\|\swt D_t^4 p\|_{\lei}+\|\swt D_t^3 p\|_{\lei}).
\end{aligned}
\]

Note that the constant in the equality depends on $\volo$ because we use Poincar\'e's inequality on $p$.
\end{itemize} 

(3) $R_3=-\idt Q(\p^sD_t^k u^i, \p^s([\p_i, D_t^k]P))\dx$.

Recall \eqref{commutator2} the highest order terms in the commutator $[\p, D_t^k]f$ are $(\p D_t^{k-1}u)(\p f)$ and $(\p u)(\p D_t^{k-1}f)$. Hence we can get the following estimates up to lower order terms:
\begin{itemize}
\item $s=4,~k=0:~R_3=0$;

\item $s=3,~k=1:~R_3\lesssim_{M,\volo} \|u\|_{3,1}(\|u\|_{4,0}+\|u\|_{3,0}+\|p\|_{4,0}+\|B\|_{4,0})$;

\item $s=2,~k=2:~R_3\lesssim_{M,\volo} \|u\|_{2,2}(\|u\|_{3,1}+\|u\|_{2,1}+\|p\|_{3,0}+\|B\|_{3,0}+\|u\|_{3,0}+\|p\|_{3,1}+\|B\|_{3,1})$;

\item $s=1,~k=3:~R_3\lesssim_{M,\volo} \|u\|_{1,3}(\|u\|_{2,2}+\|u\|_{2,1}+\|p\|_{2,2}+\|B\|_{2,2}+\|p\|_{2,1}+\|B\|_{2,1})$;

\item $s=0,~k=4$:  We have 
\[
\begin{aligned}
R_3&=\idt \wt D_t^4 u\cdot [\p_i, D_t^4]P\dx\\
&\lesssim_{M, c_0} \|\swt D_t^4 u\|_{\lei}(\|\swt \nabla D_t^3 p\|_{\lei}+\|B\|_{1,3}+\|B\|_{2,2}+\|p\|_{2,2}+\|B\|_{2,1}+\|p\|_{2,1}).
\end{aligned}
\]
\end{itemize}

(4) $R_4=\idt Q(\p^s([\p_i, D_t^k]u^i), \p^sD_t^k P)\dx$.

The commutator term is exactly of the same form as $R_3$ except we replace $P$ by $u^i$. We list the result here and omit the proof.
\begin{itemize}
\item $s=4,~k=0:~R_4=0$;

\item $s=3,~k=1:~R_4\lesssim_{K,M,\volo} (\|u\|_{4,0}+\|u\|_{3,0})(\|p\|_{3,1}+\|B\|_{3,1})$;

\item $s=2,~k=2:~R_4\lesssim_{K,M,\volo} (\|u\|_{3,1}+\|u\|_{3,0})(\|p\|_{2,2}+\|B\|_{2,2})$;

\item $s=1,~k=3:~R_4\lesssim_{K,M,\volo} (\|u\|_{2,2}+\|u\|_{2,1})(\|p\|_{1,3}+\|B\|_{1,3})$;

\item $s=0,~k=4$:  We have 
\[
R_4=\idt \wt [\p_i, D_t^4]u^i\cdot D_t^4 P\dx\lesssim_{M,c_0} (\|u\|_{1,3}+\|u\|_{1,2})(\|\wt D_t^4 p\|_{\lei}+\|B\|_{0,4}).
\]
\end{itemize}

\begin{rmk}
As we can see, the control of $R_4$ when $s=0$ illustrates that \textbf{the weight function is necessary}: If we remove the weight function, then $\|D_t^ 4p\|_{\lli}$ has no control, i.e., either wave equation or $E_{0,4}$ cannot control this term.
\end{rmk}

(5) $R_5=-\idt Q\left([\p^sD_t^k,\wk]D_t p,\p^sD_t^k(\frac{1}{2}|B|^2)\right)\dx$.

Let $D$ be $D_t$ or $\p$, then the commutator can be written as  $$\left[\p^sD_t^k,\wk\right]D_tp=\sum_{l=1}^4 C_4^l D^l\left(\wk\right) D^{4-l}D_t p=\sum_{l=1}^4D^{l}\left(\wk\right)D^{4-l}D_t p.$$

 Therefore we can find that every term is assigned at least $\swt$ weight. We have
\begin{itemize}
\item $s\geq 2:~R_5\lesssim_{K,M,c_,\volo}\|p\|_{s,k}\|B\|_{s,k}$;

\item $s=1,~k=3:~R_5\lesssim_{K,M,c_0,\volo}\|\swt\nabla D_t^3 p\|_{\lei}\|B\|_{1,3}$.

\item $s=0, k=4:$ The weighted estimate is 
\[
\begin{aligned}
R_5&=-\idt\wt \left[D_t^4,\wk\right]D_t p\cdot D_t^4(\frac{1}{2}|B|^2)\dx \\
&\lesssim_{M,c_0}(\|\wt D_t^4 p\|_{\lei}+\|\wt D_t^3 p\|_{\lei})\|B\|_{0,4}.
\end{aligned}
\]
\end{itemize}

(6) $R_6=-\idt Q\left([\p^sD_t^k,\wk]D_t p,\p^sD_t^k p\right)\dx$.

Similarly as $R_5$, we have:
\begin{itemize}
\item $s\geq 2:~R_6\lesssim_{K,M,c_0,\volo}\|p\|_{s,k}^2$;

\item $s=1,~k=3$: $$R_6\lesssim_{K,M,c_0,\volo}\|\swt D_t^4 p\|_{\lei}(\|\swt D_t^4 p\|_{\lei}+\|\swt\nabla D_t^3 p\|_{\lei})\lesssim E_4.$$

\item $s=0,~k=4$: $$R_6=-\idt \wk \left[D_t^4,\wk\right]D_t p\cdot D_t^4 p\dx\lesssim_{M, c_0} \|\swt D_t^4 p\|_{\lei}^2.$$
\end{itemize}

(7) $R_7=-\idt \wk Q(\p^sD_t^k p,[\p^s, D_t]D_t^k p)\dx$.

Since $$[D_t,\p^s]D_t^k p=-\sum_{m=0}^{s+1}C_{s}^{m+1}\p^{1+m}u\symdot\p^{s-m}D_t^k p,$$ we know

\begin{itemize}
\item $s=4,~k=0:~R_7\lesssim_{K,M, c_0, \volo}\|u\|_{4,0}+\|u\|_{3,0}+\|p\|_{4,0}$;

\item $s=3,~k=1:~R_7\lesssim_{K,M, c_0, \volo}\|u\|_{3,0}+\|p\|_{3,1}$;

\item $s=2,~k=2:~R_7\lesssim_{K,M, c_0, \volo}\|p\|_{2,2}$;

\item $s=1,~k=3:~R_7\lesssim_{K,M, c_0, \volo}\|\swt \nabla D_t^3 p\|_{\lli}$;

\item $s=0,~k=4:~R_7=0.$
\end{itemize}

(8) $R_8=\frac{1}{2}\idt\rho D_t(\frac{\rho'(p)}{\rho^2})Q(\p^sD_t^k p, \p^sD_t^k p)\dx\lesssim \|\swt\p^s D_t^k p\|_{\lei}^2\lesssim E_{s,k}$.

(9) $R_9=\ipdt Q(\p^sD_t^k P, \p^sD_t^k P)D_t \nu ~dS$.

\begin{itemize}
\item $s\leq 1:~R_9=0$ because $D_t, \Pi\p\in\mathcal{T}(\p\DD_t)$ and $P=0$ on $\p\DD_t$;

\item $s\geq 2:~R_9\lesssim_{K, M} E_{s,k}$.
\end{itemize}

(10) $R_{10}=\idt Q(\p^sD_t^k B, [D_t,\p^s]D_t^k B)\dx$.

The control of $R_{10}$ is the same as $R_1$ except replacing $u$ by $B$. Therefore we have:

\begin{itemize}
\item $s=4,~k=0:~R_{10}\lesssim_{K,M, c_0,\volo} \|B\|_{4,0}(\|u\|_{4,0}+\|u\|_{3,0}+\|B\|_{4,0})$;

\item $s=3,~k=1:~R_{10}\lesssim_{K,M, c_0,\volo} \|B\|_{3,1}(\|u\|_{3,0}+\|B\|_{3,1})$;

\item $s\leq 2:~R_{10}\lesssim_{K,M, c_0,\volo} \|B\|_{s,k}^2$.
\end{itemize}

(11) $R_{11}=\idt \rho D_t(1/\rho)Q(\p^sD_t^k B, \p^s D_t^k B)\dx\lesssim_{c_0}E_{s,k}$.

(12) $R_{12},\cdots,R_{15}:$ The control of $R_{12},R_{13},R_{15}$ are the same as $R_1, R_{10}, R_2$ respectively when $s=4,~k=0$, and $R_{14}\lesssim K_4$. So we have:
\[
R_{12}+\cdots+R_{15}\lesssim_{K,M, c_0,\volo}K_4+(\|u\|_{4,0}+\|u\|_{3,0}+\|B\|_{4,0}+\|p\|_{4,0})^2.
\]

Before summarising the estimates, we would like to reduce the estimates of $\|u\|_{s,k}$ to that of $B$ and $p$ by using the first equation in \eqref{CMHD}, because we are going to use elliptic estimates for $B$ and $p$ in order to further reduce to control of the wave equation of $p$ and the heat equation of $B$.

We prove the following estimates for $\p^sD_t^k u$ when $1\leq k\leq r-1$, while $\|u\|_{0,r}=\|D_t^r u\|_{\lli}\leq \sqrt{E_{0,4}}$ and $\|u\|_{r,0}$ will be controlled later by div-curl estimates.
\begin{lem}\label{usk}
For $s+k=4$, one has the following bounds:
\[
\begin{aligned}
\|u\|_{3,1}&\lesssim_{K,M, c_0} \|B\|_{4,0}+\|p\|_{4,0},\\
\|u\|_{2,2}&\lesssim_{K,M, c_0}\|B\|_{3,1}+\|p\|_{3,1}+\|B\|_{3,0}+\|p\|_{3,0}+\|u\|_{3,0}, \\
\|u\|_{1,3}&\lesssim_{K,M, c_0} \|B\|_{2,2}+\|p\|_{2,2}+\|B\|_{2,1}+\|p\|_{2,1}+\|B\|_{3,0}+\|p\|_{3,0}.
\end{aligned}
\] 
While for $s+k=r<4$, the result becomes $\|u\|_{s,k}\lesssim_{M,c_0}\|B\|_{s+1,k-1}+\|p\|_{s+1,k-1}$.
\end{lem}

\begin{proof} For simplicity, we only prove it for $s+k=4$. The proof is quite straightforward by the first equation in \eqref{CMHD}. We have
\[
\begin{aligned}
\p^sD_t^k u&=\p^s D_t^{k-1}\left(\frac{1}{\rho}(B\cdot \p B-\p p- (\p B)\cdot B)\right)\\
&=\frac{1}{\rho} \left(B\cdot \p^{s+1}D_t^{k-1} B+\p^{s+1}D_t^{k-1}p\right)+\text{commutators},
\end{aligned}
\]
 The main term can be easily controlled by $C(M)(\|B\|_{s+1,k-1}+\|p\|_{s+1,k-1})$ by H\"older's inequality.

\begin{itemize}

\item $s=3,~k=1$: In this case the commutator term is $\sum_{k=1}^3\p^k(1/\rho)\p^{3-k}(B\cdot \p B-\p P)$ which can be controlled by $\|B\|_{4,0}$ and $\|p\|_{4,0}$ by Poincar\'e's inequality.

When $k\geq 2$. The highest order terms in the commutators consist of $\p^s([D_t^{k-1},\p]B)$, $\p^s([D_t^{k-1},\p]p)$ and $[\p^s D_t^{k-1},\p](1/\rho)$. 

\item $s=2,~k=2$: From the specific representation of $[D_t,\p]=(\p u)\symdot\p$, ee know the highest order commutator terms are $\p^2(\p u\symdot\p p)$ and $\p^2(\p u\symdot\p B)$ which can be bounded by $\|B\|_{3,0}+\|p\|_{3,0}+\|u\|_{3,0}$.

\item $s=1,~k=3$: Similarly as above, one can get the commutator terms bounded by $\|B\|_{2,1}+\|p\|_{2,1}+\|u\|_{2,1}$. Then apply the same method to $\|u\|_{2,1}$ to derived the result.
\end{itemize}
\end{proof}

Combining all the estimates above, we get
{\small\begin{equation}\label{tg}
\begin{aligned}
&~~~~\frac{d}{dt} \left(\sum_{s+k=4} E_{s,k}+K_4\right) \\
&\lesssim_{M,c_0,\volo}\sum_{s+k=4} E_{s,k}+K_4 \\
&~~~~+(\|p\|_{3,1}+\|p\|_{3,0}+\|B\|_{3,1}+\|B\|_{3,0}+\|u\|_{3,0})(\|p\|_{1,3}+\|B\|_{1,3})\\
&~~~~+\left(\sum_{s+k=4}\|B\|_{s,k}+\sum_{s+k=4, s\geq 2}\|p\|_{s,k}+\|\swt \p D_t^3p\|_{\lei}+\|\swt D_t^4 p\|_{\lei} \right)^2\\
&~~~~+\sum_{s+k=4, k\geq 1}\left(\|B\|_{s+1,k-1}+\|p\|_{s+1,k-1}\right)\|(B\cdot\p)B\|_{s,k}\\
&~~~~+\sum_{s+k=4}\|B\|_{s,k}\left(\|B\|_{s,k+1}+\|\wt \p^sD_t^{k+1}p\|_{\lei}\right)\\
&~~~~+(\|u\|_{4,0}+\|B\|_{4,0})\|(B\cdot\p)B\|_{4,0}+\|B\|_{4,0}\|B\|_{4,1} \\
&~~~~+\sum_{s+k=4,2\leq s\leq 3}\bigg(|\Pi\p^s D_t^k P|_{\leb} \Big(|\Pi\p^s D_t^{k+1} P|_{\leb}\\
&~~~~+|\Pi(\p_i P)(\p^sD_t^k u^i)|_{\leb}+\sum_{0\leq m\leq s-1}|\Pi((\p^{m+1}u)\symdot\p^{s-m}D_t^k P)|_{\leb}\Big)\bigg)\\
&~~~~+|\Pi\p^4 P|_{\leb}\Big(|\Pi \p^4 D_t P|_{\leb}+\sum_{0\leq m \leq 2}|\Pi((\p^{m+1}u)\symdot \p^{4-m}P)|_{\leb}\Big).
\end{aligned}
\end{equation}}

Similar estimate holds for $s+k=r\leq 3$.
{\small\begin{equation}\label{tglow}
\begin{aligned}
&\frac{d}{dt} \left(\sum_{s+k=r} E_{s,k}+K_r\right)\\
&\lesssim_{M,c_0,\volo}\sum_{s+k=r} E_{s,k}+K_r \\
&~~~~~~~~+\left(\sum_{s+k=r}\|B\|_{s,k}+\sum_{s+k=r, s\geq 2}\|p\|_{s,k}+\|p\|_{2,2}+\|\swt D_t^r p\|_{\lei} \right)^2\\
&~~~~~~~~+\left(\sum_{s+k=4, s\geq 1}\|B\|_{s,k}+\sum_{s+k=r, s\geq 2}\|p\|_{s,k}+\|\swt \p D_t^3 p\|_{\lei} \right)^2\\
&~~~~~~~~+\sum_{s+k=r,s\geq 2}\sum_{k+s=r,s\geq 2}\bigg(|\Pi\p^s D_t^k P|_{\leb} \Big(|\Pi\p^s D_t^{k+1} P|_{\leb}\\
&~~~~+|\Pi(\p_i P)(\p^sD_t^k u^i)|_{\leb}+\sum_{0\leq m\leq s-1}|\Pi((\p^{m+1}u)\symdot\p^{s-m}D_t^k P)|_{\leb}\Big)\bigg)\\
&~~~~+|\Pi\p^r P|_{\leb}\Big(|\Pi \p^r D_t P|_{\leb}+\sum_{0\leq m \leq r-2}|\Pi((\p^{m+1}u)\symdot \p^{r-m}P)|_{\leb}\Big).
\end{aligned}
\end{equation}}

\bigskip

\section{Control of interior and boundary terms of top order}\label{section5}

Now we come back to use Lagrangian coordinate. With a little abuse of terminology, we still define

\begin{itemize}
\item $\|f\|_{s,k} = \|\nabla^sD_t^k f\|_{\lli}$,
\item $|f|_{s,k}= |\nabla^sD_t^k f|_{\llb}$. 
\end{itemize} As stated in Section 2, our proof is coordinate-independent.

We are going to use elliptic estimates in Section \ref{section3} to reduce the interior terms in \eqref{tg} and \eqref{tglow}.

\subsection{Div-curl estimates: Full spatial derivatives of $u$ and $B$}
By the Hodge's decomposition inequality, we have 
\[
\|u\|_{r,0}\lesssim\|u\|_{0,0}+\|\dive \nabla^{r-1}u\|_{\lli}+\underbrace{\|\curl\nabla^{r-1}u\|_{\lli}+\frac{1}{2}\idt \rho Q(\p^r u,\p^r u)\dx}_{\lesssim\sqrt{K_r}+\sqrt{E_{r,0}}}.
\]and
\[
\|B\|_{r,0}\lesssim\|B\|_{0,0}+\|\underbrace{\dive\nabla^{r-1}B}_{=0}\|_{\lli}+\underbrace{\|\curl\nabla^{r-1}B\|_{\lli}+\frac{1}{2}\idt Q(\p^r B, \p^r B)\dx}_{\lesssim\sqrt{K_r}+\sqrt{E_{r,0}}}.
\]

Now we use $\dive u=-\wk D_t p$ to control $\|\dive \nabla^{r-1}u\|_{\lli}$:

\[
\dive \nabla^{r-1}u=\nabla^{r-1}\dive u=-\nabla^{r-1}\left(\wk D_t p\right)=-\wk\nabla^{r-1}D_t p-\left[\nabla^{r-1}, \wk\right]D_t p.
\]

Hence, 
\[
\|\dive \nabla^{r-1}u\|_{\lli}\lesssim_{M}\|\swt\nabla^{r-1}D_t p\|_{\lli}+\|p\|_{r-1,0}\lesssim \sqrt{E_{r-1}}+\|p\|_{r-1,0},
\]and thus
\begin{equation}\label{ub40}
\|u\|_{r,0}+\|B\|_{r,0}\lesssim_{M}\sqrt{E_0}+\sqrt{E_r}+\sqrt{E_{r-1}}+\|p\|_{r-1,0}.
\end{equation}

\subsection{Elliptic estimates: Control of $\|B\|_{s,k}$ and $\|p\|_{s,k}$}

In this part we try to control $\|B\|_{s,k}$ and $\|p\|_{s,k}$ by using the elliptic estimates in Section 3. The only exception is $\|p\|_{1,3}$ because it has no weight function $\swt$ and thus it cannot be bounded, independently of the lower bound of $\wt$ (this lower bound goes to 0 when passing to the incompressible limit), by the terms in our proposed energy \eqref{Er}. This term will be controlled by $W_5$ after using Poincar\'e's inequality. For simplicity we only consider the top order case: $s+k=4$. 

\subsubsection*{When $s\geq 2$}

\begin{itemize}
\item $s=4,~k=0:$

By the elliptic estimates, we know $\forall\delta>0$, we have
\[
\|p\|_{4,0}:=\|\nabla^4 p\|_{\lli}\lesssim_{K,M,\volo} \delta\sum_{s\leq 4}|\Pi\nabla^s p|_{\llb}+\delta^{-1}\sum_{j\leq 2}\|\nabla^j\Delta p\|_{\lli}.
\]\label{p400}

Using the boundary tensor estimates, we have
\[
|\Pi\nabla^s p|_{\llb}\lesssim_{K,\volo}|\cnab^{s-2}\theta|_{\llb}|\nabla_N p|_{\linf}+\sum_{l=1}^{s-1}|\nabla^l p|_{\llb}.
\]

Using trace lemma and the estimates of $\|B\|_{4,0}$, we can control the second fundamental form as follows:
\begin{equation}\label{22nd}
\begin{aligned}
|\cnab^2\theta|_{\llb}&\lesssim_{K,1/\epsilon_0}|\Pi\nabla^4 P|_{\llb}+\sum_{l=1}^3 |\nabla^l P|_{\llb}\\
&\lesssim_{K,M,1/\epsilon_0} \sqrt{E_0}+\sqrt{E_4}+\sqrt{E_3}+\|p\|_{4,0}.
\end{aligned}
\end{equation}

By trace lemma and Sobolev embedding, one has $|\nabla_N p|_{\linf}\lesssim_{\volo}\|p\|_{4,0}$. Combining with the estimate above, one can pick a suitably small $\delta>0$ such that $\delta\sum_{s\leq 4}|\Pi\nabla^s p|_{\llb}$ is absorbed by LHS of \eqref{p400}, i.e.,

\begin{equation}\label{p401}
\|p\|_{4,0}\lesssim_{K,M,\volo,1/\epsilon_0}\sqrt{E_4^*}+\sum_{j\leq 2}\|\nabla^j\Delta p\|_{\lli}.
\end{equation}

\item $s=3,~k=1:$ 

Similarly as above, we first use the elliptic estimates to get $\forall\delta>0$
\[
\begin{aligned}
\|B\|_{3,1}&=\|\nabla^3 D_tB\|_{\lli}\lesssim_{K,M,\volo}\delta|\Pi \nabla^3 D_t B|_{\llb}+\delta^{-1}\sum_{j\leq 1}\|\nabla^j \Delta D_t B\|_{\lli} \\
&\lesssim_{K,M,\volo}\delta\left(|\cnab\theta|_{L^4(\p\Omega)}|\nabla_N D_t B|_{L^4(\p\Omega)}+\sum_{l=1}^{2}|\nabla^l D_t B|_{\llb}\right)\\
&~~~~~~~~~~~~~~~~+\delta^{-1}\sum_{j\leq 1}\|\nabla^j \Delta D_t B\|_{\lli} \\
&\lesssim_{K,M,\volo}\delta\bigg(|\cnab\theta|_{H^1(\p\Omega)}^{1/2}|\nabla_N D_t B|_{H^1(\p\Omega)}^{1/2}|\cnab\theta|_{L^2(\p\Omega)}^{1/2}|\nabla_N D_t B|_{L^2(\p\Omega)}^{1/2} \\
&~~~~~~~~~~~~~~~~~~~~~+\sum_{l=1}^{2}|\nabla^l D_t B|_{\llb}\bigg)+\delta^{-1}\sum_{j\leq 1}\|\nabla^j \Delta D_t B\|_{\lli},
\end{aligned}
\]where we use the Sobolev interpolation Theorem \ref{gag-ni thm} in the last step. 

By tensor estimates, one can get 
\begin{equation}\label{12nd}
\begin{aligned}
|\cnab \theta|_{\llb}&\lesssim_{K,1/\epsilon_0}|\Pi\nabla^3 P|_{\llb}+\sum_{l=1}^2 |\nabla^l P|_{\llb}\\
&\lesssim_{K,M,1/\epsilon_0} \sqrt{E_3^*}+\|p\|_{3,0}.
\end{aligned}
\end{equation}
Therefore, using Sobolev trace lemma, \eqref{22nd}, \eqref{p401}, \eqref{12nd} and Poincar\'e's inequality one has 
\[
\begin{aligned}
\|B\|_{3,1}&\lesssim_{K,M,1/\epsilon_0,\volo}\delta \left(\sqrt{E_3^*}+\sum_{j\leq 1}\|\nabla^j\Delta p\|_{\lli}\right)^{1/2}\|B\|_{2,1}^{1/2}\cdot\left(\sqrt{E_4^*}+\sum_{j\leq 2}\|\nabla^j \Delta p\|_{\lli}+\|B\|_{3,1}\right) \\
&~~~~~~~~~~~~~~~~+\delta\|B\|_{3,1}+\delta^{-1}\sum_{j\leq 1}\|\nabla^j \Delta D_t B\|_{\lli}.
\end{aligned}
\]
If we choose $\delta>0$ to be suitbaly small, then $\delta\|B\|_{3,1}$ will be absorbed to LHSof the last inequality,and thus we have
\begin{equation}\label{B311}
\begin{aligned}
\|B\|_{3,1}\lesssim_{K,M,1/\epsilon_0,\volo}\delta &\left(\sqrt{E_3^*}+\sum_{j\leq 1}\|\nabla^j\Delta p\|_{\lli}\right)^{1/2}\|B\|_{2,1}^{1/2}\\
&\times\left(\sqrt{E_4^*}+\sum_{j\leq 2}\|\nabla^j \Delta p\|_{\lli}+\|B\|_{3,1}\right)\\
&~~~~\sum_{j\leq 1}\|\nabla^j \Delta D_t B\|_{\lli}
\end{aligned}
\end{equation}for sufficiently small $\delta>0$.

Replace $B$ by $p$ in \eqref{B311}, we can get the estimates of $\|p\|_{3,1}$:
\begin{equation}\label{p311}
\begin{aligned}
\|p\|_{3,1}&\lesssim_{K,M,1/\epsilon_0,\volo}\delta \left(\sqrt{E_3^*}+\sum_{j\leq 1}\|\nabla^j\Delta p\|_{\lli}\right)\|p\|_{2,1}^{1/2}\cdot\left(\sqrt{E_4^*}+\sum_{j\leq 2}\|\nabla^j \Delta p\|_{\lli}+\|p\|_{3,1}\right) \\
&~~~~~~~~~~~~~~~~+\sum_{j\leq 1}\|\nabla^j \Delta D_t p\|_{\lli}
\end{aligned}
\end{equation}for sufficiently small $\delta>0$.

\item $s=2,~k=2$:

Similarly as above, one can get the following estimates by elliptic estimate:
\[
\begin{aligned}
\|B\|_{2,2}&=\|\nabla^2 D_t^2B\|_{\lli}\lesssim_{K,M,\volo}\delta|\Pi \nabla^2 D_t^2 B|_{\llb}+\delta^{-1}\|\Delta D_t^2 B\|_{\lli} \\
&\lesssim_{K,M,\volo}\delta\left(|\theta|_{\linf}|\nabla_N D_t^2 B|_{L^2(\p\Omega)}+|\nabla D_t^2 B|_{\llb}\right)\\
&~~~~~~~~~~~~~~~~+\delta^{-1}\|\Delta D_t^2 B\|_{\lli} \\
&\lesssim_{K,M,\volo}\delta\|B\|_{2,2}+\delta^{-1}\|\Delta D_t^2 B\|_{\lli},
\end{aligned}
\]where the last step we use the a priori assumption $|\theta|\leq K$ and Sobolev trace lemmma. Now choosing $\delta>0$ suitably small so that the $\delta$-term can be absorbed by LHS, one gets
\begin{equation}\label{B221}
\|B\|_{2,2}\lesssim_{K,M,\volo}\|\Delta D_t^2 B\|_{\lli}.
\end{equation}

Also one can get 
\begin{equation}\label{p221}
\|p\|_{2,2}\lesssim_{K,M,\volo}\|\Delta D_t^2 p\|_{\lli}.
\end{equation}
\end{itemize}

\subsubsection*{When $s\leq 1$}

We already know $\|B\|_{0,4}=\|D_t^4 B\|_{\lli}$ is a part of $\sqrt{E_{0,4}}$ and $\|B\|_{1,3}=\|\nabla D_t^3B\|_{\lli}$ is a part of the parabolic equation energy $H_4$. From \eqref{tg} and \eqref{tglow} we know there must be a weight function $\swt$ or $\wt$ multiplying on $D_t^4 p$ as long as $D_t^4 p$ appears, and thus can also be controlled by either $\sqrt{E_{0,4}}$ or $W_4$.

The only term we need to do extra work is $\|p\|_{1,3}$, because in our imposed energy function, all the terms that can control $\nabla D_t^3 p$ contain a weight function $\wt$ or $\swt$. Hence, one cannot get the uniform control with respect to the sound speed $c:=\sqrt{p'(\rho)}$ as it goes to infinity when passing to the incompressible limit. 

To avoid this problem, we use Poincar\'e's inequality to get 
\[
\|p\|_{1,3}=\|\nabla D_t^3 p\|_{\lli}\lesssim_{\volo}\|\nabla^2 D_t^3 p\|_{\lli}=\|p\|_{2,3}.
\] In other words, we make it to be a higher order term of the form $\|p\|_{s,k+1}$(recall $s+k=4$), which can be reduced to the control of $5$-th order wave equation. We will deal with these terms in the next section.

\subsection{Elliptic estimates: Reduction of higher order terms}

So far, what remained to be controlled are of the form $\|(B\cdot\nabla)B\|_{s,k},\|p\|_{s,k+1},~\|B\|_{s,k+1}$, tangential projections $|\Pi\nabla^s D_t^{k+1}P|_{\llb}$ and the wave equation of $p$ coupled with the parabolic equation of $B$ when $s+k=4$. In this section, we will reduce all the control of $\|p\|_{s,k+1},~\|B\|_{s,k+1}$ and $|\Pi\nabla^s D_t^{k+1}P|_{\llb}$ to that of wave equation and parabolic equation.

First we would like to control those interior higher order terms. In fact we cannot control these terms directly. Instead, we need to control $\|(B\cdot \nabla)B\|_{s,k}$, $\|B\|_{s,k+1}=\|D_t B\|_{s,k}$ together with $|B|_{s,k}+|B|_{s-1, k+1}$; $\|p\|_{s,k+1}=\|D_tp\|_{s,k}$ together with $|p|_{s-1,k+1}$ if $s\geq 2$, so that we can use Young's inequality to absorb the higher order terms. While for $s\leq 1$, weight functions must appear as long as all these terms containing $p$ appear in the previous estimates.

\begin{itemize}
\item $s=4,~k=0:$ We consider $$\|\nabla^4((B\cdot\nabla )B)\|_{\lli}+\|\nabla^4 D_t B\|_{\lli}+|\nabla^4 B|_{\llb}+|\nabla^3 D_t B|_{\llb}.$$ 

Since $(B\cdot\nabla)B=0$ on $\p\Omega$, by elliptic estimates, we have $\forall\delta>0$:
\[
\begin{aligned}
&~~~~\|\nabla^4((B\cdot\nabla)B)\|_{\lli}+\|\nabla^4 D_t B\|_{\lli}+|\nabla^4 B|_{\llb}+|\nabla^3 D_t B|_{\llb}\\
&\lesssim_{K,M,\volo}\delta\left(|\Pi\nabla^4(B\cdot\nabla)B|_{\llb}+|\Pi\nabla D_t^4 B|_{\llb}\right)\\
&~~~~~~~~~~~~~+\delta^{-1}\sum_{j\leq 2}\left(\|\nabla^j\Delta(B\cdot\nabla B)\|_{\lli}+\|\nabla^j\Delta D_t B\|_{\lli}\right)\\
&\lesssim_{K,M,\volo}\delta\bigg(|\cnab^2\theta|_{\llb}(|\nabla_N (B\cdot\nabla B)|_{\linf}+|\nabla_N D_t B|_{\linf}) \\
&~~~~~~~~~~~~+\sum_{l=1}^{3}|\nabla^l(B\cdot\nabla B)|_{\llb}+|\nabla^l D_t B|_{\llb}\bigg)\\
&~~~~~~~~~~~~~+\delta^{-1}\sum_{j\leq 2}\left(\|\nabla^j\Delta(B\cdot\nabla B)\|_{\lli}+\|\nabla^j\Delta D_t B\|_{\lli}\right).
\end{aligned}
\]

Using Sobolev trace lemma and Poincar\'e's inequality, we know 
\[
\begin{aligned}
&~~~~|\nabla_N (B\cdot\nabla B)|_{\linf}+|\nabla_N D_t B|_{\linf}+\sum_{l=1}^{3}|\nabla^l(B\cdot\nabla B)|_{\llb}+|\nabla^l D_t B|_{\llb}\\
&\lesssim_{K,M,\volo}\|\nabla^4((B\cdot\nabla)B)\|_{\lli}+\|\nabla^4 D_t B\|_{\lli}+|\nabla^4 B|_{\llb}+|\nabla^3 D_t B|_{\llb},
\end{aligned}
\] and thus these $\delta$-terms can be absorbed by LHS of last inequality if we choose a suitably small $\delta>0$ ,i.e.,

\begin{equation}
\begin{aligned}\label{B411}
&~~~~\|\nabla^4((B\cdot\nabla)B)\|_{\lli}+\|\nabla^4 D_t B\|_{\lli}+|\nabla^4 B|_{\llb}+|\nabla^3 D_t B|_{\llb}\\
&\lesssim_{K,M,1/\epsilon_0,\volo}\delta\underbrace{\left(\sqrt{E_4^*}+\sum_{j\leq 2}\|\nabla^j\Delta p\|_{\lli}\right)}_{\text{estimates of } |\cnab^2\theta|_{\llb}}\left(\|\nabla^4((B\cdot\nabla)B)\|_{\lli}+\|\nabla^4 D_t B\|_{\lli}\right) \\
&~~~~~~~~~~~~~+\sum_{j\leq 2}\left(\|\nabla^j\Delta(B\cdot\nabla B)\|_{\lli}+\|\nabla^j\Delta D_t B\|_{\lli}\right).
\end{aligned}
\end{equation} holds for sufficiently small $\delta>0$.

One can mimic the steps above to get a similar estimate on $\|p\|_{4,1}+|p|_{3,1}$:

\begin{equation}\label{p411}
\begin{aligned}
&~~~~\|\nabla^4 D_t p\|_{\lli}+|\nabla^3 D_t p|_{\llb}\\
&\lesssim_{K,M,\volo}\delta\left(\sqrt{E_4^*}+\sum_{j\leq 2}\|\nabla^j\Delta p\|_{\lli}\right)\|\nabla^4 D_t p\|_{\lli}+\sum_{j\leq 2}\|\nabla^j\Delta D_t p\|_{\lli}.
\end{aligned}
\end{equation} holds for sufficiently small $\delta>0$.

\begin{rmk}
When $k>0$, the estimates of $\|(B\cdot\nabla)B\|_{s,k}$ can be reduced to that of $\|B\|_{s+1,k}$ plus $\|B\|_{s+1,k-1}$ together with $\|u\|_{s+1,k-1}$, while the latter two terms have been controlled above.

\[
\nabla^sD_t^k(B\cdot \nabla)B=(B\cdot\nabla)\nabla^{s}D_t^k B+\nabla^s\left[D_t^k, B\cdot \nabla\right]B +[\nabla^s,B\cdot\nabla] D_t^k B,
\] in which the commutator terms consist of $\leq 4$ derivatives of $B$ or $u$ multiplying the a priori quantities by Leibniz rule and \eqref{commutator3}. One has
\[
\|(B\cdot\nabla)B\|_{s,k}\lesssim\|B\|_{s+1,k}+\left(\|B\|_{s+1,k-1}+\|u\|_{s+1,k-1}\right).
\] Therefore, it suffices to consider $\|B\|_{s,k+1}$ in the rest of this part.
\end{rmk}

\item $s=3,~k=1:$ Using elliptic estimates, tensor estimates for the tangential projection and Sobolev interpolation Theorem \ref{gag-ni thm}, we get: $\forall\delta>0$,
\[
\begin{aligned}
&~~~~\|\nabla^3 D_t^2 B\|_{\lli}+|\nabla^2 D_t^2B|_{\llb} \\
&\lesssim_{K,M,\volo}\delta\left(|\cnab\theta|_{L^4(\p\Omega)}|\nabla_N D_t^2 B|_{L^4(\p\Omega)}+\sum_{l=1}^{2}|\nabla^l D_t^2 B|_{\llb}\right)+\delta^{-1}\sum_{j\leq 1}\|\nabla^j\Delta D_t^2 B\|_{\lli} \\
&\lesssim_{K,M,\volo}\delta\left(|\cnab\theta|_{H^1(\p\Omega)}^{1/2}|\nabla_N D_t^2 B|_{H^1(\p\Omega)}^{1/2}|\cnab\theta|_{L^2(\p\Omega)}^{1/2}|\nabla_N D_t^2 B|_{L^2(\p\Omega)}^{1/2}+\sum_{l=1}^{2}|\nabla^l D_t^2 B|_{\llb}\right) \\
&~~~~~~~~~~~~~+\delta^{-1}\sum_{j\leq 1}\|\nabla^j\Delta D_t^2 B\|_{\lli}.
\end{aligned}
\]
Using Sobolev trace lemma and Poincar\'e's inequality, it holds that $$|\nabla_N D_t^2 B|_{H^1(\p\Omega)}+\sum_{l=1}^{2}|\nabla^l D_t^2 B|_{\llb}\lesssim_{\volo} \|\nabla^3 D_t^2 B\|_{\lli}+|\nabla^2 D_t^2B|_{\llb}.$$ Hence, one can choose a suitably small delta $\delta>0$ to abosrb these $\delta$-terms to LHS. Combining with the estimates of $\theta$ \eqref{22nd} and \eqref{12nd}, we have
\begin{equation}\label{B321}
\begin{aligned}
&~~~~\|\nabla^3 D_t^2 B\|_{\lli}+|\nabla^2 D_t^2B|_{\llb}\\
&\lesssim_{K, M, 1/\epsilon_0,\volo}\delta \left(\sqrt{E_3^*}+\sum_{j\leq 1}\|\nabla^j\Delta p\|_{\lli}\right)^{1/2}|B|_{1,2}^{1/2}\cdot\left(\sqrt{E_4^*}+\sum_{j\leq 2}\|\nabla^j \Delta p\|_{\lli}+|B|_{2,2}\right) \\
&~~~~~~~~~~~~+\sum_{j\leq 1}\|\nabla^j\Delta D_t^2 B\|_{\lli},
\end{aligned}
\end{equation}for sufficiently small $\delta>0$.
 Similarly we have the same type estimate on $p$:
\begin{equation}\label{p321}
\begin{aligned}
&~~~~\|\nabla^3 D_t^2 p\|_{\lli}+|\nabla^2 D_t^2p|_{\llb}\\
&\lesssim_{K, M, 1/\epsilon_0, \volo}\delta \left(\sqrt{E_3^*}+\sum_{j\leq 1}\|\nabla^j\Delta p\|_{\lli}\right)^{1/2}|p|_{1,2}^{1/2}\cdot\left(\sqrt{E_4^*}+\sum_{j\leq 2}\|\nabla^j \Delta p\|_{\lli}+|p|_{2,2}\right) \\
&~~~~~~~~~~~~+\sum_{j\leq 1}\|\nabla^j\Delta D_t^2 p\|_{\lli},
\end{aligned}
\end{equation} holds for sufficiently small $\delta>0$.

\item $s=2,~k=2:$ Since $|\theta|\leq K$ is part of the a priori assumption, then one can mimic the proof above to get $\forall\delta>0$
\[
\begin{aligned}
&~~~~\|\nabla^2 D_t^3 B\|_{\lli}+|\nabla D_t^3 B|_{\llb}\\
&\lesssim_{K,M,\volo}\delta \left(|\theta|_{\linf}|\nabla_N D_t^3 B|_{\llb}+|\nabla D_t^3 B|_{\llb}\right)+\delta^{-1}\|\Delta D_t^3 B\|_{\lli}\\
&\lesssim_{K,M,\volo}\delta (\|\nabla^2 D_t^3 B\|_{\lli}+|\nabla D_t^3 B|_{\llb})+\delta^{-1}\|\Delta D_t^3 B\|_{\lli}.
\end{aligned}
\]
 Choosing $\delta>0$ suitably small to absorb the $\delta$-term, one gets
\begin{equation}\label{B23}
\|\nabla^2 D_t^3 B\|_{\lli}+|\nabla D_t^3 B|_{\llb}\lesssim_{K,M,\volo}\|\Delta D_t^3 B\|_{\lli},
\end{equation}as well as the version of $p$
\begin{equation}\label{p23}
\|\nabla^2 D_t^3 p\|_{\lli}+|\nabla D_t^3 p|_{\llb}\lesssim_{K,M,\volo}\|\Delta D_t^3 p\|_{\lli}.
\end{equation}

\item $s\leq 1:$ From the previous estimates, we know such terms must appear together with a weight function $\swt$ or $\wt$ (e.g., see \eqref{tg}). Therefore they can be directly controlled by the imposed energy function:
\begin{equation}\label{Bp14}
\|\wt \nabla D_t^4 p\|_{\lli}+\|\swt D_t^ 5 p\|_{\lli}+\|\nabla D_t^4 B\|_{\lli}\lesssim_{c_0} \sqrt{E_4}.
\end{equation}
For $D_t^5 B$, it only appears once in the term $\|B\|_{0,4}\|B\|_{0,5}$ in \eqref{tg}. We can control its time integral because it is still a part of $E_4$:
\[
\begin{aligned}
&~~~~~\int_0^T \|D_t^4 B(t)\|_{\lli}\|D_t^5 B(t)\|_{\lli}\dt \\
&\leq \delta \int_0^T\|D_t^5 B(t)\|_{\lli}^2 \dt+\frac{1}{4\delta}\int_0^T\|D_t^4 B(t)\|_{\lli}^2 \dt\\
&=\delta H_5^2(T)+\frac{1}{4\delta}H_4^2(T),
\end{aligned}
\] where one can pick $\delta>0$ sufficiently small to \textbf{absorb this term in the final estimates} of $E_4$.
\end{itemize}

Apart from the tangential projection terms, we have reduced all the other terms in \eqref{tg} to the control of $\|\nabla^{s-2}\Delta D_t^{k}B\|_{\lli}$, $\|\nabla^{s-2}\Delta D_t^{k}p\|_{\lli}$, $\|\nabla^{s-2}\Delta D_t^{k+1}B\|_{\lli}$ and $\|\nabla^{s-2}\Delta D_t^{k+1}p\|_{\lli}$ for $s\geq 2$, which will be controlled through the 4th and 5th order wave equation of $p$ and the parabolic equation of $B$. Those tangential projections will be bounded after we control $r$-th order wave equation.

\section{Estimates of wave and heat equation of $\leq 4$ order}\label{section6}

In this section we are going to give a common control for $W_{r+1}^2+H_{r+1}^2$, which is the only thing left to close the a priori bound. We will first control the energy of 3rd and 4th order wave/heat equation in order to bound interior terms and tangential projections by $E_4^*$.

Recall the heat equation of $B$ is 
\begin{equation}\label{heat0}
D_t B-\lambda\Delta B=B\cdot\nabla u-B\dive u=B\cdot u+B\wk D_t p.
\end{equation}

Taking divergence of the first equation of MHD system \eqref{CMHD}, then commuting $\nabla_i$ with $\rho D_t$, one has
\[
\rho D_t \dive u-\nabla_i(B^k\nabla_k B_i)+\Delta\left( \frac{1}{2}|B|^2\right)=-\Delta p+\left[\rho D_t,\nabla_i\right] u^i.
\]

Plugging the continuity equation, $\dive B=0$ and $D\rho=\rho'(p)Dp$ ($D=\nabla$ or $D_t$) into the last equation, one gets the wave equation of $p$:
\begin{equation}\label{wave0}
\rho'(p)D_t^2 p-\Delta p=B^k\Delta B_k+w,
\end{equation}where 
\begin{equation}\label{w}
w=\left(\wk-\rho''(p)\right)(D_tp)^2+\wk\nabla_i p\left((B\cdot\nabla B_i)-\nabla_i P\right)+\rho\nabla_i u^k\nabla_k u^i-\nabla^i B_k\nabla_k B^i +|\nabla B|^2.
\end{equation}

\begin{rmk}
The derivation of \eqref{w} is: The first term $\left(\wk-\rho''(p)\right)(D_tp)^2$ comes from $D_t\dive u=D_t(-\wk D_t p)$. The second and the third term come from $\left[\rho D_t,\nabla_i\right] u^i$. The term $\nabla^i B_k\nabla_k B^i$ comes from $\nabla_i(B^k\nabla_k B_i)$ and $\dive B=0$. The last term appears because $\Delta\left( \frac{1}{2}|B|^2\right)=B\cdot \Delta B+|\nabla B|^2$.
\end{rmk}

\subsection{Higher order equations: Reduction of $\nabla^{s-2}\Delta D_t^{k+1}p$ and $\nabla^{s-2}\Delta D_t^{k+1}B$}

Now we are going to derive the higher order heat/wave equation. Taking $D_t^k$ on the heat equation, one gets
\begin{equation}\label{heatk+1}
\begin{aligned}
D_t^{k+1}B-\lambda\Delta D_t^k B&=\lambda [D_t^k,\Delta]B+(B\cdot\nabla)D_t^k u+B\wk D_t^{k+1}p\\
&~~~~+[D_t^k, B\cdot\nabla]u+\left[D_t^k, B\wk\right]D_t p\\
&=:h_{k+1}^{\lambda}+h_{k+1}+\hh_{k+1},
\end{aligned}
\end{equation}
where
\begin{equation}\label{hk+1}
\begin{aligned}
\hl_{k+1}&:=\lambda [D_t^k,\Delta]B, \\
h_{k+1}&:=(B\cdot\nabla)D_t^k u+B\wk D_t^{k+1}p,\\
\hh_{k+1}&:=[D_t^k, B\cdot\nabla]u+\left[D_t^k, B\wk\right]D_t p.
\end{aligned}
\end{equation}

Similarly, taking $D_t^k$ on the wave equation, one gets
\begin{equation}\label{wavek+10}
\begin{aligned}
\rho'(p) D_t^{k+2} p-\Delta D_t^k p&=[D_t^k, \Delta]p+ B D_t^k \Delta B+[D_t^k,B^l]\Delta B_l \\
&~~~~+D_t^k w+\wh_{k+1},
\end{aligned}
\end{equation}where
\begin{equation}
\wh_{k+1}=\sum_{i_1+\cdots+i_m=k+2,~1\leq i_l\leq k+1}\rho^{(m)}(p)(D_t^{i_1}p)\cdots(D_t^{i_l}p).
\end{equation}

Recall from \eqref{heat0} and \eqref{heatk+1} that $D_t^k\Delta B=\lambda^{-1}(D_t^{k+1}B-h_{k+1}-\hh_{k+1})$. We can rewrite the $(k+1)$-th order wave equation as 
\begin{equation}\label{wavek+1}
\wt D_t^{k+2}p-\Delta D_t^k p=w_{k+1}+\wh_{k+1}+\wl_{k+1},
\end{equation}where
\begin{equation}\label{wk+1}
\begin{aligned}
w_{k+1}&=D_t^kw+[D_t^k,\Delta]p, \\
\wh_{k+1}&\text{ defined as above}, \\
\wl_{k+1}&=\lambda^{-1}\left(B\cdot D_t^{k+1}B-B\cdot h_{k+1}-B\cdot\hh_{k+1}+[D_t^k, B^l]\big(D_t B_l -(B\cdot \nabla)u_l-B_l\wk D_t p\big)\right).
\end{aligned}
\end{equation}

From the precise form of the commutators \eqref{commutator4}, we know all the terms onthe RHS of \eqref{hk+1} and \eqref{wk+1} are of $\leq k+1$ derivatives.

\subsection{Energy estimates for $W_3$ and $H_3$: Reduced to the a priori quantities}\label{whe3}

We first give the control for $3^{rd}$ order wave/heat equation. This can give us the control of $\|\Delta D_t p\|_{\lli},~\|\Delta D_t B\|_{\lli}$ and $\|\nabla\Delta p\|_{\lli},~\|\nabla\Delta B\|_{\lli}$ which helps us close the estimates for the terms with 3 derivatives, i.e., $\|u\|_{3,0},\|p\|_{3,0},\|p\|_{2,1}$ and $\|B\|_{2,1}$.

Let $k=2$ in \eqref{heatk+1} and \eqref{wavek+1}, and then we have 
\begin{equation}\label{heatlow3}
\begin{aligned}
\Delta D_t B&=\lambda^{-1}(D_t^2 B-h_2-\hh_2-\hl_2), \\
\nabla\Delta B&=\lambda^{-1}\nabla(D_t B-(B\cdot\nabla B)-B\wk D_t p).
\end{aligned}
\end{equation} as  well as
\begin{equation}\label{wavelow3}
\begin{aligned}
\Delta D_t p&=\wt D_t^3 p-\wl_2-w_2-\wh_2, \\
\nabla\Delta p&=\nabla(\wt D_t^2 p-\wl_1-w_1-\wh_1).
\end{aligned}
\end{equation}

Therefore one has 
\begin{equation}\label{w3h30}
\begin{aligned}
\|\Delta D_t B\|_{\lli}&\leq\lambda^{-1}(\|D_t^2 B\|_{\lli}-\|h_2\|_{\lli}-\|\hh_2\|_{\lli}-\|\hl_2\|_{\lli}), \\
\|\nabla \Delta B\|_{\lli}&\lesssim_M \lambda^{-1}(\|\nabla D_t B\|_{\lli}+\|\nabla(B\cdot \nabla B)\|_{\lli}+\|\wt \nabla D_t p\|_{\lli}), \\
\|\Delta D_t p\|_{\lli}&\leq \|\wt D_t^3 p\|_{\lli}+\|w_2\|_{\lli}+\|\wh_2\|_{\lli}+\|\wl_2\|_{\lli} \\
\|\nabla \Delta p\|_{\lli}&\leq \|\wt \nabla D_t^2 p\|_{\lli}+\|\nabla w_1\|_{\lli}+\|\nabla \wh_1\|_{\lli}+\|\nabla \wl_1\|_{\lli}.
\end{aligned}
\end{equation}

We notice that all the terms except $\|\wt D_t^3 p\|_{\lli}$ and $\|\wt \nabla D_t^2 p\|_{\lli}$ on the RHS of \eqref{w3h30} are of $\leq 2$ derivatives and thus are our a priori assumed quantities. Therefore, we have
\begin{equation}\label{w3h3}
\|\Delta D_t B\|_{\lli}+\|\nabla \Delta B\|_{\lli}+\|\Delta D_t p\|_{\lli}+\|\nabla \Delta p\|_{\lli}\lesssim_{M,c_0} \frac{1}{\lambda}(1+W_3)\leq \frac{1}{\lambda}(1+\sqrt{E_2^*}).
\end{equation}

Combining with the results in the last section, we actually have that 
\begin{equation}\label{sk3}
\sum_{s+k=3,s\geq 2}\|p\|_{s,k}+\sum_{s+k=3}\|B\|_{s,k}+\|\swt \nabla D_t p\|_{\lli}+\|\wt D_t^3 p\|_{\lli}\lesssim_{K,M,c_0,\volo,1/\epsilon_0,1/\lambda} 1+\sqrt{E_3^*}.
\end{equation}

\subsection{Energy estimates for $W_4$ and $H_4$: Close the estimates for 4-th order derivatives}

The computation in the previous section shows that we need to bound $$\sum_{j=0}^2\|\nabla^{j}\Delta D_t^{2-j} p\|_{\lli}+\|\nabla^{j}\Delta D_t^{2-j} B\|_{\lli}$$ by $\sqrt{E_4^*}$ in order to give a common control for those terms with $\leq 4$ derivatives, i.e., $\|u\|_{r,0}, \|B\|_{s,k}$ and $\|p\|_{s,k}$ for $s+k=4,~s\geq 2$. The proof is almost the same as Section \ref{whe3}.

Let $k=3$ in \eqref{heatk+1} and \eqref{wavek+1}, and then we have 
\begin{equation}\label{heatlow4}
\begin{aligned}
\Delta D_t^2 B&=\lambda^{-1}(D_t^3 B-h_3-\hh_3-\hl_3), \\
\nabla\Delta D_t B&=\lambda^{-1}\nabla(D_t^2 B-h_2-\hh_2-\hl_2), \\
\nabla^2\Delta  B&=\lambda^{-1}\nabla^2(D_t B-(B\cdot \nabla )u-B\wk D_t p);
\end{aligned}
\end{equation} as  well as
\begin{equation}\label{wavelow4}
\begin{aligned}
\Delta D_t^2 p&=\wt D_t^4 p-\wl_3-w_3-\wh_3, \\
\nabla\Delta D_t p&=\nabla(\wt D_t^3 p-\wl_2-w_2-\wh_2), \\
\nabla^2\Delta p&=\nabla^2(\wt D_t^2 p-\wl_1-w_1-\wh_1).
\end{aligned}
\end{equation}

Again, one can notice that all the terms except $\|\wt D_t^4 p\|_{\lli}$ and $\|\wt \nabla D_t^3 p\|_{\lli}$ on the RHS of \eqref{heatlow4} and \eqref{wavelow4} are of $\leq 3$ derivatives and thus can be bounded by \eqref{sk3}. Therefore, we have
\begin{equation}\label{w4h4}
\begin{aligned}
&~~~~\|\Delta D_t^2 B\|_{\lli}+\|\nabla \Delta D_t B\|_{\lli}+\|\nabla^2 \Delta B\|_{\lli}+\|\Delta D_t^2 p\|_{\lli}+\|\nabla \Delta D_t p\|_{\lli}+\|\nabla^2 \Delta p\|_{\lli}\\
&\lesssim_{M,c_0}\|D_t^4 p\|_{\lli}+\|\nabla D_t^3 p\|_{\lli}+\|\nabla^2 D_t^2 p\|_{\lli}+(\text{terms of }\leq 3\text{ derivatives})\\
&\lesssim_{M,c_0}  \frac{1}{\lambda}(1+\sqrt{E_4^*}).
\end{aligned}
\end{equation}

Now, \eqref{w3h3} and \eqref{w4h4} help us to bound the second fundamental form $\cnab^2\theta$ on the boundary and thus all the interior terms $\|B\|_{s,k},\|p\|_{s,k}$:
\begin{itemize}

\item Control of $\theta$: 

Combining \eqref{22nd}, \eqref{p401} and \eqref{w4h4}, one gets
\begin{equation}\label{2nd}
|\cnab^2\theta|_{\llb}+|\cnab\theta|_{H^1(\p\Omega)}\lesssim_{K,1/\lambda,1/\epsilon_0} 1+\sqrt{E_4^*}.
\end{equation}

\item Control of interior terms:

Summing up \eqref{ub40}, \eqref{p401}, \eqref{B311}, \eqref{p311}, \eqref{B221}, \eqref{p221}, then using \eqref{w4h4} and \eqref{2nd}, we have 
\begin{equation}\label{sk41}
\begin{aligned}
&~~~~\|u\|_{4,0}+\sum_{s+k=4}\|B\|_{s,k}+\sum_{s+k=4,s\geq 2}\|p\|_{s,k}+\|\swt \nabla D_t^3 p\|_{\lli}+\|\wt D_t^4 p\|_{\lli}\\
&\lesssim_{K,M,c_0,\volo,1/\epsilon_0,1/\lambda}\delta\sqrt{E_3^*}(\sum_{s+k=4,s\geq 2}\|p\|_{s,k}+\|B\|_{s,k})+\delta\sqrt{E_3^*}\sqrt{E_4^*}+\sqrt{E_4^*}\\
&~~~~~~~~~~~~~~~~~~~~~~~~+\|\Delta D_t^2 B\|_{\lli}+\|\nabla \Delta D_t B\|_{\lli}+\|\nabla^2 \Delta B\|_{\lli}\\
&~~~~~~~~~~~~~~~~~~~~~~~~+\|\Delta D_t^2 p\|_{\lli}+\|\nabla \Delta D_t p\|_{\lli}+\|\nabla^2 \Delta p\|_{\lli} \\
&\lesssim_{K,M,c_0,\volo,1/\epsilon_0,1/\lambda}\delta\sqrt{E_3^*}(\sum_{s+k=4,s\geq 2}\|p\|_{s,k}+\|B\|_{s,k})+(1+\sqrt{E_3^*})\sqrt{E_4^*}.
\end{aligned}
\end{equation}

By using Young's inequality and choosing a sufficiently small $\delta>0$ such that the $\delta$-term can be absorbed to LHS of \eqref{sk41}, one has
\begin{equation}\label{sk4}
\begin{aligned}
&~~~~\|u\|_{4,0}+\sum_{s+k=4}\|B\|_{s,k}+\sum_{s+k=4,s\geq 2}\|p\|_{s,k}+\|\swt \nabla D_t^3 p\|_{\lli}+\|\wt D_t^4 p\|_{\lli}\\
&\lesssim_{K,M,c_0,\volo,1/\epsilon_0,1/\lambda}\left(1+\sqrt{E_3^*}\right)\sqrt{E_4^*}.
\end{aligned}
\end{equation}
\end{itemize}

With the help of \eqref{sk4}, one can repeat the steps above for one more time to derive the control of $\|\nabla^s D_t^k(B\cdot\nabla B)\|_{\lli}, \|\nabla^s D_t^{k+1}B\|_{\lli}$ and $\|\nabla^s D_t^{k+1}p\|_{\lli}$ for $s\geq 2$. In fact, summing up \eqref{B411}, \eqref{p411}, \eqref{B321}, \eqref{p321}, \eqref{B23}, \eqref{p23}, then combining \eqref{sk4}, we can get the following bounds for the higher order interior terms after choosing a sufficiently small $\delta>0$ in those previous estimates to absorb the $\delta$-terms to LHS
\begin{equation}\label{sk5}
\begin{aligned}
&~~~~\sum_{s+k=4,s\geq 2}\|p\|_{s,k+1}+\|B\|_{s,k+1}+\|B\cdot\nabla B\|_{s,k}+|B|_{s-1,k+1}+|p|_{s-1,k+1}\\
&\lesssim_{K,M,c_0,\volo,1/\epsilon_0,1/\lambda}\left(1+\sqrt{E_4^*}\right)\sqrt{E_4^*}+\sum_{s+k=4,s\geq 2}\|\nabla^{s-2}\Delta D_t^{k+1} B\|_{\lli}+\|\nabla^{s-2}\Delta D_t^{k+1} p\|_{\lli}.
\end{aligned}
\end{equation}

\subsection{Control of tangential projections}

We still need to control the tangential projection terms which appears in \eqref{tg}
\begin{equation}\label{tgproj}
\begin{aligned}
&\sum_{s+k=r,s\geq 2}|\Pi\nabla^sD_t^k P|_{\llb}\bigg(|\Pi\nabla^s D_t^{k+1}P|_{\llb}+|\Pi(\nabla_i P)(\nabla^s D_t^k u^i)|_{\llb}\\
&~~~~~~~~~~~~~~~~~~~~~~~~~~~~~~~~~~~~~~~~~~~~~~~~~~~~~~~~~+\sum_{m=0}^{s-1}|\Pi((\nabla^{m+1}u)\symdot\nabla^{s-m}D_t^k P)|_{\llb}\bigg) \\
&+|\Pi\nabla^r P|_{\llb}\bigg(|\Pi\nabla^r D_t P|_{\llb}+\sum_{m=0}^{r-2}|\Pi((\nabla^{m+1}u)\symdot\nabla^{r-m}P)|_{\llb}\bigg).
\end{aligned}
\end{equation}For simplicity we still only give the details for the top order case $r=4$. Lower order cases are similar and easier.

First we control the term $|\Pi \nabla^s D_t^{k+1} P|_{\llb}$ for $s\geq 2$. We have 
\[
|\Pi\nabla^sD_t^{k+1} P|_{\llb}\lesssim_K |\cnab^{s-2}\theta(\nabla_N D_t^{k+1}P)|_{\llb}+\sum_{l=1}^{s-1}|\nabla^l D_t^{k+1}P|_{\llb}.
\]

The second term is a part of $|P|_{s-1,k+1}\lesssim|p|_{s-1,k+1}+|B|_{s-1,k+1}$ which has been controlled before, while the first term is bounded in the same way as the previous sections.

The remaining work is to bound the following terms for $s+k=4, s\geq 2$ :
\[
|\Pi(\nabla_i P)(\nabla^s D_t^k u^i)|_{\llb},~\sum_{m=0}^{s-1}|\Pi((\nabla^{m+1}u)\symdot\nabla^{s-m}D_t^k P)|_{\llb},~\sum_{m=0}^{2}|\Pi((\nabla^{m+1}u)\symdot\nabla^{4-m}P)|_{\llb}.
\]

\begin{itemize}

\item $\sum_{m=0}^{s-1}|\Pi((\nabla^{m+1}u)\symdot\nabla^{s-m}D_t^k P)|_{\llb}$ for $k>0$

\begin{itemize}
\item $s=3,~k=1:$ We use Sobolev interpolation Theorem \ref{gag-ni thm} to get
\[
\begin{aligned}
&~~~~|\Pi(\nabla u\symdot \nabla^3 D_t P)|_{\llb}+|\Pi(\nabla^2 u\symdot \nabla^2 D_t P)|_{\llb}+|\Pi(\nabla^3 u\symdot\nabla D_t P)|_{\llb}\\
&\lesssim_{K,M}|\nabla^3 D_t P|_{\llb}+|\nabla^2 D_t P|_{\llb}^{1/2}|\nabla^2 u|_{\llb}^{1/2}|\nabla^2 D_t P|_{H^1(\p\Omega)}^{1/2}|\nabla^2 u|_{H^1(\p\Omega)}^{1/2}\\
&~~~~~~~+|\nabla^3 u|_{\llb}  \\
&\lesssim_{K,M,\volo,1/\epsilon_0} (1+\sqrt{E_3^*})^2\sqrt{E_4^*}.
\end{aligned}
\]

\item $s=k=2:$ Again, we use Sobolev interpolation to get
\[
\begin{aligned}
&~~~~|\Pi(\nabla u\symdot \nabla^2 D_t^2 P)|_{\llb}+|\Pi(\nabla^2 u\symdot \nabla D_t^2 P)|_{\llb}\\
&\lesssim_{K,M}|\nabla^2 D_t^2 P|_{\llb}+|\nabla D_t^2 P|_{\llb}^{1/2}|\nabla^2 u|_{\llb}^{1/2}|\nabla D_t^2 P|_{H^1(\p\Omega)}^{1/2}|\nabla^2 u|_{H^1(\p\Omega)}^{1/2}\\
&\lesssim_{K,M,\volo,1/\epsilon_0} (1+\sqrt{E_3^*})^2\sqrt{E_4^*}.
\end{aligned}
\]

\end{itemize}

\item $\sum_{m=0}^{2}|\Pi((\nabla^{m+1}u)\symdot\nabla^{4-m}P)|_{\llb}$.

To bound this term, one needs the following lemma:
\begin{lem}\label{tensorprod}
Let $S,T$ be two tensors, then it holds that
\[
\Pi(S\symdot T)=\Pi(S)\symdot\Pi(T)+\Pi(S\symdot N)\tilde{\otimes}\Pi(N\symdot T),
\]where $\otimes$ denotes the symmetric tensor product which is defind similarly as the symmetric dot product.
\end{lem}

\begin{proof}
This is a straightforward result of $g^{ab}=\gamma^{ab}+N^a N^b$.
\end{proof}

The three terms in this sum are $$|\Pi((\nabla u)\symdot\nabla^4 P)|_{\llb}+|\Pi((\nabla^{2}u)\symdot\nabla^3 P)|_{\llb}+|\Pi((\nabla^{3}u)\symdot\nabla^{2}P)|_{\llb},$$ which by Lemma \ref{tensorprod} can be bounded by 
\begin{equation}\label{tgmid}
\begin{aligned}
&~~~~~~~~|\Pi\nabla u|_{\linf}|\Pi\nabla^4 P|_{\llb}+|\Pi\nabla^3 u|_{\llb}|\Pi\nabla^2P|_{\linf}\\
&~~~~+|\Pi N\cdot\nabla u|_{\linf}|\Pi N^j\nabla^3\nabla_j P|_{\llb}+|\Pi N^j\nabla^2\nabla_j u|_{\llb}|\Pi N^j\nabla\nabla_jP|_{\linf}\\
&~~~~+|\Pi\nabla^2u|_{L^4(\p\Omega)}|\Pi\nabla^3 P|_{L^4(\Omega)}+|\Pi N^j\nabla\nabla_j u|_{L^4(\p\Omega)}|\Pi N_j\nabla^2\nabla_j P|_{L^4(\Omega)}.
\end{aligned}
\end{equation}

The first and the second line of \eqref{tgmid} can be controlled by $\sqrt{E_4^*}$ times the quantities in the a priori assumptions. The terms in the last line can be bounded by using tensor interpolation in Theorem \ref{tensor interpolation}. The result is 
\begin{equation}
\begin{aligned}
&~~~~|\Pi\nabla^2u|_{L^4(\p\Omega)}|\Pi\nabla^3 P|_{L^4(\Omega)}+|\Pi N^j\nabla\nabla_j u|_{L^4(\p\Omega)}|\Pi N_j\nabla^2\nabla_j P|_{L^4(\Omega)}\\
&\lesssim_{K,M} (|\nabla u|_{\linf}+\sum_{j\leq 2}|\nabla^j v|_{\llb})|\nabla^4 P|_{\llb} \\
&~~~~~~~~+(|\nabla^2 P|_{\linf}+\sum_{j\leq 3}|\nabla^j P|_{\llb})|\nabla^3 u|_{\llb}\\
&~~~~~~~~+(|\theta|_{\linf}+|\cnab^2\theta|_{\linf})(|\nabla u|_{\linf}+\sum_{j\leq 2}|\nabla^ju|_{\llb})\\
&~~~~~~~~~~\times(|\nabla^2 P|_{\linf}+\sum_{j\leq 3}|\nabla^j P|_{\llb})\\
&\lesssim_{K,M,\volo,1/\epsilon_0} 1+E_4^*.
\end{aligned}
\end{equation}

\item $|\Pi(\nabla_i P)(\nabla^s D_t^k u^i)|_{\llb}$ for $k>0$:

For this term, we can mimic the proof of Lemma \ref{usk}, i.e., use the first equation of the MHD system $\eqref{CMHD}$ to reduce the estimates of $\nabla^s D_t^k u$ to that of $|B|_{s,k}, |p|_{s,k}, |u|_{r-1,0}$. This term has the following control:
\begin{equation}
|\Pi(\nabla_i P)(\nabla^3 D_tu^i)|_{\llb}\lesssim_{M} |B|_{4,0}+|p|_{4,0},
\end{equation}
and
\begin{equation}\label{}
|\Pi(\nabla_i P)(\nabla^2 D_t^2 u^i)|_{\llb}\lesssim_{M} |B|_{3,1}+|p|_{3,1}+|B|_{3,0}+|p|_{3,0}+\|u\|_{4,0},
\end{equation}
where these terms again have been bounded in the previous sections.
\end{itemize}
\begin{flushright}
$\square$
\end{flushright}

Now we have reduced all the estimates (except $W_5$ and $H_5$) to the control of $\|\nabla^{s-2}\Delta D_t^{k+1}B\|_{\lli}$ and  $\|\nabla^{s-2}\Delta D_t^{k+1}p\|_{\lli}$ for $s\geq 2,~s+k=4$. Considering $$E_r=\sum_{s+k=r}E_{s,k}+K_r+W_{r+1}^2+H_{r+1}^2,$$ or from the diagram \eqref{process} we can assert that all the difficulties have been reduced to the control of $W_{r+1}^2+H_{r+1}^2$. We will do this in the next section to complete all the a priori estimates.

\section{Energy estimates for $W_5$ and $H_5$: The last step to close the energy bound}\label{section7}

In this section we will give control of $\|\nabla^{s-2}\Delta D_t^{k+1}B\|_{\lli}$ and  $\|\nabla^{s-2}\Delta D_t^{k+1}p\|_{\lli}$ for $s\geq 2,~s+k=4$ together with $W_5^2+H_5^2$ to complete all the estimates under the a priori assumptions.

Again, from the heat/wave equations \eqref{heatk+1} and \eqref{wavek+1}, one has
\begin{equation}\label{heatlow5}
\begin{aligned}
\Delta D_t^3 B&=\lambda^{-1}(D_t^4 B-h_4-\hh_4-\hl_4), \\
\nabla\Delta D_t^2 B&=\lambda^{-1}\nabla(D_t^3 B-h_3-\hh_3-\hl_3), \\
\nabla^2\Delta D_t B&=\lambda^{-1}\nabla^2(D_t^2 B-h_2-\hh_2-\hl_2);
\end{aligned}
\end{equation} as  well as
\begin{equation}\label{wavelow5}
\begin{aligned}
\Delta D_t^3 p&=\wt D_t^5 p-\wl_4-w_4-\wh_4, \\
\nabla\Delta D_t^2 p&=\nabla(\wt D_t^4 p-\wl_3-w_3-\wh_3), \\
\nabla^2\Delta D_t p&=\nabla^2(\wt D_t^3 p-\wl_2-w_2-\wh_2).
\end{aligned}
\end{equation}

As one can see from the wave equation \eqref{wavek+10}, the estimates of $\nabla^{s-2}\Delta D_t^{k+1}p$ can be converted to that of $\nabla^{s-2}(\wt D_t^{k+3}p)$ and $\nabla^{s-2}\Delta D_t^{k+1}B$, i.e., $\wt D_t^5 p, \swt \nabla D_t^4 p$ , $\nabla^2 D_t^3 p$, $\nabla^3 D_t^2 p$ (this one appears in some commutators) and $\nabla^{s-2}\Delta D_t^{k+1}B$ plus the other terms with $\leq 4$ derivatives. On one hand, $\wt D_t^5 p$ and $\swt \nabla D_t^4 p$ is a part of $W_5$, while $\nabla^2 D_t^3 p$ and $\nabla^3 D_t^2 p$ can again be simplified to $\wt D_t^5 p$ and $\nabla D_t^4 p$ after using elliptic estimate and invoking wave equation. The energy $W_5$ will be controlled together with $H_5$. On the other hand, from \eqref{heatlow5}, one finds that $\nabla^{s-2}\Delta D_t^{k+1}B$ can be reduced to $\nabla^{s-2}D_t^{k+2}B$ plus other terms with $\leq 4$ derivatives. In other words, $\nabla^{s-2}\Delta D_t^{k+1}B$ can all be reduced to the estimates of 4-derivative terms computed in the previous sections. Therefore, \textbf{all the difficulties are further reduced to seek a common control of $W_5$ and $H_5$ by those terms with $\leq 4$ derivatives.} 

\bigskip

\subsubsection*{Heat equation}

\eqref{heatk+1} gives us the 5-th order heat equation for $B$ is
\begin{equation}\label{heat5}
D_t^5 B-\lambda\Delta D_t^4 B=h_5+\hh_5+\hl_5.
\end{equation}

Multiplying $D_t^5 B$ on both sides of \eqref{heat5}, integrating in $y\in\Omega$, then integrating by part to eliminiate the Laplacian, we get 
\[
\begin{aligned}
&~~~~\io |D_t^5 B|^2~ J\dy+\frac{\lambda}{2}\frac{d}{dt}\io |\nabla D_t^4 B|^2~ J\dy \\
&=\io (h_5+\hh_5+\hl_5)\cdot D_t^5 B ~J\dy +\lambda\io \nabla D_t^4 B \cdot([D_t, \nabla]D_t^4 B)~ J\dy-\lambda\io \nabla D_t^4 B\cdot D_t^5 B ~\nabla J\dy.
\end{aligned}
\]

Then we integrate the last identity in time $t\in[0,T]$ for some $T>0$ and the use H\"older's inequality, Young's inequality to get: $\forall\delta>0$,
\begin{equation}\label{H51}
\begin{aligned}
&~~~~H_5^2(T)-H_5^2(0)=\int_0^T\io |D_t^5 B|^2~ J\dy\dt+\frac{\lambda}{2}\frac{d}{dt}\io |\nabla D_t^4 B|^2~ J\dy\\
&=\int_0^T\left(\io (h_5+\hh_5+\hl_5)\cdot D_t^5 B ~J\dy +\lambda\io \nabla D_t^4 B \underbrace{[D_t, \nabla]D_t^4 B~ J}_{=\nabla u\symdot \nabla D_t^4 B}\dy-\lambda\io \nabla D_t^4 B\cdot D_t^5 B ~\nabla J\dy\right)\dt\\
&\lesssim_{M}\|D_t^5 B\|_{\lylt}^2\left(\|h_5\|_{\lylt}+\|\hh_5\|_{\lylt}+\|\hl_5\|_{\lylt}\right)\\
&~~~~+ \int_0^T H_5^2(t)\dt +\|D_t^5 B\|_{\lylt}\|\nabla D_t^4 B\|_{\lylt}\\
&\lesssim \delta \int_0^T\|D_t^5B\|_{\lli}^2\dt+\frac{1}{4\delta}\int_0^T\|h_5\|_{\lli}^2+\|\hh_5\|_{\lli}^2+\|\hl_5\|_{\lli}^2\dt +\int_0^T H_5^2(t)\dt.
\end{aligned}
\end{equation}

Choosing a suitably small $\delta>0$ such that the first term in the last step can be absorbed by LHS of \eqref{H51},and thus we have 
\begin{equation}\label{H52}
H_5^2(T)-H_5^2(0)\lesssim_M \int_0^T\|h_5\|_{\lli}^2+\|\hh_5\|_{\lli}^2+\|\hl_5\|_{\lli}^2\dt +\int_0^T H_5^2(t)\dt.
\end{equation}

Now we are going to control $\|h_5\|_{\lli}$, $\|\hl_5\|_{\lli}$ and $\|\hh_5\|_{\lli}$. The first two terms are 5-th order terms.
\begin{itemize}
\item Control of $\|\hh_5\|_{\lli}$:

We have 
\[
\begin{aligned}
\hh_5&=[D_t^k, B\cdot\nabla]u+[D_t^k, B\wk]D_t p\\
&=\sum_{m=1}^{4} C_4^l D_t^m B^l\nabla_lD_t^m u+ D_t^m\left( B\wk\right)D_t^{5-m}p+ D_t^{4-m} B^l ([D_t^m,\nabla_l]u),
\end{aligned}
\] where all the terms with $\geq 3$ derivatives are $\|u\|_{1,3},\|u\|_{1,2},\|B\|_{0,4},\|B\|_{0,3},\|\wt D_t^4 p\|_{\lli}$ and $\|\swt D_t^3p\|_{\lli}$. Hence we have the estimates for $\hh_5$ that
\begin{equation}\label{hh5}
\|\hh_5\|_{\lli}\lesssim_M (1+ \sqrt{E_4^*})^2.
\end{equation}

Before coming to control $\|\hl_5\|_{\lli}$ and $\|h_5\|_{\lli}$, we need the following lemma to convert the terms containing 5 derivatives of $u$ to that of $p$ and $B$ by using the first equation of the MHD system \eqref{CMHD}.

\begin{lem}\label{u5}
We have the following estimates for $u$:
\[
\|\nabla D_t^4 u\|_{\lli}\lesssim_M \|B\|_{2,3}+\|p\|_{2,3}+1+E_4^*,
\]and
\[
\|\Delta D_t^3 u\|_{\lli}\lesssim_M \|\nabla\Delta D_t^2 B\|_{\lli}+\|\nabla\Delta D_t^2 p\|_{\lli}+1+E_4^*.
\]
\end{lem}

\begin{proof}
 The proof is almost the same as Lemma \ref{usk}. From the first equation of \eqref{CMHD}, \eqref{commutator2} and \eqref{commutator4}, one has 
\[
\begin{aligned}
\nabla D_t^4 u&=\nabla D_t^3 \left( \frac{1}{\rho} ((B\cdot\nabla B)-\nabla p-(\nabla B)\cdot B)\right)
&=\frac{1}{\rho} \left(B\cdot \nabla D_t^4 B -\nabla^2 D_t^3 p\right)\\
&~~~~+(\text{terms of }\leq 4\text{ derivatives})\\
&\lesssim_M \|B\|_{2,3}+\|p\|_{2,3}+1+E_4^*.
\end{aligned}
\] Similar proof holds for $\Delta D_t^3 u$, so we omit the details.
\end{proof}

\item Control of $\|\hl_5\|_{\lli}$:

Recall we have
\[
\begin{aligned}
\lambda^{-1}\hl_5&= [D_t^4,\Delta]B \\
&=C\left((\Delta D_t^3 u)(\nabla B)+(\nabla v)(\nabla^2 D_t^3 B)  \right) \\
&~~~~+\sum_{l=0}^2 c_l (\Delta D_t^{l}u)(\nabla D_t^{3-l}B)+\sum_{l=0}^{2} d_l (\nabla^{3-l} u)(\nabla^2 D_t^{l} B) +L.O.T.,
\end{aligned} 
\]where L.O.T means the terms with $\leq 3$ derivatives in the commutator.

Therefore one has the bound
\begin{equation}\label{hl5}
\begin{aligned}
\|\hl_5\|_{\lli}&\lesssim_M\lambda\left(\|\Delta D_t^3u\|_{\lli}+\|\nabla^2 D_t^3 B\|_{\lli}+1+E_4^*\right) \\
&\lesssim_M\lambda\left(\|\nabla\Delta D_t^2 B\|_{\lli}+\|\nabla\Delta D_t^2 p\|_{\lli}+\|\nabla^2 D_t^3 B\|_{\lli}+1+E_4^*\right)\\
&\lesssim_{K,M,\volo}\lambda\left(\|\nabla\Delta D_t^2 B\|_{\lli}+\|\nabla\Delta D_t^2 p\|_{\lli}+\|\Delta D_t^3 B\|_{\lli}+1+E_4^*\right),
\end{aligned}
\end{equation}where in the second step we use Lemma \ref{u5}, and in the last step we use \eqref{B23}.

\item Control of $\|h_5\|_{\lli}$:

This step also needs Lemma \ref{u5} to convert $\nabla D_t^4 u$ to $\nabla^2 D_t^3 p$ and $\nabla^2 D_t^3 B$. We have
\begin{equation}\label{h5}
\begin{aligned}
~~~~h_5&=(B\cdot\nabla)D_t^4 u+ B\wk D_t^5 p \\
\Rightarrow \|h_5\|_{\lli}&\lesssim_{M,c_0}\|\nabla^2D_t^3 B\|_{\lli}+\|\nabla^2 D_t^3 p\|_{\lli}+\|\wt D_t^5 p\|_{\lli}+1+E_4^* \\
&\lesssim_{K,M,c_0,\volo} \|\Delta D_t^3 B\|_{\lli}+\|\Delta D_t^3 p\|_{\lli}+W_5+1+E_4^*.
\end{aligned}
\end{equation}
\end{itemize}

Combining \eqref{H52}, \eqref{hh5}, \eqref{hl5} and \eqref{h5}, we have the bound for $H_5$:
\begin{equation}\label{H53}
\begin{aligned}
H_5^2(T)-H_5^2(0)&\lesssim_{K,M,\volo,c_0}\int_0^T H_5^2(t)+W_5^2(t)+\PP(E_4^*(t))\dt \\
+\int_0^T &\|\nabla^2D_t^3 B(t)\|_{\lli}^2+\|\nabla^2 D_t^3 p(t)\|_{\lli}^2+\|\nabla\Delta D_t^2 B(t)\|_{\lli}^2+\|\nabla\Delta D_t^2 p(t)\|_{\lli}^2\dt
\end{aligned}
\end{equation}

From \eqref{heatlow5} and \eqref{wavelow5}, one can reduce the 5-th order terms in \eqref{H53} to $\leq 4$-th order terms. Therefore we are able to use $E_4^*$ and $W_5$ to bound $H_5$
\begin{equation}\label{H5}
\begin{aligned}
&~~~~H_5^2(T)-H_5^2(0) \\
&\lesssim_{K,M,\volo,c_0}\int_0^T H_5^2(t)+W_5^2(t)+\PP(E_4^*(t))\dt \\
&~~~~+\left(1+\frac{1}{\lambda}\right)\int_0^T \bigg(\|D_t^4 B(t)\|_{\lli}^2+\|h_4(t)\|_{\lli}^2+\|\hl_4(t)\|_{\lli}^2+\|\hh_4(t)\|_{\lli}^2 \\
&~~~~~~~~~~~~+\|\nabla D_t^3 B\|_{\lli}^2+\|\nabla h_3(t)\|_{\lli}^2+\|\nabla \hl_3(t)\|_{\lli}^2+\|\nabla \hh_3(t)\|_{\lli}^2\bigg)\dt \\
&~~~~+\int_0^T \bigg(\|\wt D_t^4 p(t)\|_{\lli}^2+\|w_4(t)\|_{\lli}^2+\|\wl_4(t)\|_{\lli}^2+\|\wh_4(t)\|_{\lli}^2 \\
&~~~~~~~~~~~~+\|\swt\nabla D_t^3 p(t)\|_{\lli}^2+\|\nabla w_3(t)\|_{\lli}^2+\|\nabla \wl_3(t)\|_{\lli}^2+\|\nabla \wh_3(t)\|_{\lli}^2\bigg)\dt \\
&\lesssim_{K,M,\volo,c_0}\int_0^T H_5^2(t)+W_5^2(t)+\left(1+\frac{1}{\lambda}\right)\PP(E_4^*(t))\dt.
\end{aligned}
\end{equation}

\subsubsection*{Wave equation}

Let $k=4$ in \eqref{wavek+1} and we can get the 5-th order wave equation:
\begin{equation}\label{wave5}
\wt D_t^6 p-\Delta D_t^4 p=w_5+\wh_5+\wl_5.
\end{equation}

Multiplying $\wt D_t^5 p$ on both sides of \eqref{wave5}, then integrating by parts to eliminate Laplacian term, one has
\[
\begin{aligned}
&~~~~\frac{1}{2}\frac{d}{dt} \left(\io \|\wt D_t^5p\|_{\lli}^2+\|\swt \nabla D_t^4 p\|^2 \right)=\frac{d}{dt} W_5^2(t)\\
&=\io \wt (w_5+\wh_5+\wl_5) D_t^5 p~J\dy+\io \wt \nabla D_t^4 p\cdot ([D_t, \nabla]D_t^4 p)~J\dy-\io \wt\nabla D_t^4 p\cdot\nabla J~ D_t^5 p\dy\\
&~~~~+\io\nabla(\wt)\cdot \nabla D_t^4 p D_t^5 p\dy.
\end{aligned}
\]

Note that $[D_t, \nabla]D_t^4 p=\nabla u\symdot \nabla D_t^4 p$ and $|\rho''(p)|\lesssim_{c_0}\wt^2$, so one has the following estimates for $W_5$ after integrating in time variable in $t\in[0,T]$.
\begin{equation}\label{W51}
\begin{aligned}
&~~~~W_5^2(T)-W_5^2(0) \\
&\lesssim_{M, c_0} \int_0^T (\|w_5\|_{\lli}+\|\wh_5\|_{\lli}+\|\wl_5\|_{\lli})\|\wt D_t^5 p\|_{\lli} \dt+\int_0^T \|\swt \nabla D_t^4p(t)\|_{\lli}^2\dt \\
&~~~~+\int_0^T \|\wt D_t^5 p(t)\|_{\lli} \|\swt \nabla D_t^4p(t)\|_{\lli} \dt\\
&\lesssim_{M, c_0}\int_0^T(\|w_5\|_{\lli}+\|\wh_5\|_{\lli}+\|\wl_5\|_{\lli})\|\wt D_t^5 p\|_{\lli} \dt+\int_0^T W_5^2(t)\dt.
\end{aligned}
\end{equation}

Now we are going to bound $\wl_5,~\wh_5$ and $w_5$.
\begin{itemize}
\item Control of $\wh_5$: 

From \eqref{wk+1}, we know
\[
\begin{aligned}
\wh_5&=\sum_{i_1+\cdots+i_m=6,~1\leq i_k\leq 5}c_{i_1,\cdots,i_m}\rho^{(m)}(p)(D_t^{i_1}p)\cdots(D_t^{i_m}p)\\
&=\rho''(p) D_t^5 D_t p+\sum_{i_1+\cdots+i_m=6,~1\leq i_k\leq 4}c_{i_1,\cdots,i_m}\rho^{(m)}(p)(D_t^{i_1}p)\cdots(D_t^{i_m}p).
\end{aligned}
\]

Since $|\rho^{(m)}(p)|\lesssim_{c_0}\wt^m$, one has the energy bound for $\wh_5$
\begin{equation}\label{wh5}
\|\wh_5\|_{\lli}\lesssim_{M,c_0} W_5+ 1+E_4^*.
\end{equation}

\item Control of $\wl_5$:

From \eqref{wk+1} we know
\[
\begin{aligned}
\lambda\wl_5&=B\cdot D_t^5B-B\cdot h_5-B\cdot \hh_5\\
&~~~~+\sum_{l=1}^4C_4^l D_t^l B\cdot \bigg(D_t^{5-l}B-D_t^{4-l}(B\cdot\nabla)u-  D_t^{4-l}\left( B\wk D_t p\right)\bigg).
\end{aligned}
\]

We notice that the second line only contains $\leq 4$ derivatives of $u, B, p$, and thus controlled by $E_4^*$. Combining this with \eqref{hh5} and \eqref{h5}, one has 
\[
\begin{aligned}
&~~~~\int_0^T\io \wt \wl_5 (t) D_t^5 p(t) ~J\dy\dt \\
&\lesssim_{K, M, c_0, \volo}\frac{1}{\lambda}\int_0^T \|D_t^5B(t)\|_{\lli}\|\wt D_t^5 p\|_{\lli}\dt+\int_0^T W_5(t) (1+E_4^*(t))\dt \\
&~~~~~~~~+\frac{1}{\lambda}\|\wt D_t^5 p(t)\|_{\lli}\left(\|\nabla^2 D_t^3 B(t)\|_{\lli}+\|\nabla^2 D_t^3 p(t)\|_{\lli}+\|\wt D_t^5 p(t)\|_{\lli}\right)\dt\\
\end{aligned}
\]

By elliptic estimate \eqref{B23}, we know $\|\nabla^2 D_t^3 B\|_{\lli}$ can be bounded by $\|\Delta D_t^3 B\|_{\lli}$ which can again be reduced to 4-th order terms by using \eqref{heatlow5}. Hence the above estimates can be rewritten to be
\begin{equation}
\begin{aligned}\label{wl51}
&~~~~\int_0^T\io \wt \wl_5 (t) D_t^5 p(t) ~J\dy\dt \\
&\lesssim_{K,M,\volo,c_0}\frac{1}{\lambda}\delta H_5^2(T)+\left(1+\frac{1}{4\delta\lambda}+\frac{1}{\lambda}\right) \int_0^T W_5^2(t)\dt+\left(1+\frac{1}{\lambda}\right)\int_0^TE_4^*(t)\dt \\
&~~~~~~~~+\frac{1}{\lambda}\int_0^T \|\wt D_t^5 p(t)\|_{\lli}\|\nabla^2 D_t^3 p(t)\|_{\lli}\dt
\end{aligned}
\end{equation} for any $\delta>0$.

Again, by the elliptic estimate \eqref{p23} and \eqref{wavelow5}, one has
\[
\|\nabla^2 D_t^3 p\|_{\lli}\lesssim\|\Delta D_t^3 p\|_{\lli}\lesssim\|\wt D_t^5 p\|_{\lli}+\underbrace{\|\wl_4\|_{\lli}+\|\wh_4\|_{\lli}+\|w_4\|_{\lli}}_{\leq\text{ 4 derivatives}}
\]  and thus 
\[
\|\nabla^2 D_t^3 p\|_{\lli}\lesssim_{K,M,c_0,\volo}\|\wt D_t^5 p\|_{\lli}+\left(1+\frac{1}{\lambda}\right)(1+E_4^*).
\]

Combining this with \eqref{wl51}, one can bound $\wl_5$ as follows
\begin{equation}
\begin{aligned}\label{wl5}
&~~~~\int_0^T\io \wt \wl_5 (t) D_t^5 p(t) ~J\dy\dt \\
&\lesssim_{K,M,\volo,c_0}\frac{1}{\lambda}\delta H_5^2(T)+\left(1+\frac{1}{4\delta\lambda}+\frac{1}{\lambda}\right) \int_0^T W_5^2(t)\dt+\left(1+\frac{1}{\lambda}\right)\int_0^T E_4^*(t)\dt.
\end{aligned}
\end{equation} for any $\delta>0$.

\item Control of $w_5$:

Recall from \eqref{w} and \eqref{wk+1} that we have
\begin{equation}\label{mhdnmsl}
\begin{aligned}
w_5&=D_t^4 w+[D_t^4,\Delta ]p\\
&=D_t^4\bigg(\left(\wk-\rho''(p)\right)(D_tp)^2+\wk\nabla_i p(B\cdot\nabla B_i)-\nabla_i P)+\rho\nabla_i u^k\nabla_k u^i-\nabla^i B_k\nabla_k B^i +|\nabla B|^2\bigg)\\
&~~~~+[D_t^4,\Delta ]p\\
&=2\left(\wk-\rho''(p)\right)D_t^5p D_t p+2\rho \nabla D_t^4 u\symdot\nabla u+(\nabla D_t^4 B)(\nabla B)+(\Delta D_t^3 u)(\nabla p)+(\nabla u)(\nabla^2 D_t^3 p)\\
&~~~~-\wk \bigg(\nabla D_t^4 p\cdot\nabla P+\nabla p\cdot\nabla D_t^4 P-(\nabla D_t^4 P)\cdot (B\cdot\nabla)B-\nabla P\cdot(B\cdot\nabla)D_t^4 B\bigg)\\
&~~~~+X_5,
\end{aligned}
\end{equation}

where $X_5$ consists of:
\begin{itemize}
\item commutators produced when taking $D_t^4$ on $w$;

\item all the terms in $[D_t^4,\Delta ]p$ except $(\Delta D_t^3 u)(\nabla p)+(\nabla u)(\nabla^2 D_t^3 p)$, i.e., all the terms with $\leq 4$ derivatives in $[D_t^4,\Delta]p$.
\end{itemize}

From the commutator \eqref{commutator4}, the precise formula of $X_5$ is:

\begin{equation}\label{X50}
\begin{aligned}
X_5&=\sum_{l=0}^2 c_l (\Delta D_t^{l}u)(\nabla D_t^{3-l}p)+\sum_{l=0}^{2} d_l (\nabla^{3-l} u)(\nabla^2 D_t^{l} p) \\
&~~~~+\sum_{l_1+\cdots+l_n=5-n,n\geq 3}c_{l_1\cdots l_n}(\nabla D_t^{l_3} u)\cdots(\nabla D_t^{l_n}u)(\Delta D_t^{l_1}u)(\nabla D_t^{l_2}p) \\
&~~~~+\sum_{l_1+\cdots+l_n=5-n,n\geq 3}d_{l_1\cdots l_n}(\nabla D_t^{l_3} u)\cdots(\nabla D_t^{l_n}u)(\nabla^2 D_t^{l_1}u)(\nabla D_t^{l_2}p) \\
&~~~~+\sum_{l_1+\cdots+l_n=5-n,n\geq 3}e_{l_1\cdots l_n}(\nabla D_t^{l_3} u)\cdots(\nabla D_t^{l_n}u)(\nabla D_t^{l_1}u)(\nabla^2 D_t^{l_2}p) \\
&~~~~+\left[D_t^4,\wk-\rho''(p)\right](D_t p)^2-\left[D_t^4,\wk\right]((\nabla p)\cdot(\nabla P-(B\cdot\nabla)B))\\
&~~~~-\wk ([D_t^4,\nabla]p)\cdot(\nabla P-(B\cdot\nabla)B)-\wk\sum_{l=1}^{3} C_4^l (D_t^l\nabla_i p)(D_t^{4-l}(B\cdot\nabla)B) \\
&~~~~+[D_t^4,\rho](\nabla u\symdot\nabla u)+2\rho([D_t^4,\nabla]u)\symdot\nabla u+\lambda^{-1}[D_t^4,\nabla]B\cdot(\underbrace{(B\cdot\nabla)u+B\wk D_t p}_{\Delta B})\\
&~~~~+\sum_{l=1}^3 C_4^l (D_t^l\nabla B)\symdot(D_t^{4-l}\nabla B).
\end{aligned}
\end{equation}

One has 
\begin{equation}\label{X5}
\|X_5\|_{\lli}\lesssim_{K,M,c_0,\volo}\left(1+\frac{1}{\lambda}\right)(1+E_4^*),
\end{equation}because all these terms are of $\leq 4$ derivatives and thus controlled by $E_4^*$.

Combining \eqref{mhdnmsl}, \eqref{X5}, together with Lemma \ref{u5}(control of $u$), \eqref{B321}, \eqref{p321}, \eqref{B23} and \eqref{p23}(elliptic estimates for $B$ and $p$), one can finally get the estimates on $w_5$:
\begin{equation}\label{w5}
\begin{aligned}
&~~~~\int_0^T\io\wt D_t^5 p \cdot w_5~J\dy\dt \\
&\lesssim_{K,M,c_0,\volo,1/\lambda} \int_0^T W_5^2(t)\dt+\int_0^T \PP(E_4^*(t))\dt \\
&~~~~~~~~+\int_0^T\|\wt D_t^5 p(t)\|_{\lli}\bigg(\|\nabla D_t^4 B(t)\|_{\lli}+\|\nabla^2 D_t^3 B(t)\|_{\lli}+\|\nabla^2 D_t^3 p(t)\|_{\lli} \\
&~~~~~~~~~~~~~~~~~~~~~~~~~~~~~~~~~~~~~~~~~~~~+\|\nabla^3 D_t^2 B(t)\|_{\lli}+\|\nabla^3 D_t^2 p(t)\|_{\lli}\bigg)\dt\\
&\lesssim_{K,M,c_0,\volo,1/\lambda} \int_0^T W_5^2(t)+\PP(E_4^*(t))\dt+ H_5^2(t)\dt.
\end{aligned}
\end{equation}

Summing \eqref{wh5}, \eqref{wl5} and \eqref{w5}, one gets the estimates on $W_5$:
\begin{equation}\label{W5}
W_5^2(T)-W_5^2(0)\lesssim_{K,M,c_0,\volo,1/\lambda}\delta H_5^2(T)+\int_0^T W_5^2(t)+\PP(E_4^*(t))\dt+ H_5^2(t)\dt,
\end{equation} for any $\delta>0$.
\end{itemize}

Summing up \eqref{H5} and \label{W5}, then picking $\delta>0$ suitably small such that $\delta H_5^2(T)$ can be absorbed by LHS of \eqref{H5}, we finally get the common control of $W_5^2+H_5^2$ by 
\begin{equation}\label{waveheat5}
W_5^2(T)+H_5^2(T)-W_5^2(0)-H_5^2(0)\lesssim_{K,M,c_0,\volo,1/\epsilon_0,1/\lambda}\int_0^T W_5^2(t)+H_5^2(t)+\PP(E_4^*(t))\dt.
\end{equation} Therefore we can bound $W_5$ and $H_5$ by $E_4^*$ and initial data in $t\in[0,T]$ for sufficiently small $T>0$.

\begin{flushright}
$\square$
\end{flushright}

\section{Summary of the estimates and the incompressible limit}

Summing up \eqref{tg}, \eqref{tglow} and \eqref{waveheat5}, we get 
\begin{equation}\label{sum1}
E_4^*(T)-E_4^*(0)\lesssim_{K,M,c_0,\volo,1/\epsilon_0,1/\lambda}\int_0^T \PP(E_4^*(t))\dt
\end{equation}under the a priori assumptions
\[
\begin{aligned}
|\theta|+\frac{1}{\iota_0}&\leq K~~on~~\p \DD_t,\\
-\nabla_N P\geq \epsilon_0&>0~~~on~~\p\DD_t,\\
1\leq|\rho|&\leq M~~in~~\DD_t,\\
\sum_{s+k\leq 2}|\p^s D_t^k p|+|\p^s D_t^k B|+|\p^s D_t^k u|&\leq M~~in~~\DD_t.\\
\end{aligned}
\] Hence, it suffices to recover the bounds of these a priori quantites so that our a priori estimates can be completed.

\subsection{Justification of the a priori assumptions}

The following lemma gives control of these a priori quantities. 
\begin{lem}\label{justify0}
Define $\EE (t):= |(\nabla_N P(t,\cdot))^{-1}|_{\linf}$. Then there exist continuous functions $G$ such that
\begin{equation}\label{justify}
\begin{aligned}
&~~~~\sum_{1\leq s+k\leq 2}\|\nabla^s D_t^k p\|_{L^{\infty}(\Omega)}+\|\nabla^s D_t^k B\|_{L^{\infty}(\Omega)}+\|\nabla^s D_t^k u\|_{L^{\infty}(\Omega)}+|\theta|_{\linf}+|\EE'(t)|\\
&\leq G(K,c_0,E_0,\cdots, E_4,\volo)
\end{aligned}
\end{equation}
\begin{proof}
By Sobolev embedding, one has 
\[
\begin{aligned}
&~~~~\sum_{1\leq s+k\leq 2}\|\nabla^s D_t^k p\|_{L^{\infty}(\Omega)}+\|\nabla^s D_t^k B\|_{L^{\infty}(\Omega)}+\|\nabla^s D_t^k u\|_{L^{\infty}(\Omega)} \\
&\lesssim_K \sum_{s+k\leq 2}\sum_{j=0}^2\|\nabla^{j+s}D_t^k u\|_{\lli}+\|\nabla^{j+s}D_t^k B\|_{\lli}+\|\nabla^{j+s}D_t^k p\|_{\lli}.
\end{aligned}
\]As a result of our previous estimates, the bound for $u,~B,~p$ in \eqref{justify} holds. 

By the definition of $\EE$, one has $|\nabla^2 P|\geq |\Pi\nabla^2 P|=|\nabla_N P\|\theta|\geq\EE^{-1}|\theta|$. Finally,
\[
\frac{d}{dt}|(-\nabla_N P(t,\cdot))^{-1}|_{\linf}\lesssim|(-\nabla_N P(t,\cdot))^{-1}|_{\linf}^2|\nabla_N D_t P(t,\cdot)|_{\linf}
\]implies the bound of $\EE'(t)$.
\end{proof}
\end{lem}

\subsection{Energy estimates}\label{section8}

Now we can close all the a priori estimates with the help of Lemma \ref{justify0}.

\begin{prop}\label{total}
There exists a positive continuous function $\TT$, such that: If $0<T\leq \TT(c_0,K,\EE(0),E_4^*(0),\volo)$, then any solution of \eqref{CMHD} in $t\in [0,T]$ satisfies the following bounds for some polynomial $\PP$ with positive coefficients:
\begin{equation}\label{complete}
\begin{aligned}
E_4^*(t)&\lesssim_{1/\lambda}  2 E_4^*(0), \\
\EE(t)&\lesssim_{1/\lambda} 2 \EE(0), \\
g_{ab}(t,y)Z^aZ^b&\sim g_{ab}(0,y)Z^aZ^b,
\end{aligned}
\end{equation} and there exists some fixed $\eta>0$ such that the following bounds hold
\begin{equation}\label{justify2}
\begin{aligned}
|N(x(t,\bar{y}))-N(x(0,\bar{y}))|&\lesssim\eta,~~\forall \bar{y}\in\p\Omega,\\
|x(t,y)-x(0,y)|&\lesssim\eta,~~\forall y\in\Omega,\\
\left|\frac{\p x (t,\bar{y})}{\p y}-\frac{\p x(0,\bar{y})}{\p y}\right|&\lesssim\eta,~~\forall\bar{y}\in\p\Omega.
\end{aligned}
\end{equation}
\end{prop}

\begin{proof}
From \eqref{sum1} and Lemma \ref{justify0}, one has 
\[
E_4^2(t)-E_4^2(0)\lesssim_{c_0,K,\EE,E_0,\cdots,E_4,\volo,1/\lambda}\int_0^t \PP(E_4^*(s))~ds,
\]where $\PP$ is a polynomial with positive coefficients. The the Gronwall's inequality in \cite{tao2006nonlinear} yields the bound of $E_4^*$ provided that $\TT(c_0,K,\EE(0),E_4^*(0),\volo)>0$ is sufficently small. Therefore the estimates for $\EE(t)$ is a straightforward result from \eqref{justify} and the bounds for $E_4^*$. 

In addition, we get from $E_4^*(t)\lesssim_{1/\lambda}  P(E_4^*(0))$ that all the a priori quantities can be controlled by their intial data for $t\in [0,T]$:
\[
\begin{aligned}
&~~~~\sum_{1\leq s+k\leq 2}\|\nabla^s D_t^k p(t,\cdot)\|_{L^{\infty}(\Omega)}+\|\nabla^s D_t^k B(t,\cdot)\|_{L^{\infty}(\Omega)}+\|\nabla^s D_t^k u(t,\cdot)\|_{L^{\infty}(\Omega)}+|\theta(t,\cdot)|_{\linf}\\
&\lesssim_{1/\lambda}\PP\left(\sum_{1\leq s+k\leq 2}\|\nabla^s D_t^k p(0,\cdot)\|_{L^{\infty}(\Omega)}+\|\nabla^s D_t^k B(0,\cdot)\|_{L^{\infty}(\Omega)}+\|\nabla^s D_t^k u(0,\cdot)\|_{L^{\infty}(\Omega)}+|\theta(0,\cdot)|_{\linf}\right).
\end{aligned}
\]

Besides, one can also bound the $L^{\infty}(\Omega)$ norm of $u, B,\rho$ by their initial data. This follows directly from \eqref{CMHD}. $g_{ab}(t,y)Z^aZ^b\sim g_{ab}(0,y)Z^aZ^b$ holds because $D_t g\sim\nabla u$. Furthermore, this inequality together with $$D_t N_a=-\frac{1}{2} N_a (D_t g^{cd}N_c N_d)$$ implies $$|N(x(t,\bar{y}))-N(x(0,\bar{y}))|\lesssim\eta,~~\forall \bar{y}\in\p\Omega.$$ Finally, the definition of Lagrangian coordiantes $D_t x(t,y)=u(t,x(t,y))$ yields that
\[
\begin{aligned}
|x(t,y)-x(0,y)|&\lesssim\eta,~~\forall y\in\Omega,\\
\left|\frac{\p x (t,\bar{y})}{\p y}-\frac{\p x(0,\bar{y})}{\p y}\right|&\lesssim\eta,~~\forall\bar{y}\in\p\Omega.
\end{aligned}
\]

Before we end the proof of Proposition \ref{total}, we have to make sure that the constants of Sobolev embedding inequalities can be controlled. In fact, these constants depend on $K_0:=\iota_0^{-1}$ which can be chosen to be only dependent on the inital conditions. This result (see the following lemma) has been proved in Lemma 3.6 in Christodoulou-Lindblad \cite{christodoulou2000motion}:
\begin{lem}\label{iota}
Let $0\leq \eta \leq 2$ be a fixed number, and define $l_1=l_1(\eta)$ to be the largest number such that $$|N(\bar{x_1})-N(\bar{x_2})|\leq\eta$$ whenever $|\bar{x_1}-\bar{x_2}|\leq l_1$ and $\bar{x_1},\bar{x_2}\in\p\DD_t$. Suppose also $|\theta|\leq K_0$. Then the injective radius satisfies 
\[
\iota_0\geq\min\{l+1/2,1/K_0\},~~l_1\geq \min\{2\iota_0, \eta/K_0\}.
\]
\end{lem}
\end{proof}

Actually, Lemma \ref{iota} shows that $\iota_0$ and $l_1$ are comparable if the free surface is regular. 

\begin{cor}
Fix $\eta>0$ sufficently small. Let $\TT$ be in Proposition \ref{total}. Choose $l_1>0$ such that $$|N(x(0,y_1))-N(x(0,y_2))|\leq \eta/2$$ holds whenever $|x(0,y_1)-x(0,y_2)|\leq 2l_1$. Then for $t\leq \TT$, one has $$|N(x(t,y_1))-N(x(t,y_2))|\leq \eta$$ whenever $|x(t,y_1)-x(t,y_2)|\leq l_1$.
\end{cor}
\begin{proof}
See Lemma 5.11 in Lindblad-Luo \cite{lindblad2018priori}.
\end{proof}

\begin{rmk}
As shown above, our a priori estimates depend on $1/\lambda$ and thus there is no ``vanishing-resistivity limit". In the rest of this paper, we will suppose $\lambda=1$ for simplicity.
\end{rmk}

\bigskip

\subsection{Incompressible limit}

Now we are able to prove that the energy estimates for compressible resistive MHD equations are actually uniform in sound speed. In physics the sound speed is defined by 
\[
c(t,x):=\sqrt{p'(\rho)}.
\]


We assume $\{\rho_{\kk}(p)\}$ is parametrized by $\kk\in\R^+$ such that $p_\kk'(\rho)|_{\rho=1}=\kk$ Therefore one has
\begin{equation}
\rho_{\kk}\to 1~~as~~\kk\to\infty,
\end{equation} and for some fixed constant $c_0$ and $\forall m\leq 6$
\begin{equation}
|\rho_{\kk}^{(m)}(p)|\leq c_0\text{ and }|\rho_{\kk}^{(m)}(p)|\leq c_0|\rho_{\kk}'(p)|.
\end{equation}

\bigskip

From now on, we set the magnetic diffusion constant $\lambda=1$ because our previous estimates in Proposition \ref{total} deny the possibility of getting vanishing resistivity limit. The previous computation still implies the energy estimates in Proposition \ref{total} are uniform in $\kk$.

\begin{prop}\label{justifykk}
For $t\in[0,T],~r\leq 4$, the following estimates hold for all $\kk$:
\begin{equation}
E_{r,\kk}(t)-E_{r,\kk}(0)\lesssim_{K,1/\epsilon_0,M,c_0,\vol\DD_t,E_{r-1,\kk}^*} \int_0^t \PP(E_{r,\kk}^*(s)) ds,
\end{equation} for some polynomial $\PP$ with positive coefficients(the upper bound is uniform in $\kk$), provided the following a priori assumptions together with the imposed conditions on $\rho_{\kk}(p)$ hold:
\begin{equation}\label{justify0kk}
\begin{aligned}
|\theta_{\kk}|+\frac{1}{\iota_0}&\leq K~~on~~\p \Omega,\\
-\nabla_N P_{\kk}\geq \epsilon_0&>0~~~on~~\p\Omega,\\
1\leq|\rho_{\kk}|&\leq M~~in~~\Omega,\\
\sum_{s+k\leq 2}|\p^s D_t^k p_{\kk}|+|\p^s D_t^k B_{\kk}|+|\p^s D_t^k u_{\kk}|&\leq M~~in~~\Omega.\\
\end{aligned}
\end{equation}
\end{prop}

\begin{flushright}
$\square$
\end{flushright}

Mimicing the proof of Proposition \ref{total}, one can get the following estimates uniform in $\kk$ from Proposition \ref{justifykk}:
\begin{thm}\label{totalkk}
There exists a positive continuous function $\TT$, such that: If $0<T\leq \TT(c_0,K,1/\epsilon_0,E_{4,\kk}^*(0),\volo)$, then any solution of \eqref{CMHD} in $t\in [0,T]$ satisfies the following bounds for some polynomial $\PP$ with positive coefficients:
\begin{equation}\label{completekk}
E_{4,\kk}^*(t)\lesssim 2 E_{4, \kk}^*(0),
\end{equation}provided the Rayleigh-Taylor physical sign condition
\[
-\nabla_N P_{\kk}\geq \epsilon_0>0~~on~~\p\Omega
\]holds.
\end{thm}

\begin{flushright}
$\square$
\end{flushright}

Given a sequence of initial data $(u_{0,\kk}, B_0, p_{0,\kk})$, if $E_{4,\kk}^*(0)$ are uniformly bounded in $\kk$, then a straightforward result of Theorem \ref{totalkk} is that the corresponding solution $(u_{\kk},B_{\kk},p_{\kk})$ converges in $C^2([0,T];\Omega)$.

\begin{thm}\label{limit}
Let $v_0,B_0$ be two divergence free vector fields with $B_0|_{\p\DD_0}=0$ such that its corrsponding pressure $q_0$ defined by
\[
\Delta \left(q_0+\frac{1}{2}|B_0|^2\right)=-(\p_i v_0^k \p_kv_0^i)+(\p_i B_0^k)(\p_k B_0^i),~~p_0|_{\p\DD_0}=0,
\]satisfies the Rayleigh-Taylor physical sign condition 
\[
-\nabla_N \left(q_0+\frac{1}{2}|B_0|^2\right)\bigg|_{\p\DD_0}\geq \epsilon_0>0.
\] Let $(v,B,q)$ be the solution to the incompressible resistive MHD equations with data $(v_0,B_0)$, i.e.,
\begin{equation}\label{iMHD}
\begin{cases}
D_t v=B\cdot\p B-\p (q+\frac{1}{2}|B|^2)~~~& \text{in}~\DD; \\
\dive v=0~~~&\text{in}~\DD; \\
D_t B-\Delta B=B\cdot\p v,~~~&\text{in}~\DD; \\
\dive B=0~~~&\text{in}~\DD,\\
q,B|_{\p\DD_0}=0~~~&~~~\\
(v,B)|_{t=0}=(v_0,B_0).~~~&~~~
\end{cases}
\end{equation}

Furthermore, let $(u_{\kk},B_{\kk},\rho_{\kk})$ be the solution to the compressible resistive MHD equations \eqref{CMHD} with density function $\rho_{\kk}(p)$ with initial data $(u_{0,\kk},B_0,\rho_{0,\kk})$ satisfying the compatibility condition up to 5-th order as well as the physical sign condition in \eqref{justify0kk}.

If we have $\rho_{\kk}\to\rho_0=1$ and $u_{0,\kk}\to v_0$ such that $E_{4,\kk}^*(0)$ is uniformly bounded in $\kk$, then one has 
\[
(u_{\kk},B_{\kk},\rho_{\kk})\to(v, B, 1)~~~~~\text{  in }C^2([0,T];\Omega).
\]
\end{thm}
\begin{proof}
By Sobolev embedding, the $C^2$ norm of $u_{\kk},B_{\kk},p_{\kk}$ can be bounded by $E_{4,\kk}^*(t)$:
\[
\begin{aligned}
&~~~~\|u_{\kk}\|_{C^2([0,T];\Omega)}+\|B_{\kk}\|_{C^2([0,T];\Omega)}+\|\rho_{\kk}\|_{C^2([0,T];\Omega)}\\
&\lesssim_K \sum_{s+k\leq 2}\sum_{j=0}^2\|\nabla^{s+j}D_t^k u_{\kk}\|_{\lli}^2+\|\nabla^{s+j}D_t^k B_{\kk}\|_{\lli}^2+\|\nabla^{s+j}D_t^k \rho_{\kk}\|_{\lli}^2 \\
&\lesssim E_{4,\kk}^*(t)\leq 2E_{4,\kk}^*(0).
\end{aligned}
\] Hence this together with energy estimates in Theorem \ref{totalkk} yields the uniform boundedness of the $C^2$ norm of $u_{\kk},B_{\kk},\rho_{\kk}$. Besides, using Morrey's embedding theorem, the uniform boundedness of $$\sum_{s+k\leq 2}\sum_{j=0}^2\|\nabla^{s+j}D_t^k u_{\kk}\|_{\lli}^2+\|\nabla^{s+j}D_t^k B_{\kk}\|_{\lli}^2+\|\nabla^{s+j}D_t^k \rho_{\kk}\|_{\lli}^2$$ implies that $$\nabla^s D_t^k u_{\kk}, \nabla^s D_t^k B_{\kk}, \nabla^s D_t^k \rho_{\kk}\in C^{0,\frac{1}{2}}(\Omega).$$ This H\"older continuity implies the equi-continuity of $u_{\kk},B_{\kk},\rho_{\kk}$ in $C^2([0,T];\Omega)$.  Therefore, Arzel\`a-Ascoli theorem gives a convergent subsequence (we still call it $\{(u_{\kk},B_{\kk},\rho_{\kk})\}_{\kk}$).

Finally, as $\kk\to\infty$, we have $(u_{\kk},B_{\kk},\rho_{\kk})\to (v,B,1)$ because now the wave equation (for compressible MHD) converges to the elliptic equation (for incompressible MHD) and the term $B_{\kk}\dive u_{\kk}$ will vanish when $\kk=\infty$, i.e., the equation of $B_{\kk}$ for compressible MHD converges to that of $B$ for incompressible MHD. This is actually a direct consequence of the uniform boundedness of $\|\rho_\kk\|_{C^2([0,T];\Omega)}$.
\end{proof}

\section{Construction of the initial data satisfying the compatibility conditions}\label{section9}

Now we are going to the last step of passing to the incompressible limit: Given an initial data $(v_0,B_0)$ for the incompressible resistive MHD system, we construct a sequence of initial datum of compressible resistive MHD system $\{(u_{0,\kk},B_{0,\kk},\rho_{0,\kk})\}_{\kk\in\R^+}$ , depending on the sound speed $\kk$, that satisfies the compatibility conditions of wave and heat equations and converges to $(v_0,B_0,1)$ as $\kk\to\infty$. Once we can do this, then by Theorem \ref{limit}, the incompressible limit exists for this sequence. From now on, we assume for simplicity\footnote{The proof in general case can be similarly proceeded. See Luo \cite{luo2018ww}.} that $$p_{\kk}(\rho_\kk)=\kk(\rho_{\kk}-1),~~i.e.,~\rho_{\kk}=1+\frac{p_{\kk}}{\kk}.$$

\subsection{Construction of the initial data}

\subsubsection*{Review of compatibility conditions}

Consider the compressible resistive MHD equations in Lagrangian coordinates
\[
\begin{cases}
\left( 1+\frac{p}{\kappa}\right) D_t u=B\cdot\nabla  B-\nabla (p+\frac{1}{2}|B|^2)~~~& \text{in}~\Omega; \\
\frac{1}{p+\kk}D_t p+\dive u=0~~~&\text{in}~\Omega; \\
D_t B-\lambda\Delta B=B\cdot\nabla  u+B\frac{1}{p+\kk}D_t p,~~~&\text{in}~\Omega; \\
\dive B=0~~~&\text{in}~\Omega,
\end{cases}
\] with boundary conditions
\[
p|_{\p\Omega}=0,~~B|_{\p\Omega}=\mathbf{0}.
\]and initial data
\[
u|_{t=0}=u_0,~p|_{t=0}=p_0,~B|_{t=0}=B_0, \text{ depending on }\kk.
\] 

In order for the initial data to be compatiable with the boundary condition, we need $$p_0|_{\p\Omega}=0,~B_0|_{\p\Omega}=\mathbf{0}.$$ Also we need $\dive u_0|_{\p\Omega}=0$ to guarentee the compatibility condition $D_t p|_{\p\Omega}=0$ when $t=0$.

$B$ satisfies the following heat equation
\begin{equation}\label{ih0}
D_t B-\Delta B\sim B\cdot\nabla u+\kk^{-1} D_t p B,
\end{equation}
while $p$ satisfies the following wave equation after taking divergence of the first equation of the compressible MHD system
\begin{equation}\label{iw0}
\begin{aligned}
\kk^{-1} D_t^2 p-\Delta p\sim &B\cdot\Delta B+(\nabla u)\symdot(\nabla u)+(\nabla B)(\nabla B) \\
&+\kk^{-1}((D_t p)^2-|\nabla p|^2+(\nabla p)(\nabla B)B) \\
\sim& -B\cdot D_t B+(B\cdot\nabla u)\cdot B+(\nabla u)\symdot(\nabla u)+(\nabla B)(\nabla B) \\
&+\kk^{-1}(|B|^2 D_t p+(D_t p)^2-|\nabla p|^2+(\nabla p)(\nabla B)B).
\end{aligned}
\end{equation}

The compatibility condition for wave/heat equation requires that $D_t B|_{\p\Omega}=\mathbf{0}$ and $D_t^2 p|_{\p\Omega}=0$ at time $t=0$. Therefore we must have
\begin{equation}
\begin{aligned}
\Delta_0 p_0+(\p u_0)\symdot(\p u_0)+(\p B_0)(\p B_0)&=0~~on~\p \Omega, \\
\Delta B_0&=\mathbf{0}~~on~\p \Omega,\\
\end{aligned}
\end{equation}where $\Delta_0$ is the Laplacian with respect to the smooth metric at $t=0$ on $\p\Omega$, and $\p_i=\p y^a/\p x_i\cdot \p/\p y^a$ is a smooth differential operator at $t=0$.

Similarly, if we take more time derivative to get higher order wave/heat equations
\[
\begin{aligned}
\kk^{-1} D_t^{k} B&=\Delta D_t^{k-1} B+T_{k}\\
\kk^{-1} D_t^{k+1}p&=\Delta D_t^{k-1} p+S_{k}
\end{aligned} 
\]for some function $T_k, S_k$ , then we need to guarentee that $D_t^k p|_{\p\Omega}=0, D_t^{k} B|_{\p\Omega}=\mathbf{0}$ at $t=0$ by requiring 
\[
\begin{aligned}
\Delta B_{k-1}+T_{k}|_{t=0}&=\mathbf{0}~~on~\p \Omega.\\
\Delta p_{k-1}+S_{k}|_{t=0}&=0~~on~\p \Omega.
\end{aligned} 
\]Here $p_k:=D_t^k p|_{t=0}$ and $B_k:=D_t^k B|_{t=0}$.

\subsubsection*{Constructing the initial data}

Now we construct the initial data $p_k, B_k$ which satisfies the compatibility conditions up to order $N$.

Suppose $v_0$ and $B_0$ are given divergence-free vector field. We still choose $B_0$ as the initial data for compressible equations. Now we define 
\begin{equation}\label{constructu0}
u_0=v_0+\p\phi.
\end{equation} Then the continuity equation requires that 
\begin{equation}\label{constructphi}
\Delta_0 \phi\sim -\kk^{-1} p_1,
\end{equation}and we will choose boundary condition such as 
\begin{equation}\label{phibdry}
\nabla_N\phi|_{\p\Omega}=0.
\end{equation}

Moreover, taking $D_t$ on \eqref{ih0} and \eqref{iw0} repeatedly, we should require that
\begin{equation}\label{constructbp}
\begin{aligned}
\Delta_0 B_k&\sim B_{k+1}-B_k\cdot \p u_0-\kk^{-1}(p_{k+1} B_0+p_1B_k)+G_k(u_0,B_0,p_0,B_1,p_1,\cdots,B_{k-1},p_{k-1})\\
\Delta_0 p_k&\sim ~~~\kk^{-1} p_{k+2}+B_0\cdot B_{k+1} +B_1\cdot B_k +(\p B_0)(\p B_k)+F_k(u_0,B_0,p_0,B_1,p_1,\cdots,B_{k-1},p_{k-1}) \\
&~~~~-\kk^{-1}(|B_0|^2p_{k+1}+p_1\cdot  p_{k+1}+B_0\cdot B_k\cdot p_1+p_2\cdot p_k-(\p p_k)(\p p_0))\\
&~~~~-\kk^{-1}(B_k (\p B_0)(\p p_0)+B_0 (\p B_k)(\p p_0)+B_0 (\p B_0)(\p p_k)).\\
\text{and }&~~B_k|_{\p\Omega}=\mathbf{0},~~p_k|_{\p\Omega}=0,~~k=0,1,\cdots, N.
\end{aligned}
\end{equation}Here $F_k,G_k$ are functions of $u_0,B_0,p_0,B_1,p_1,\cdots,B_{k-1},p_{k-1}$ and their spatial derivatives. If we prescrib $B_{N+1},B_{N+2},p_{N+1},p_{N+2}$ ro be any functions vanishing on $\p\Omega$, e.g.,
\begin{equation}\label{compbdry}
B_{N+1},B_{N+2}=\mathbf{0},~~p_{N+1},p_{N+2}=0,
\end{equation}
then \eqref{constructu0}, \eqref{constructphi}, \eqref{phibdry}, \eqref{constructbp} together with \eqref{compbdry} give a system of $$(u_0, p_0, B_1,p_1,\cdots, B_N, p_N, B_{N+1}, p_{N+1}, B_{N+2}, p_{N+2})$$ such that the data of compressible equation $u_{0,\kk}\to v_0 $ as $\kk\to\infty$. Since the system \eqref{constructu0}-\eqref{compbdry} is an elliptic system and $\rho$ is totally determined by $p$, so we only need to give a priori bound uniform in $\kk$ as $\kk\to\infty$ which will directly imply the existence of such data and thus complete our proof.

\subsection{A priori bounds and the existence of the initial data}

Our energy estimates in Theorem \ref{totalkk} requires the compatibility conditions up to order 5, i.e.,
\begin{equation}\label{cc}
p_k|_{\p\Omega}=0,~~B_k|_{\p\Omega}=\mathbf{0},~~\forall 0\leq k\leq 5.
\end{equation} This can be achieved by solving the following elliptic system for $0\leq k\leq 3$.
\begin{equation}\label{ellipticini} 
\begin{cases}
u_0=v_0+\p\phi~~&\text{ in }\Omega \\
\Delta\phi=-\kk^{-1}p_1~~&\text{ in }\Omega\text{ and }\nabla_N\phi|_{\p\Omega}=0 \\
\Delta B_k= B_{k+1}-B_k\cdot \p u_0-\kk^{-1}(p_{k+1} B_0+p_1B_k)+G_k~~&\text{ in }\Omega\text{ and }B_k|_{\p\Omega}=0 \\
\Delta p_k=\kk^{-1} p_{k+2}+B_0\cdot B_{k+1} +B_1\cdot B_k +(\p B_0)(\p B_k)+F_k \\
~~~~~~~~~~~~-\kk^{-1}(|B_0|^2p_{k+1}+p_1\cdot  p_{k+1}+B_0\cdot B_k\cdot p_1+p_2\cdot p_k-(\p p_k)(\p p_0))\\
~~~~~~~~~~~~-\kk^{-1}(B_k (\p B_0)(\p p_0)+B_0 (\p B_k)(\p p_0)+B_0 (\p B_0)(\p p_k))~~&\text{ in }\Omega\text{ and }p_k|_{\p\Omega}=0\\
p_4=p_5=0,~B_4=B_5=\mathbf{0}~~&\text{ in }\Omega. \\
\end{cases}
\end{equation}

Here
\begin{equation}
F_0=(\p u_0)\symdot(\p u_0),G_0=\mathbf{0}. 
\end{equation}

\begin{equation}\label{f1g1}
\begin{aligned}
F_1&=c_1(\p u_0)^3+c_{\alpha,\beta}(\p^{\alpha}u_0)(\p^{\beta}  p_0)+c_{\alpha,\mu}(\p u_0)(\p B_0)(\p B_0)+\kk^{-1}c_{\alpha,\beta,\mu}B_0(\p u_0)(\p B_0)(\p p_0),\\
G_1&=c_{\alpha,\mu}(\p^{\alpha}u_0)(\p^{\mu}  B_0)+c_1(\p u_0)(\p u_0)B_0. ~~~1\leq\alpha,\beta,\mu\leq 2,\alpha+\beta=\alpha+\mu=3.
\end{aligned}
\end{equation}

For $k=2,3$, one has
\begin{equation}\label{fk}
\begin{aligned}
F_k&=c^{\gamma_1\cdots \gamma_n}_{\alpha_1\cdots\alpha_m\beta_1\cdots\beta_n,k}(\p^{\alpha_1}u_0)\cdots(\p^{\alpha_m}u_0)(\p^{\beta_1}p_{\gamma_1})\cdots(\p^{\beta_n}p_{\gamma_n}) \\
&~~~~+c^{\nu_1\cdots \nu_l}_{\alpha_1\cdots\alpha_m\mu_1\cdots\mu_l,k}(\p^{\alpha_1}u_0)\cdots(\p^{\alpha_m}u_0)(\p^{\mu_1}B_{\nu_1})\cdots(\p^{\mu_l}B_{\nu_l}) \\
&~~~~+\kk^{-1}c^{\gamma'_1\cdots \gamma'_{n'}\nu'_1\cdots \nu'_{l'}}_{\alpha'_1\cdots\alpha'_{m'}\beta'_1\cdots\beta'_{n'}\mu'_1\cdots\mu'_{l'},k}(\p^{\alpha'_1}u_0)\cdots(\p^{\alpha'_{m'}}u_0)(\p^{\beta'_1}p_{\gamma'_1})\cdots(\p^{\beta'_{n'}}p_{\gamma_{n'}}) (\p^{\mu'_1}B_{\nu'_1})\cdots(\p^{\mu'_{l'}}B_{\nu'_{l'}}),
\end{aligned}
\end{equation}
where
\begin{equation}\label{fkk}
\begin{aligned}
\sum_{i=1}^{m}\alpha_i+\sum_{j=1}^{n}(\beta_j+\gamma_j)&=k+2 \\
\sum_{i=1}^{m}\alpha_i+\sum_{h=1}^{l}(\mu_h+\nu_h)&=k+2  \\
\sum_{i=1}^{m'}\alpha'_i+\sum_{j=1}^{n'}(\beta'_j+\gamma'_j)+\sum_{h=1}^{l'}(\mu'_h+\nu'_h)&=k+2.\\
1\leq \alpha_i\leq k, ~~1\leq\beta_j+\gamma_j\leq k+1,&~~\beta_j\geq 1,~~0\leq\gamma_j\leq k-1,~~1\leq m+n\leq k+2.\\
~\mu_h+\nu_h\leq k+1,&~~1\leq\mu_h\leq k,~~0\leq \nu_h\leq k-1,~~1\leq m+l\leq k+2.\\
1\leq\alpha'_i\leq k,~~1\leq \beta'_j+\gamma'_j+\mu'_h+\nu'_h\leq k+2,&~~1\leq \beta'_j\leq k,0\leq \gamma'_j,\nu'_h\leq k-1,~~0\leq \mu'_h\leq k,\\
1\leq m'+n'+l'\leq k+3.&~~~
\end{aligned}
\end{equation}
 and
\begin{equation}\label{gk}
\begin{aligned}
G_k&=c^{\gamma_1\cdots \gamma_n}_{\alpha_1\cdots\alpha_m\beta_1\cdots\beta_n,k}(\p^{\alpha_1}u_0)\cdots(\p^{\alpha_m}u_0)(\p^{\beta_1}B_{\gamma_1})\cdots(\p^{\beta_n}B_{\gamma_n}) \\
&~~~~+c^{\nu_1\cdots \nu_l}_{\alpha_1\cdots\alpha_{m'}\mu_1\cdots\mu_l,k}(\p^{\alpha_1}u_0)\cdots(\p^{\alpha_m}u_0)(\p^{\mu_1}B_{\nu_1})\cdots(\p^{\mu_l}B_{\nu_l}), 
\end{aligned}
\end{equation}
where
\begin{equation}\label{gkk}
\begin{aligned}
\sum_{i=1}^{m}\alpha_i+\sum_{j=1}^{n}(\beta_j+\gamma_j)&=k+2 \\
\sum_{i=1}^{m'}\alpha_i+\sum_{h=1}^{l}(\mu_h+\nu_h)&=k+1  \\
1\leq \alpha_i\leq k, ~~1\leq\beta_j+\gamma_j\leq k+1,&~~\beta_j\geq 1,~~0\leq\gamma_j\leq k-1,~~1\leq m+n\leq k+2.\\
~\mu_h+\nu_h\leq k,&~~1\leq\mu_h\leq k,~~0\leq \nu_h\leq k-1,~~1\leq m'+l\leq k+2.\\
\end{aligned}
\end{equation}

This is an elliptic system. To show the existence of a solution to \eqref{cc}, one only needs to give the a priori bound uniform in $\kk$ for this system which directly implies the existence. We impose $v_0\in  H^5$ and $B_0\in H^6$. For $0\leq k\leq 3$, we define
\[
m_k:=\|p_k\|_{H^{5-k}(\Omega)}+\|B_k\|_{H^{5-k}(\Omega)},~~~m_*:=\sum_{k} m_k+\|u_0\|_{H^5}.
\]We will repeatedly use elliptic estimates. 

\begin{itemize}

\item Estimates on  $u_0$

We have 
\begin{equation}\label{iniu0}
\|u_0\|_{H^5}\leq\|v_0\|_{H^5}+\|\p\phi\|_{H^5}\lesssim\|v_0\|_{H^4}+\kk^{-1}\|p_0\|_{H^4}
\end{equation}

\item Control of $F_k,G_k$

The precise form of $F_k$ and $G_k$ are the same as $F_k$ in Section 7.1 of Lindblad-Luo \cite{lindblad2018priori} up to some lower order terms. Therefore we only list the result and refer readers to that paper for details:
\begin{equation}
\begin{aligned}
\|F_2\|_{H^1}+\|G_2\|_{H^1}&\lesssim \mathcal{P}(\|u_0\|_{H^5},\|B_0\|_{H^5},\|p_1\|_{H^2},\|B_1\|_{H^2})\\
\|F_3\|_{\lli}+\|G_3\|_{\lli}&\lesssim \mathcal{P}(\|u_0\|_{H^5},\|B_0\|_{H^5},\|p_1\|_{H^3},\|B_1\|_{H^3},\|p_2\|_{H^2},\|B_2\|_{H^2}).
\end{aligned}
\end{equation}

\item Reduce all the diffuculty to $\|B_2\|_{\lli}$ and $\|B_3\|_{\lli}$.

Using elliptic estimates and Poincar\'e's inequality, one has

\begin{equation}\label{ini01}
\begin{aligned}
\|p_0\|_{H^5}&\lesssim\kk^{-1}(\|p_2\|_{H^3}+\|p_0\|_{H^4}^2+\|p_1\|_{H^3}^2+\|p_0\|_{H^4}\|B_0\|_{H^4}\|B_0\|_{H^3})+\PP(\|u_0\|_{H^4},\|B_0\|_{H^5}) \\
\|B_1\|_{H^4}&\lesssim \|B_2\|_{H^2}+\|B_1\|_{H^2}\|u_0\|_{H^3}+\|G_1\|_{H^2}+\kk^{-1}(\|B_0\|_{H^2}\|p_2\|_{H^2}+\|B_1\|_{H^2}\|p_1\|_{H^2}) \\
\|p_1\|_{H^4}&\lesssim \kk^{-1}(\|p_3\|_{H^2}+\|p_1\|_{H^3}\|p_0\|_{H^3}+\|p_2\|_{H^2}\|p_1\|_{H^2}+\|p_1\|_{H^3}\|B_0\|_{H^2}\|B_0\|_{H^3}\\
&+\|p_0\|_{H^3}\|B_1\|_{H^3}\|B_0\|_{H^2})+\|F_1\|_{H^2}+\|B_1\|_{H^2}\|\Delta B_0\|_{H^2}+\|B_0\|_{H^2}\|\Delta B_1\|_{H^2}+\|B_1\|_{H^3}\|B_0\|_{H^3}.
\end{aligned}
\end{equation}

As we can see, we reduce the estimates of $\|p_0\|_{H^5}+\|B_1\|_{H^4}+\|p_1\|_{H^4}$ to $\|B_2\|_{H^2}$, lower order terms of $p_1, B_1$, initial data and $\kk^{-1}m_*$. For those lower order terms, one can repeat the elliptic estimates above to reduce these terms to further lower order until these terms are only assigned by $L^2$-norm. 

The elliptic estimates for $B_k$ and $p_k$ when $k\geq 2$ are listed as follows:
\begin{equation}\label{ini23}
\begin{aligned}
\|B_2\|_{H^3}&\lesssim\|B_3\|_{H^1}+\|B_2\|_{H^1}\|u_0\|_{H^3}+\|G_2\|_{H^1} \\
&~~~~+\kk^{-1}(\|B_0\|_{H^2}\|p_3\|_{H^1}+\|B_1\|_{H^2}\|p_2\|_{H^1}+\|B_2\|_{H^1}\|p_1\|_{H^2}), \\
\|p_2\|_{H^3}&\lesssim\kk^{-1}(\|p_2\|_{H^1}\|p_0\|_{H^3}+\|p_3\|_{H^1}\|p_1\|_{H^2}+\|p_2\|_{H^2}\|B_0\|_{H^2}\|B_0\|_{H^3}+\|p_0\|_{H^3}\|B_0\|_{H^3}\|B_2\|_{H^2})\\
&~~~~+\|B_2\|_{H^1}\|\Delta B_0\|_{H^2}+\|B_0\|_{H^2}\|\Delta B_2\|_{H^1}+\|B_2\|_{H^2}\|B_0\|_{H^3}+\|F_2\|_{H^1}; \\
\|B_3\|_{H^2}&\lesssim\|B_3\|_{L^2}\|u_0\|_{H^3}+\|G_3\|_{L^2}+\kk^{-1}(\|B_0\|_{H^2}\|p_3\|_{L^2}+\|B_1\|_{H^2}\|p_3\|_{L^2}+\|B_3\|_{L^2}\|p_1\|_{L^2}) \\
\|p_3\|_{H^2}&\lesssim\kk^{-1}(\|p_3\|_{L^2}\|p_0\|_{H^3}+\|p_3\|_{H^1}\|B_0\|_{H^3}+\|p_3\|_{L^2}\|B_0\|_{H^3}^2+\|p_0\|_{H^3}\|B_0\|_{H^3}\|B_3\|_{L^2}) \\
&~~~~+\|F_3\|_{L^2}+\|B_3\|_{L^2}\|\Delta B_0\|_{H^2}+\|B_0\|_{H^2}\|u_0\|_{H^3}\|B_3\|_{L^2}
\end{aligned}
\end{equation}

Summing up \eqref{ini01} and \eqref{ini23}, we can find that $\|p_k\|_{H^{5-k}},\|B_k\|_{H^{5-k}}$ are bounded by lower order terms of themselves together with initial data and $\kk^{-1} m_*$. These lower order terms can be repeatedly reduced to further lower order until being assigned with $L^2$ norm. In other words, after repeatedly using elliptic estimates, one actually can get the estiamtes of the following form:
\begin{equation}\label{inireduce}
\sum_{k=1}^{5}m_k\lesssim \kk^{-1}m_*+\PP(\|u_0\|_{H^5},\|B_0\|_{H^5},\|p_0\|_{H^4})+\PP(\|B_0\|_{H^3},\|u_0\|_{H^3})(\|B_2\|_{L^2}+\|B_3\|_{L^2}).
\end{equation}

\item Reduction to $B_0$:

It remains to deal with $\|B_2\|_{L^2}+\|B_3\|_{L^2}$. We can use the heat equation of $B$ again to reduce it to $B_0$. The advantage is that $B_0$ is a prescribed data with given regularity instead of those $p_k, B_k$ whose control relies on the equations. In fact, we have
\[
B_3=\Delta B_2+B_2\cdot u_0+\kk^{-1}\text{ terms + lower order terms containing }B_1, B_0, p_0, u_0,
\]and
\[
B_2=\Delta B_1+B_1\cdot u_0+\kk^{-1}\text{ terms + lower order terms containing }B_0, u_0.
\]
Then $\|B_2\|_{L^2}+\|B_3\|_{L^2}$ can be bounded by $\|B_1\|_{H^4}$ together with initial data and $\kk^{-1}m_*$.
In other words, we can re-write the energy estimates to be
\begin{equation}\label{inireduce}
m_*\lesssim \kk^{-1}m_*+\PP(\|u_0\|_{H^5},\|B_0\|_{H^5})+\PP(\|B_0\|_{H^3},\|u_0\|_{H^3},\|B_1\|_{H^4}).
\end{equation}

Finally we have
\[
B_1=\Delta B_0+B_0\cdot\p u_0+\kk^{-1} B_0 p_1
\]which is derived by \eqref{ih0}, and thus
\[
\|B_1\|_{H^4}\lesssim \|B_0\|_{H^6}+\|B_0\|_{H^4}\|v_0\|_{H^5}+\kk^{-1}\PP(m_*).
\]

Therefore, we get the energy estimates uniform in $\kk$ as follows
\begin{equation}\label{inie}
m_*\lesssim \kk^{-1}\PP(m_*)+\PP(\|B_0\|_{H^6},\|v_0\|_{H^5}).
\end{equation} Let $\kk\to\infty$, and we finally get the uniform a priori bound for the elliptic system \eqref{cc}. Therefore we complete the construction of initial data satisfying the compatibility conditions of wave/heat equations. 
\end{itemize}

\subsection{Uniform enegry bounds, convergence of data and Rayleigh-Taylor physical sign condition}

Now we are able to show that $E_{4,\kk}(0)$ in Theorem \ref{totalkk} is uniformly bounded regardless of $\kk$. In fact
\[
\sum_{s+k\leq 4}\io\rho_0 Q(\p^s p_k,\p^s p_k)+ Q(\p^s B_k,\p^s B_k)\dx\lesssim\sum_{k=0}^4 \|p_k\|_{H^{4-k}}^2+\|B_k\|_{H^{4-k}}^2\lesssim m_*
\] and by the Sobolev trace lemma together with $P=p+\frac{1}{2}|B|^2$,
\[
\sum_{s+k\leq 4}\int_{\p\Omega}\rho_0 Q(\p^s P_k,\p^s P_k)\dx\lesssim\sum_{k=0}^4 \|p_k\|_{H^{5-k}}^2+\|B_k\|_{H^{5-k}}^2\lesssim m_*.
\]

Additionally, we can mimic the proof of Lemma \eqref{usk} to prove that
\[
\sum_{k+s\leq 4}\io \rho_0Q(\p^s D_t^k u|_{t=0},\p^s D_t^k u|_{t=0})\dx\lesssim m_*.
\]

Since $p_4=p_5=0$ and $B_4=B_5=\mathbf{0}$, we have $$\sum_{k\leq 5} W_k^2(0)+H_k^2(0)\lesssim m_*.$$

Summing up these bounds, we know $E_{4,\kk}^*(0)$ is bounded uniformly in $\kk$ as $\kk\to\infty$.

To achieve the incompressible limit, the very last thing is to verify the uniform convergence of the initial data we constructed above and the Rayleigh-Taylor sign condition, as $\kk\to\infty$. Actually,
\[
\|u_{0,\kk}-v_0\|_{H^5}\leq\|\p \phi_{\kk}\|_{H^5}\lesssim\kk^{-1}\|p_{1,\kk}\|_{H^4},
\]and thus by Sobolev embedding $H^5\hookrightarrow C^2$ in a bounded domain of $\R^3$, we actually prove $u_{0,\kk}\to v_0$ in $C^2$ because $\|p_{1,\kk}\|_{H^4}$ has uniform upper bound independent of $\kk$.

As for the Rayleigh-Taylor physical sign condition, we can assume it holds when $t=0$, i.e.,
\begin{equation}\label{inisign}
\nabla_N\left(p_0+\frac{1}{2}|B_0|^2\right)\leq-\epsilon_0<0~~on~\p\DD_0.
\end{equation} Due to Lemma \eqref{justify0}, it can be perturbed in a small time interval $[0,T]$.

Now, given any data for the incompressible resistive MHD equations $(v_0,B_0)$ such that the corresponding pressure term $q_0$ satisfies $$-\nabla_N \left(q_0+\frac{1}{2}|B_0|^2\right)\geq\epsilon_0>0,$$ our initial data $p_{0,\kk}$ will also satisfy \eqref{inisign} when $\kk^{-1}$ is sufficiently small. In fact, we have
\[
\Delta p_{0,\kk}\sim (\p u_{0,\kk})(\p u_{0,\kk})+(\p B_{0})(\p B_{0})+\kk^{-1}p_{2,\kk},
\]which implies
\[
\Delta (p_{0,\kk}-q_0)\sim (\p u_{0,\kk})(\p^2\phi_{\kk})+(\p^2\phi_{\kk})^2+\kk^{-1}p_{2,\kk}.
\]The standard elliptic estimate yields the convergence. Hence, the incompressible limit of compressible resistive MHD equations is achieved.

\begin{appendix}
\section{Appendix}
\noindent {\bf List of notations:}

\begin{itemize}
\item $D_{t}$: the material derivative $D_t=\p_t+u\cdot\p$
\item $\p_i$: partial derivative with respect to Eulerian coordinate $x_i$
\item $\DD_t\in\R^n$: the domain occupied by fluid particles at time $t$ in Eulerian coordinate
\item $\Omega\in\R^n$: the domain occupied by fluid particles in Lagrangian coordinate
\item $\p_a = \frac{\p}{\p y_a}$: partial derivative with respect to Lagrangian coordinate $y_a$
\item $\nab_a$: covariant derivative with respect to $y_a$
\item $\Pi S$: projected tensor $S$ on the boundary
\item $\cnab,\cp$: projected derivative on the boundary
\item $N$: the outward unit normal of the boundary
\item $\theta=\cnab N$: the second fundamental form of the boundary
\item $\sigma=\text{Tr }\theta$: the mean curvature
\end{itemize}
{\bf Mixed norms}
\begin{itemize}
\item $\|f\|_{s,k} = \|\nab^sD_t^k f\|_{\lli}$
\item $|f|_{s,k}= |\nab^sD_t^k f|_{\llb}$
\end{itemize}

\subsection{Extension of the normal to the interior and the geodesic normal coordinate}

The definition of our energy \eqref{Er} relies on extending the normal to the interior. This can be accomplished by foliating the domain close to the boundary into the surface that is not self-intersecting.  Also we want to control the evolution of the moving boundary, which can be estimated by the time derivative of the normal in Lagrangian coordinate. We conclude the above statements by the following two lemmata, whose proof can be found in \cite{christodoulou2000motion}.\\
\lem \label{trace 1} let $\iota_0$ be the injective radius (\ref{inj rad}), and let $d(y)=\text{dist}_g(y,\p\Omega)$ be the geodesic distance in the metric $g$ from $y$ to $\p\Omega$. Then the co-normal $n=\nab d$ to the set $S_a=\p\{y\in\Omega:d(y)=a\}$ satisfies, when $d(y)\leq \frac{\iota_0}{2}$ that
\begin{align}
|\nab n| \lesssim |\theta|_{L^{\infty}(\p\Omega)}, \label{extending nab n}\\
|D_t n| \lesssim |D_t g|_{L^{\infty}(\Omega)}.
\end{align}

\begin{flushright}
$\square$
\end{flushright}

\lem \label{trace 2} let $\iota_0$ be the injective radius (\ref{inj rad}),and $d_0$ be a fixed number such that $\frac{\iota_0}{16}\leq d_0\leq \frac{\iota_0}{2}$. Let $\eta$ be a smooth cut-off function satisfying $0\leq \eta(d)\leq1$, $\eta(d)=1$ when $d\leq\frac{d_0}{4}$ and $\eta(d)=0$ when $d>\frac{d_0}{2}$. Then the psudo-Riemannian metric $\gamma$ given by
$$
\gamma_{ab}=g_{ab}-\tilde{n}_a\tilde{n}_b,
$$
where $\tilde{n}_c=\eta(\frac{d}{d_0})\nab_cd$
satisfies
\begin{align}
|\nab\gamma|_{L^{\infty}(\Omega)}\lesssim(|\theta|_{L^{\infty}(\p\Omega)}+\frac{1}{\iota_0})\\
|D_t\gamma(t,y)|\lesssim |D_t g|_{L^{\infty}(\Omega)}. \label{Dtgamma}
\end{align}

\begin{flushright}
$\square$
\end{flushright}

\rmk The above two lemmata show that $|D_t n|$ and $|D_t\gamma(t,y)|$ involved in the $Q$-tensor can be controlled by the a priori assumptions \eqref{asspriori}, because the behaviour of $D_tg$ is almost like $\nab v$ by that of \eqref{Dtg}. Therefore, the time derivative on the coefficients of the $Q$-tensor only produces lower order terms. In addition, by the first equation of \eqref{asspriori}, $|\nab n|$ and $|\nab\gamma|$ are controlled by $K$, which is essential when proving the elliptic estimates.

\subsection{Sobolev inequalities: Embedding, interpolation and trace lemma}

The following results are standard in $\R^n$, but we need to illustrate how it depends on the geometry of the moving domain. The coefficients involved in our inequalities depend on $K$, whose reciprocal is the lower bound for the injective radius $\iota_0$. The proofs of these lemmata are omitted which can be found in the appendix of \cite{christodoulou2000motion} and \cite{lindblad2018priori}.

\subsubsection*{Sobolev embedding}

First we list some Sobolev lemmata in a domain with boundary.

\lem \label{interior sobolev} (Interior Sobolev inequalities)
Suppose $\frac{1}{\iota_0}\leq K$ and $\alpha$ is a $(0,r)$ tensor, then
\begin{align}
\|u\|_{L^{\frac{2n}{n-2s}}(\Omega)}\lesssim_{K} \sum_{l=0}^s\|\nab^l u\|_{\lli},\q 2s<n \label{interior sobolev 2s<n},\\
\|u\|_{L^{\infty}(\Omega)}\lesssim_{K} \sum_{l=0}^{s}\|\nab^l u\|_{\lli},\q 2s>n \label{interior sobolev 2s>n}.
\end{align}

\begin{flushright}
$\square$
\end{flushright}

Similarly on $\p\Omega$, we have
\lem \label{boundary soboolev} (Boundary Sobolev inequalities)
\begin{align}
\|u\|_{L^{\frac{2(n-1)}{n-1-2s}}(\Omega)}\lesssim_{K} \sum_{l=0}^s|\nab^l u|_{\llb},\q 2s<n-1,\\
\|u\|_{L^{\infty}(\Omega)}\lesssim_{K} \delta|\nab^s u|_{\llb}+\delta^{-1}\sum_{l=0}^{s-1}|\nab^l u|_{\llb},\q 2s>n-1,
\end{align}
for any $\delta>0$. In addition, for the boundary we can also interpret the norm be given by the inner product $\langle \alpha, \alpha\rangle=\gamma^{IJ}\alpha_I\alpha_J$, and the covariant derivative is then given by $\cnab$.

\begin{flushright}
$\square$
\end{flushright}

\subsubsection*{Poincar\'e's inequalities}

\lem \label{poincare}(Poincar\'{e} type inequalities)
Let $q:\Omega\subset\R^n\to\R$ be a smooth and $q|_{\p\Omega}=0$, then
\begin{align}
\|q\|_{\lli}&\lesssim (\vol\Omega)^{\frac{1}{n}}\|\nab q\|_{\lli},\label{FB}\\
\|\nab q\|_{\lli}&\lesssim (\vol\Omega)^{\frac{1}{n}}\|\lap q\|_{\lli}.\label{Poincare}
\end{align}
\begin{proof} The first inequality is called Faber-Krahns theorem which can be found in \cite{YS}. The second inequality follows from the first and integration by parts.
\end{proof}

\subsubsection*{Interpolation inequalities}

\thm \label{int interpolation} (Interior interpolation)
Let $u$ be a $(0,r)$ tensor, and suppose $\iota_0^{-1}\leq K$, we have
\begin{align}
\sum_{j=0}^{l}\|\nab^ju\|_{L^{\frac{2r}{k}}(\Omega)}\lesssim \|u\|_{L^{\frac{2(r-l)}{k-l}}(\Omega)}^{1-\frac{l}{r}}(\sum_{i=0}^{r}\|\nab^iu\|_{\lli}K^{r-i})^{\frac{l}{r}}.
\end{align}
In particular, if $k=l$,
\begin{align}
\sum_{j=0}^{k}\|\nab^ju\|_{L^{\frac{2r}{k}}(\Omega)}\lesssim \|u\|_{L^{\infty}(\Omega)}^{1-\frac{k}{r}}(\sum_{i=0}^{r}\|\nab^iu\|_{\lli}K^{r-i})^{\frac{k}{r}}. \label{31}
\end{align}
\begin{flushright}
$\square$
\end{flushright}

\subsubsection*{Interpolation on $\p\Omega$}
We need the following boundary interpolation inequalities to control the boundary part of our energy \eqref{Er}.
\thm \label{bdy interpolation} (Boundary interpolation)
Let $u$ be a $(0,r)$ tensor, then
\begin{equation}
|\cnab^l u|_{L^{\frac{2r}{k}}(\p\Omega)}\lesssim |u|_{L^{\frac{2(r-l)}{k-l}}(\p\Omega)}^{1-\frac{l}{r}}|\cnab^ru|_{\llb}^{\frac{l}{r}}.
\end{equation}
In particular, if $k=l$,
\begin{align}
|\cnab^k u|_{L^{\frac{2r}{k}}(\p\Omega)}\lesssim |u|_{L^{\infty}(\p\Omega)}^{1-\frac{k}{r}}|\cnab^ru|_{\llb}^{\frac{k}{r}}.
\end{align}

\begin{flushright}
$\square$
\end{flushright}

\thm \label{gag-ni thm}(Gagliardo-Nirenberg interpolation inequality)
Let $u$ be a $(0,r)$ tensor, and suppose $\p\Omega\in\R^2$ and $\frac{1}{\iota_0}\leq K$, we have
\begin{align}
|u|_{L^4{(\p\Omega})}^2 \lesssim_K |u|_{\llb}|u|_{H^1(\p\Omega)},\label{gag-ni}
\end{align}
where the boundary Sobolev norm $\|u\|_{H^1(\p\Omega)}$ is defined via tangential derivative $\cnab$.
\begin{proof}
See Theorem A.8 in Lindblad-Luo \cite{lindblad2018priori} for details. Its proof requires the result of Constantin-Seregin \cite{Cons}.
\end{proof}
\rmk One can also prove a generalized \eqref{gag-ni} of the form
\begin{equation}
|u|_{L^p{(\p\Omega})}^2 \lesssim |u|_{L^{p/2}(\p\Omega)}|u|_{H^1(\p\Omega)},\qquad p\geq 4.
\end{equation}

The next theorem is to delta with the interpolation of tangential projections on the boundary. First, we define that the projected $(0,r), r<t$ derivative $\Pi^{r,0}\nab^r \alpha$ has components
$$(\Pi\nab^r)_{i_1,\cdots,i_r} \alpha_{i_r+1,\cdots,i_t}=\gamma_{i_1}^{j_1}\cdots\gamma_{i_r}^{j_r}\nab_{j_1}\cdots\nab_{j_r} \alpha_{i_r+1,\cdots,i_t},$$ for any $(0,t)$ tensor $\alpha$. The detailed proof can be found in \cite{christodoulou2000motion}.

\thm \label{tensor interpolation} (Tensor interpolation)
Let $\alpha$ be a $(0,t)$ tensor, $r'= r-2$. Suppose $|\theta|+|\frac{1}{\iota_0}|\leq K$, then we have for $t+s<r$ 
\begin{equation}
\begin{aligned}
|(\Pi^{s,0}\nab^s)\alpha|_{L^{\frac{2r'}{s}}(\p\Omega)}\lesssim_K |\alpha|_{\linf}^{1-s/r'}&\Big(|\nab^{r'}\alpha|_{\llb}+(1+|\theta|_{\linf})^{r'}\\
&\cdot(|\theta|_{\linf}+|\cnab^{r'}\theta|_{\llb})\sum_{l=0}^{r'-1}|\nab^l\alpha|_{\llb}\Big)\\
&+(1+|\theta|_{\linf})^s(|\theta|_{\linf}+|\cnab^{r'}\theta|_{\llb})^{s/r'}\sum_{l=0}^{r'-1}|\nab^l\alpha|_{\llb}.
\end{aligned}
\end{equation}
In particular,
\begin{equation}
\begin{aligned}
&\Big{|}|(\Pi^{s,0}\nab^s)\alpha|\cdot|(\Pi^{r'-s,0}\nab^{r'-s}\beta)|\Big{|}_{\llb}\\
&\lesssim_K(|\alpha|_{\linf}+\sum_{l=0}^{r'-1}|\nab^l\alpha\|_{\llb})|\nab^{r'}\beta|_{\llb}\\
&~~~~~~~~+(|\beta|_{\linf}+\sum_{l=0}^{r'-1}|\nab^l\beta|_{\llb})|\nab^{r'}\alpha|_{\llb}+(1+|\theta|_{\linf})^{r'}(|\theta|_{\linf}+|\cnab^{r'}\theta|_{\llb})\\
&~~~~~~~~~+(|\alpha|_{\linf}+\sum_{l=0}^{r'-1}|\nab^l\alpha|_{\llb})(|\beta|_{\linf}+\sum_{l=0}^{r'-1}|\nab^l\beta|_{\llb}) \label{tensor int}.
\end{aligned}
\end{equation}
\begin{proof}
See \cite{christodoulou2000motion}, section 4.
\end{proof}

\subsubsection*{Sobolev trace theorem}
\thm \label{trace theorem} (Trace theorem)
Let $u$ be a $(0,r)$ tensor, and assume that $|\theta|_{\linf}+\frac{1}{\iota_0}\!\leq \!K$. Then
\begin{align}
|u|_{\llb}\lesssim_{K,r,n} \sum_{j\leq 1}|\nab^j u|_{\lli} \label{trace}
\end{align}
\begin{proof}
Let $N'$ be the extension of the normal to the interior, then the Green's identity yields
$$
\int_{\p\Omega}|u|^2\,d\mu_{\gamma} = \int_{\Omega}\nab_k(N'^k|u|^2)\,d\mu.
$$
Hence, by Lemma \ref{trace 1} and \ref{trace 2}, (\ref{trace}) follows.
\end{proof}
\end{appendix}

\bigskip

\section*{Acknowledgement} The author would like to thank his advisor Professor Hans Lindblad for his guidance in PDE. The author would also like to thank Professor Chenyun Luo and Professor Chengchun Hao for helpful discussions.

\end{document}